\documentclass[11pt]{article}
\usepackage[latin1]{inputenc}
\usepackage{epsfig}
\usepackage{color}
\usepackage[british,english]{babel}
\usepackage{amsthm}
\usepackage{amsmath}
\usepackage{amsfonts}
\usepackage{amssymb}
\usepackage{graphicx}

\usepackage{tikz}
\usepackage{pgfplots}
\usepgfplotslibrary{colormaps}

\usepackage{color}

\newcommand{\R}{\mathbb{R}}
\newcommand{\eps}{\epsilon}

\definecolor{darkgreen}{rgb}{0,0.6,0}

\definecolor{darkblue}{rgb}{0,0,0.7} 
\definecolor{darkred}{rgb}{0.9,0.1,0.1}
\definecolor{darkgreen}{rgb}{0,0.6,0}

\usepackage{fancyhdr}

\newtheorem{corollary}{Corollary}[section]
\newtheorem{definition}[corollary]{Definition}

\newtheorem{lemma}[corollary]{Lemma}

\newtheorem{remark}[corollary]{Remark}
\newtheorem{theorem}[corollary]{Theorem}

\numberwithin{equation}{section}

\date{}
\begin{document}
\title{Effective models for generalized Newtonian fluids through a thin porous medium following the Carreau law}\maketitle

\vskip-30pt
 \centerline{Mar\'ia ANGUIANO\footnote{Departamento de An\'alisis Matem\'atico. Facultad de Matem\'aticas.
Universidad de Sevilla, 41012 Sevilla (Spain)
anguiano@us.es}, Matthieu BONNIVARD\footnote{Ecole Centrale de Lyon, CNRS, INSA Lyon, Universite Claude Bernard Lyon 1, Université Jean Monnet, ICJ UMR5208,
69130 Ecully, France,
matthieu.bonnivard@ec-lyon.fr}, and  Francisco J. SU\'AREZ-GRAU\footnote{Departamento de Ecuaciones Diferenciales y An\'alisis Num\'erico. Facultad de Matem\'aticas. Universidad de Sevilla, 41012 Sevilla (Spain) fjsgrau@us.es}}

 \renewcommand{\abstractname} {\bf Abstract}
\begin{abstract} 
We consider the flow of a generalized Newtonian fluid through a thin porous medium of thickness $\epsilon$, perforated by periodically distributed solid cylinders of size $\epsilon$. We assume that the fluid is described by the 3D incompressible Stokes system, with a non-linear viscosity following the Carreau law of flow index $1<r<+\infty$, and scaled by a factor $\epsilon^{\gamma}$, where $\gamma\in \mathbb{R}$. Generalizing (Anguiano {\it et al}.\! , Q. J. Mech. Math., 75(1), 2022, 1-27), where the particular case $r<2$ and $\gamma=1$ was addressed, we perform a new and complete study on the asymptotic behaviour of the fluid as $\epsilon$ goes to zero. Depending on $\gamma$ and the flow index $r$, using homogenization techniques, we derive and rigorously justify different effective linear and non-linear lower-dimensional Darcy's laws. Finally, using a finite element method, we study numerically the influence of the rheological parameters of the fluid and of the shape of the solid obstacles on the behaviour of the effective systems.

\end{abstract}

 {\small \bf AMS classification numbers: } 76-10, 76A05, 76M50, 76A20, 76S05, 35B27, 35Q35.

  \section {Introduction}\label{S1} 
  
An incompressible generalized Newtonian fluid is a type of non-Newtonian fluid which is characterized by a viscosity depending on the principal invariants of the  strain-rate tensor $\mathbb{D}[u]$. If $u$ is the velocity, $p$ the pressure and $Du$ the gradient velocity tensor, $\mathbb{D}[u]=(Du+D^tu)/2$  denotes the  strain-rate tensor  and $\sigma$ the stress tensor given by $\sigma=-pI+2\eta_{r}\mathbb{D}[u]$. The viscosity $\eta_r$ is constant for a Newtonian fluid but dependent of the shear rate, \emph{i.e.}\! $\eta_r=\eta_{r}(\mathbb D[u])$, for generalized Newtonian fluids. The deviatoric stress tensor $\tau$, \emph{i.e.}\! the part of the total stress tensor that is zero at equilibrium, is then a nonlinear function of the shear rate $\mathbb D[u]$: $
\tau=\eta_r(\mathbb D[u]) \mathbb D[u]$  (see Barnes {\it et al.} \cite{Barnes}, Bird {\it et al.} \cite{Bird} and Mikeli\'c \cite{Mikelic1} for more details). 

The \emph{power law} or Ostwald-de Waele
model (Ostwald, 1925; de Waele, 1923) is commonly used to describe the motion of a generalized Newtonian fluid. The corresponding viscosity formula is
\begin{equation}\label{PowerLaw}\eta_r(\mathbb{D}[u])=\mu |\mathbb{D}[u]|^{r-2},\quad 1<r<+\infty,\quad \mu>0,
\end{equation} 
where $\mu>0$ is the consistency of the fluid and $r$ is the flow index.  The matrix norm $|\cdot|$ is defined by $|\xi|^2=Tr(\xi \xi^t)$ with $\xi \in \mathbb{R}^{3\times 3}$.  We recall that the generalized Newtonian fluids are classified in two main categories (see Saramito \cite[Chapter 2]{Saramito} for more details): 
\begin{itemize}
\item[--]  {\it pseudoplastic} or {\it shear thinning} fluids, where the viscosity decreases with the shear rate, which corresponds to the case of a flow index  $1<r<2$;
\item[--]  {\it dilatant} or {\it shear thickening} fluids, where the viscosity increases with the shear rate, and $r>2$.
\end{itemize}
We also recall that the case $r=2$ corresponds to a Newtonian fluid.\\

 The power law was introduced to model polymer solutions at high shear rates, and by its simplicity, permits analytical calculations in simple geometries. However, it has the
disadvantage of not describing a Newtonian plateau and even predicts an infinite viscosity
as the shear rate goes to zero and $1<r<2$ (see Agassant \emph{et al.}\!~\cite[p.~49]{Agassant}), whereas for
real fluids it tends to some constant value $\eta_0$ called the zero-shear-rate viscosity. For these reasons, other viscosity models are used, which better describe the real behaviour of pseudoplastic or dilatant fluids, but are more difficult to analyze mathematically.

\paragraph{The Carreau law.} Among those, an important model is the well-known {\it Carreau law}, which will be considered in this paper and is defined by
\begin{equation}\label{Carreau}
\eta_r(\mathbb{D}[u])=(\eta_0-\eta_\infty)(1+\lambda|\mathbb{D}[u]|^2)^{{r\over 2}-1}+\eta_\infty,\quad 1<r<+\infty,\quad \eta_0> \eta_\infty>0,\quad \lambda>0\, .
\end{equation} 
In this relation, $r$ is the flow index of the fluid, $\eta_0$ is the low-shear-rate limit of the viscosity, and for $1<r<2$, $\eta_\infty$ is the high-shear-rate limit of the viscosity. Parameter $\lambda$ is a time constant and $r-2$ describes the slope in the power law region. 
For $r=2$, as in the power-law model, one recovers a Newtonian fluid model, with viscosity $\eta_0$.

In this paper, we consider the flow of a Carreau fluid in a thin porous medium. In this context, the complexity of the fluid domain geometry makes it very costly to solve the equations from continuum mechanics directly. The homogenisation theory provides an alternative approach, which aims to derive an average model by analyzing the effects of the micro-structure, \emph{i.e.}\! the pore structure, on the solutions of these partial differential equations. This mathematical method allows one to derive rigorously the equations describing the filtration of a generalized Newtonian fluid (see for instance Mikeli\'c \cite{Mikelic1}).

To derive the averaged law describing the generalized Newtonian fluid flow through a porous medium $\Omega_{\epsilon}\subset\mathbb{R}^3$,  which is a domain with fixed height and periodically perforated by obstacles of size $\epsilon$, Bourgeat and Mikeli\'c in \cite{Bourgeat1} (see also Bourgeat {\it et al.} \cite{Bourgeat2} and Kalousek \cite{Kalousek}) used  the two-scale convergence method. We also refer to Anguiano \cite{Anguiano_Med2020, Anguiano_Med2023} for nonlinear parabolic problems in porous medium. The domain without perforations is the bounded smooth domain $\Omega\subset\mathbb{R}^3$ that is divided into two parts, the fluid part $\Omega_\epsilon$ and the solid part $\Omega\setminus\Omega_\epsilon$. 
Moreover, assuming that the flow is  sufficiently slow to neglect inertial effects, the following stationary Stokes system with a non-linear viscosity following the {\it Carreau law} (\ref{Carreau}) was considered:
\begin{equation}\label{Bourgeat_Mikelic}
\left\{\begin{array}{rl}\displaystyle
\medskip
-\epsilon^\gamma {\rm div}\left(\eta_r(\mathbb{D}[u_\epsilon])\mathbb{D}[u_\epsilon]\right)+ \nabla p_\epsilon = f & \hbox{in }\Omega_\epsilon,\\
\medskip
\displaystyle
{\rm div}\,u_\epsilon=0& \hbox{in }\Omega_\epsilon,\\
\medskip
\displaystyle
u_\epsilon=0& \hbox{on }\partial \Omega_\epsilon.
\end{array}\right.
\end{equation}
In the momentum equation~(\ref{Bourgeat_Mikelic}), the viscosity $\eta_r$ is affected by a factor $\eps^\gamma$. This can be interpreted as a scaling of the low-shear-rate and high-shear-rate limit viscosities $\eta_0,\eta_\infty$ appearing in the Carreau law~\eqref{Carreau} by $\eps^\gamma$, which allows one to characterize mathematically the impact of these quantities on the asymptotic behaviour of the system. Indeed, letting $\epsilon$ tend to zero, different types of averaged momentum equations connecting the velocity and the pressure gradient were rigorously derived, depending on the value of $\gamma$ and the flow index $r$.

\begin{itemize}
\item If $\gamma < 1$ and $1<r<+\infty$, the homogenized law is the classical 3D Darcy's law for Newtonian fluids 
\begin{equation}\label{DarcyMikless1}V(x)={K\over \eta}\left(f(x)-\nabla_{x}p(x)\right)\ \hbox{in }\Omega,\quad {\rm div}_x V(x)=0\ \hbox{ in }\Omega,\quad V(x)\cdot n=0\ \hbox{ on }\partial\Omega,
\end{equation}
where $p$ is the limit pressure and the {\it permeability tensor} $K\in\mathbb{R}^{3\times 3}$ is obtained by solving 3D Stokes local problems posed in a reference cell which contains the information of the geometry of the obstacles. The viscosity $\eta$ is equal to $\eta_0$.

\item If $\gamma=1$ and  $r\neq 2$, the mean global filtration velocity as a function of the pressure gradient is given by
\begin{equation}\label{CarreauintroMik}
V(x)=\mathcal{U}\left(f(x)-\nabla_{x}p(x)\right)\ \hbox{in }\Omega,\quad {\rm div}_x V(x)=0\ \hbox{ in }\Omega,\quad V(x)\cdot n=0\ \hbox{ on }\partial\Omega,
\end{equation}
where  $\mathcal{U}:\mathbb{R}^3\to \mathbb{R}^3$ is a {\it permeability operator}, not necessary linear,  and is defined through
the solutions of  3D local Stokes problems with non-linear viscosity following the {\it Carreau law} and posed in a reference cell. Linearity holds only if $r=2$, where  the classical 3D Darcy's law for Newtonian fluids  (\ref{DarcyMikless1}) is derived with $\eta=\eta_0$.

\item If $\gamma>1$, the following cases hold
\begin{itemize}
\item For $1<r\leq 2$, the homogenized law is the classical 3D Darcy's law for Newtonian fluids (\ref{DarcyMikless1}) with $\eta=\eta_\infty$ if $1<r<2$ and $\eta=\eta_0$ if $r=2$.
\item For $r> 2$, the mean global filtration velocity as a function of the pressure gradient is given by (\ref{CarreauintroMik}), where the {\it permeability operator} $\mathcal{U}:\mathbb{R}^3\to \mathbb{R}^3$ is defined through
the solutions of  3D local non-Newtonian Stokes problems with non-linear viscosity following the {\it power law} (\ref{PowerLaw}) with $\mu=(\eta_0 -\eta_\infty)\lambda^{(r/2)-1}$ and posed in a reference cell.

\end{itemize}
\end{itemize}

On the other hand, in \cite{Tapiero2} Boughanim and Tapi\'ero considered the generalized Newtonian fluids   flow through a thin domain $\Omega_\epsilon=\omega\times (0,\epsilon) \subset \mathbb{R}^3$, with $\omega\subset\mathbb{R}^2$, where $\epsilon$ is the small thickness of the domain. From the Stokes system (\ref{Bourgeat_Mikelic}), with external body force of the form $f=(f',0)$ such that $f'=(f_1,f_2)$,
by using dimension reduction and homogenization techniques, they derived the limit when the thickness tends to zero. Depending on the value of $\gamma$ and the flow index $r$, they obtained the following:
\begin{itemize}
\item If $\gamma< 1$, the following cases hold
\begin{itemize}
\item For $1<r\leq 2$,  the homogenization law is the classical linear 2D Reynolds law for Newtonian fluids
\begin{equation}\label{Reynoldsintro1}\left\{\begin{array}{l}
\medskip
\displaystyle V'(x')={1\over 6\mu}\left(f'(x')-\nabla_{x'}p(x')\right),\quad V_3(x')=0\quad\hbox{in }\omega,\\
\medskip
{\rm div}_{x'}V(x')=0\quad\hbox{in }\omega,\quad V(x')\cdot n=0\quad\hbox{on }\partial\omega,
\end{array}\right.
\end{equation}
where $V'=(V_1,V_2)$, $x'=(x_1,x_2)$. The viscosity $\mu$ is equal to $\eta_0$.
\item For $r> 2$, the filtration velocity is zero, \emph{i.e.}\! $V\equiv 0$.
\end{itemize}

\item If $\gamma=1$ and  $r\neq 2$, the homogenization law corresponds to a non-linear 2D Reynolds law of Carreau type
$$\left\{\begin{array}{l}
\medskip
\displaystyle
V'(x')=2(f'(x')-\nabla_{x'}p(x'))\int_{-{1\over 2}}^{1\over 2}{({1\over 2}+\xi)\xi\over 
\psi(2|f'(x')-\nabla_{x'}p(x')||\xi|)}\,d\xi,\quad V_3(x')=0\quad\hbox{in }\omega,\\
\medskip
{\rm div}_{x'}V'(x')=0\quad\hbox{in }\omega,\quad V'(x')\cdot n=0\quad\hbox{on }\partial\omega,
\end{array}\right.$$
where the function $\psi=\psi(\tau)$, $\tau\in\mathbb{R}^+$, is the inverse of the equation  $\tau=\psi\sqrt{{2\over \lambda}\left({\psi-\eta_\infty\over \eta_0-\eta_\infty}\right)^{2\over r-2}-1}$. Linearity holds only if $r=2$, where the the classical linear 2D Reynolds law for Newtonian fluids (\ref{Reynoldsintro1}) with $\eta=\eta_0$ is derived.

\item If $\gamma>1$, the following cases hold
\begin{itemize}
\item For $1<r\leq 2$, the homogenization law is the classical linear 2D Reynolds law for Newtonian fluids (\ref{Reynoldsintro1}) with $\mu=\eta_\infty$ if $1<r<2$ and $\mu=\eta_0$ if $r=2$.
\item For $r>2$, the homogenization law corresponds to a non-linear 2D Reynolds law of power  type
$$\left\{\begin{array}{l}
\medskip
\displaystyle
V'(x')={\lambda^{r'/2-1}\over (\eta_0-\eta_\infty)^{r'-1}}{1\over 2^{r'/2}(r'+1)}|f'(x')-\nabla_{x'} p(x')|^{r'-2}(f'(x')-\nabla_{x'} p(x'))\quad\hbox{in }\omega,\\
\medskip V_3(x')=0\quad\hbox{in }\omega,\\
\medskip
{\rm div}_{x'}V'(x')=0\quad\hbox{in }\omega,\quad V'(x')\cdot n=0\quad\hbox{on }\partial\omega.
\end{array}\right.$$
\end{itemize}
\end{itemize}

We remark that in \cite{Tapiero2}  and \cite{Bourgeat1} the Navier-Stokes equation is considered, which implies an upper limit on $\gamma$ due to the contribution of the inertial term. However, as pointed out in the beginning of Section 1.4 in  \cite{Bourgeat1}  (see also \cite{Bourgeat2}), the corresponding results   remain unchanged for the non-Newtonian Stokes systems. Furthermore, there are no upper limits on $\gamma$ and the convergence of the pressure is  stronger.
  This motivates our choice of studying the flow of a generalized Newtonian fluid governed by the stationary Stokes system (\ref{Bourgeat_Mikelic}), with a non-linear viscosity following the {\it Carreau law} (\ref{Carreau}) through a thin periodic medium. We now detail the geometry of the medium that we consider.

\paragraph{  The thin porous medium.}

The interest in the behaviour of generalized Newtonian fluids through thin porous media has increased recently, mainly because of their use in many industrial processes (see Prat and Aga${\rm \ddot{e}}$sse~\cite{Prat} for more details).   A thin porous medium is a microstructured thin domain, that can be represented as a domain of thickness $\epsilon$ with
$0<\epsilon\ll 1$, perforated by an array of periodically distributed solid cylinders of diameter $a_\epsilon$, where the parameter $0<a_\epsilon\ll 1$ tends to zero with $\epsilon$. 
Recently,  this problem has been analyzed in Anguiano and Su\'arez-Grau \cite{Anguiano_SuarezGrau, Anguiano_SuarezGrau2} and Frabricius {\it et al.}\! \cite{Fabricius}, where the behaviour of Newtonian or power law fluid flows through a thin porous medium is considered. Three classes of thin porous media were introduced.
\begin{itemize}
\item[-]  The {\it proportionally thin porous medium}, corresponding to the critical case where the cylinder   diameter is proportional to the interspatial distance, with $\nu$ the proportionality constant, \emph{i.e.}\! $a_\epsilon\approx \epsilon$, with $a_\epsilon/\epsilon\to \nu$, $0<\nu<+\infty$. 
\item[-]  The {\it homogeneously thin porous medium}, corresponding  to the case where the cylinder  diameter is much larger than the interspatial distance,  \emph{i.e.}\! $a_\epsilon\ll \epsilon$ which is equivalent to $\nu=0$.
\item[-]  The {\it very thin porous medium}, corresponding to the case where the cylinder   diameter is much smaller than the interspatial distance,  \emph{i.e.}\! $a_\epsilon\gg \epsilon$ which is equivalent to $\nu=+\infty$.
\end{itemize}

A lower-dimensional Darcy law is obtained in every case, but the {\it permeability operator} depends on the type of thin porous medium considered. More precisely, for a proportionally thin porous medium, the permeability operator is defined through the solutions of 3D local Stokes problems depending on $\nu$. For a homogeneously thin porous medium, this operator is defined through the solutions of 2D local Stokes problems, whereas for a very thin porous medium, it is defined through the solutions of 2D local Hele-Shaw problems.

We  also refer to other recent studies by Anguiano \cite{Anguiano_MMAS, Anguiano_ZAMM,  Anguiano_Derivation}, Anguiano and Su\'arez-Grau \cite{Anguiano_SuarezGrau_Derivation, AnguianoSG_small, AnguianoSG_sharp, AnguianoSG_cilindros_delgados},  Bunoiu and Timofte \cite{Bunoiu_Timofte, BunoiuTimofte},  Fabricius {\it et al.} \cite{Fabricius3}, Jouybari  and Lundstr$\ddot{\rm o}$m \cite{Lundstrom},  Su\'arez-Grau \cite{SG_MN},  Yeghiazarian {\it et al.}\! \cite{Rosati} and Zhengan and Hongxing \cite{ZZ}, where the behaviour of Newtonian or power law fluids through different types of thin porous medium is considered. For the case of a Bingham flow we refer to Anguiano and Bunoiu \cite{Ang-Bun2}, and for the case of micropolar fluids we refer to Su\'arez-Grau \cite{SG_micropolar}. The case of a fluid flow through a thin porous medium with slip boundary conditions on the cylinders is considered in Anguiano and Su\'arez-Grau \cite{Anguiano_SG_Net} and Fabricius and Gahn \cite{Fabricius2}.  On the other hand, the first study about the reaction-diffusion equation in a thin porous medium was done recently in Anguiano~\cite{Anguiano_BMMS}, the two-phase flow problem in thin porous domains of Brinkman type has been considered in Armiti-Juber \cite{Armiti}, and an approach for effective heat transport in thin porous media has been derived by Scholz and Bringedal \cite{Scholz}. However, the literature on Carreau fluid flows in this type of domains is far less complete, although these problems have now become of great practical relevance to Chemical
Industry and Rheology, for instance in injection moulding of melted polymers, flow of oils, muds, etc. (see for example Pereire and Lecampion \cite{Pereira} and  Wrobel {\it et al}.\! \cite{Wrobel0, Wrobel}).

In this paper, we consider a  thin porous medium $\Omega_\epsilon =\omega_\epsilon\times (0,\epsilon)\subset \mathbb{R}^3$ of small height $\epsilon$ which is perforated by an array of periodically distributed solid cylinders of diameter of size $\epsilon$ (see Figure \ref{fig:omep}). Observe that this corresponds to the case of a proportionally thin porous medium with $\nu=1$. Here, the bottom of the domain without perforations $\omega\subset\mathbb{R}^2$ is made of two parts, the fluid part $\omega_\epsilon$ and the solid part $\omega\setminus\omega_\epsilon$. Similarly to     \cite{Tapiero2, Bourgeat1},  assuming that the flow is sufficiently slow to neglect inertial effects,  we consider that the generalized Newtonian fluid flow  through the thin porous medium $\Omega_\epsilon =\omega_\epsilon\times (0,\epsilon)$ is governed by the  stationary Stokes system (\ref{Bourgeat_Mikelic}) with a non-linear viscosity following the {\it Carreau law} (\ref{Carreau}) and scaled by a factor $\eps^\gamma$.

This problem has been very recently considered in Anguiano {\it et al.}\!~\cite{Carreau_Ang_Bonn_SG} in the particular case of a pseudoplastic fluid ($1<r<2$) with $\gamma=1$. Using homogenization techniques, it was proved that when $\epsilon$ tends to zero, the mean global filtration velocity is given, as a function of the pressure gradient, by a non-linear 2D Darcy law of Carreau type
\begin{equation}\label{intro_carreau}\left\{\begin{array}{l}
\medskip
\displaystyle
V'(x')=\mathcal{U}\left(f'(x')-\nabla_{x'}p(x')\right),\quad V_3(x')=0\ \hbox{in }\omega,\\
\medskip
{\rm div}_{x'} V'(x')=0\ \hbox{ in }\omega,\quad V'(x')\cdot n=0\ \hbox{ on }\partial\omega.
\end{array}
\right.
\end{equation}
In the above system, $n$ is the outward normal to $\partial\omega$, $V'=(V_1,V_2)$, $x'=(x_1,x_2)$, and  the {\it permeability operator} $\mathcal{U}:\mathbb{R}^2\to \mathbb{R}^2$ is defined through
the solutions of 3D local non-Newtonian Stokes problems with non-linear viscosity following the {\it Carreau law} (\ref{Carreau}) and posed in a reference cell.

In this paper, we perform a new and complete study on the asymptotic behaviour of Carreau fluids, modeled by the equations of motion \eqref{Bourgeat_Mikelic} in the thin porous medium $\Omega_{\epsilon}$, depending on the type of fluid and the value of $\gamma$. We generalize the results obtained in \cite{Carreau_Ang_Bonn_SG}  by considering system (\ref{Bourgeat_Mikelic}) not only for pseudoplastic fluids, but also for dilatant fluids and Newtonian fluids, and moreover, for any exponent $\gamma\in\mathbb{R}$. Starting from problem \eqref{Bourgeat_Mikelic} and using homogenization techniques, we derive different effective problems depending on the type of fluid and the value of $\gamma$, describing the asymptotic behaviour of the model as $\epsilon$ tends to zero. The approach that we use relies strongly on an adaptation of the periodic unfolding method introduced by Cioranescu {\it et al.}\!~\cite{CDG,Cioran-book}.

In order to give a taste of the kind of arguments that will allow us to distinguish between the different regimes related to $r$ and $\gamma$, let us give the heuristics of the obtention of the effective system in the pseudoplastic case $1<r<2$. The so-called \emph{unfolded velocity and pressure} $(\hat u_\epsilon,\hat P_\epsilon)$ (defined in Section~\ref{sec:unfolding}) satisfy for any admissible test function $\varphi$ the following inequality:
\begin{equation*}\label{v_ineq_carreau_12}\begin{array}{l}
\displaystyle
\medskip
 (\eta_0-\eta_\infty)\int_{\omega\times Z}(1+\lambda\epsilon^{2(1-\gamma)}|\mathbb{D}_z[\varphi]|^2)^{{r\over 2}-1}\mathbb{D}_z[\varphi]:\mathbb{D}_z[\varphi-\epsilon^{\gamma-2}  \hat u_\epsilon]\,dx'dz\\
\medskip
\displaystyle
+\ \eta_\infty\int_{\omega\times Z}\mathbb{D}_z[\varphi]:\mathbb{D}_z[\varphi-\epsilon^{\gamma-2}  \hat u_\epsilon]dx'dz\\
\medskip
\displaystyle- \int_{\omega\times Z}\hat P_\epsilon\, {\rm div}_{x'}(\varphi'-\epsilon^{\gamma-2}  \hat u'_\epsilon)\,dx'dz\ge  \int_{\omega\times Z} f'\cdot (\varphi'-\epsilon^{\gamma-2}  \hat u'_\epsilon)\,dx'dz+O_\epsilon,
\end{array}
\end{equation*}
where $O_\epsilon$ tends to zero with $\epsilon$.
Also,  $(\epsilon^{\gamma-2}\hat u_\epsilon,\hat P_\epsilon)$ converges in appropriate Sobolev spaces to a pair of functions called $(\hat u, \tilde P)$. We refer to Sections~\ref{sec:compactness} and~\ref{sec:proofs_of_thms} (in particular equation~\eqref{v_ineq_carreau_12}) for more details. 

Then,  we observe that if $\gamma<1$, $2(1-\gamma)>0$ so $\lambda\epsilon^{2(1-\gamma)}|\mathbb{D}_y[\varphi]|^2$ tends to zero, whereas for $\gamma>1$, $2(1-\gamma)<0$ so $(1+\lambda\epsilon^{2(1-\gamma)}|\mathbb{D}_y[\varphi]|^2)^{{r\over 2}-1}$ tends to zero. As a consequence, the sum of the two first terms in the previous formulation converges to the linear term
$$ \eta\int_{\omega\times Z}\mathbb{D}_z[\varphi]:\mathbb{D}_z[\varphi-\hat u]dx'dz,$$
with $\eta=\eta_0$ if $\gamma<1$ and $\eta=\eta_\infty$ if $\gamma>1$. In case $\gamma=1$,  the critical case that couples the  nonlinear term and the linear one, their sum converges to 
$$(\eta_0-\eta_\infty)\int_{\omega\times Z}(1+\lambda|\mathbb{D}_z[\varphi]|^2)^{{r\over 2}-1}\mathbb{D}_z[\varphi]:\mathbb{D}_z[\varphi-\hat u]\,dx'dz
+ \eta_\infty\int_{\omega\times Z}\mathbb{D}_z[\varphi]:\mathbb{D}_z[\varphi-\hat u]dx'dz.$$
Thus, we obtain three different asymptotic behaviours depending on whether the value of $\gamma$ is smaller, equal or greater than $1$. 

For dilatant fluids,  there exist three different convergences of the unfolding velocity depending on the value of $\gamma$ and, as consequence, three different homogenized models are derived.

\paragraph{Summary of the different asymptotic regimes.}
 
In summary, we have the following asymptotic behaviours of Carreau fluids depending on the  the value of $\gamma$ and the type of fluid:
\begin{itemize}
\item If $\gamma<1$, regardless of the value of $r$, the effective problem is the linear 2D Darcy law
\begin{equation}\label{resultintro1}
\left\{\begin{array}{l}
\medskip
\displaystyle
V'(x')={1\over \eta}\mathcal{A}\left(f'(x')-\nabla_{x'} p(x')\right),\quad V_3(x')=0\quad \hbox{in }\omega,\\
\medskip
\displaystyle
{\rm div}_{x'} V'(x')=0\ \hbox{ in }\omega,\quad V'(x')\cdot n=0\ \hbox{ on }\partial\omega,
\end{array}
\right.
\end{equation}
where the permeability tensor $\mathcal{A}\in \mathbb{R}^{2\times 2}$ is obtained
by solving 3D local Newtonian Stokes problems, posed in a reference cell containing the
information on the obstacles' geometry. The viscosity $\eta$ is equal to $\eta_0$.

\item If $\gamma=1$, the asymptotic behaviour of the model depends on the fact that the fluid is Newtonian of not.
\begin{itemize}
\item  For $1<r<\infty$ with $r\neq 2$,  the effective problem is the non-linear 2D Darcy law of {\it Carreau} type (\ref{intro_carreau}), which is obtained in \cite{Carreau_Ang_Bonn_SG}. 
\item For $r=2$,  the effective problem is the linear 2D Darcy law (\ref{resultintro1}) with viscosity $\eta=\eta_0$.
\end{itemize} 

\item If $\gamma>1$, then pseudoplastic, Newtonian and dilatant fluid flows have distinct asymptotic properties.
\begin{itemize}
\item  For $r\in (1,2)$, the effective problem is the linear 2D Darcy law (\ref{resultintro1}) with viscosity $\eta=\eta_\infty$.
\item For $r=2$, the effective problem is the linear 2D Darcy law (\ref{resultintro1}) with viscosity $\eta=\eta_0$.
\item For $r>2$, the effective problem is a non-linear 2D Darcy law of \emph{power law} type
$$\left\{\begin{array}{l}
\medskip
\displaystyle
V'(x')=\frac{\lambda^{r'/2-1}}{(\eta_0-\eta_{\infty})^{r'-1}}\mathcal{U}\left(f'(x')-\nabla_{x'}\tilde P(x')\right),  \quad V_3(x')=0\ \hbox{in }\omega,\\
\medskip
{\rm div}_{x'} V'(x')=0\ \hbox{ in }\omega,\quad V'(x')\cdot n=0\ \hbox{ on }\partial\omega,
\end{array}
\right.$$
where the permeability operator $\mathcal{U}:\mathbb{R}^2\to \mathbb{R}^2$ is defined through
the solutions of  3D local non-Newtonian Stokes problems with non-linear viscosity following the {\it power law} (\ref{PowerLaw}) and posed in a reference cell.
\end{itemize} 
\end{itemize}

In Table \ref{table_asymp1}, we summarize every asymptotic behaviour of the Carreau flow governed by (\ref{Bourgeat_Mikelic}) depending on the type of fluid and the value of $\gamma$:\\

\begin{table}[h!]\centering
\begin{tabular}{|c||c|c|c|}
\hline
 &   $1<r<2$  &  $r=2$ &  $r>2$\\
 \hline \hline
 $\gamma<1$ & \begin{tabular}{c}
 {\small Linear 2D Darcy's law}\\
{\small (viscosity $\eta_0$)}
 \end{tabular}&{\small Linear } &    \begin{tabular}{c}
 {\small Linear 2D Darcy's law}\\
{\small (viscosity $\eta_0$)}
 \end{tabular}
 \\  \cline{1-2}\cline{4-4}
  $\gamma=1$ &  \begin{tabular}{c}{\small  Non-linear  2D Darcy's law }\\
  {\small (Carreau type)}
  \end{tabular} & {\small  2D  Darcy's law }& \begin{tabular}{c}{\small Non-linear 2D Darcy's law }\\
  (Carreau type)
  \end{tabular}\\  \cline{1-2}\cline{4-4}
  $\gamma>1$ & \begin{tabular}{c} {\small Linear 2D Darcy's law }\\
  {\small (viscosity $\eta_\infty$)}
  \end{tabular}& {\small (viscosity $\eta_0$)}& \begin{tabular}{c} {\small Non-linear 2D Darcy's law }
  \\
  {\small (power law type)}
  \end{tabular}\\ \cline{1-2}\cline{4-4}
  \hline
\end{tabular} 
\caption{Asymptotic behaviours of Carreau fluids depending on the values of $r$ and $\gamma$.}
\label{table_asymp1}
\end{table}

\paragraph{How can one determine the exponent $\gamma$ in practice?}

As revealed by the asymptotic analysis carried out in~\cite{Bourgeat1} and in the present contribution, the computation of the exponent $\gamma$ is crucial to characterize the flow of pseudoplastic and dilatant fluids through porous media. The presence of the scaling factor $\eps^\gamma$ expresses the modelling assumption that the viscosities $\eta_0, \eta_{\infty}$ are the only parameters from the Carreau law that are affected by the small geometrical scale $\eps$ (and that the effect is the same on both viscosities), while the other parameters $r$ and $\lambda$ are independent on $\eps$. Under these assumptions, one possible strategy to compute the exponent $\gamma$ associated to a specific problem could be the following.

When dealing with a quasi-Newtonian fluid, start by fitting experimentally the parameters $\eta_0$, $\eta_\infty$, $r$, $\lambda$ appearing in the Carreau law~\eqref{Carreau}. We refer to~\cite[Table 4.1-1]{Bird} and~\cite[Table 3]{Yasuda} for examples of experimental results related to the shear flow of polysterene solutions. After obtaining these values, one has to apply the same technique to the flow of the fluid under study in a thin porous domain of characteristic size $\eps$. If the parameters $r$ and $\lambda$ are indeed invariant, the adjusted values of the rescaled viscosities $\eps^\gamma \eta_0$ and $\eps^\gamma \eta_\infty$ give then access to the value of $\eps^\gamma$, hence to $\gamma$ itself.

The structure of the paper is as follows. In Section \ref{sec:main} we introduce the domain, make the statement of the problem and give the main results (Theorems \ref{mainthmPseudo}, \ref{mainthmDilatant} and \ref{mainthmNewtonian}). The  proofs of the main results are provided in Section \ref{sec:proofs}. Finally, we perform in Section~\ref{Section:Numerics} a numerical study of the different effective systems described in Theorems~\ref{mainthmPseudo}, \ref{mainthmDilatant} and \ref{mainthmNewtonian}, based on the computation of permeability tensors $\mathcal A$ and permeability operators $\mathcal U$ using a finite element method. A list of references completes the paper.

\section{Setting of the problem and main result}\label{sec:main}
\paragraph{Geometrical setting.} The periodic porous medium is defined by a domain $\omega$ and an associated microstructure, or periodic cell $Z^{\prime}=(-1/2,1/2)^2$, which is made of two complementary parts: the fluid part $Z^{\prime}_{f}$, and the solid part $T^{\prime}$ ($Z^{\prime}_f  \bigcup T^{\prime}=Z^\prime$ and $Z^{\prime}_f  \bigcap T^{\prime}=\emptyset$). More precisely, we assume that $\omega$ is a smooth, bounded, connected set in $\mathbb{R}^2$ with smooth enough boundary $\partial\omega$,  that $n$ is the outward normal to $\partial\omega$, and that $T^{\prime}$ is an open connected subset of $Z^\prime$ with a smooth boundary $\partial T^\prime$, such that $\overline T^\prime$ is strictly included  in $Z^\prime$.

The microscale of the porous medium is a small positive number ${\epsilon}$. The domain $\omega$ is covered by a regular mesh of squares of size ${\epsilon}$: for $k^{\prime}\in \mathbb{Z}^2$, each cell $Z^{\prime}_{k^{\prime},{\epsilon}}={\epsilon}k^{\prime}+{\epsilon}Z^{\prime}$ is divided in a fluid part $Z^{\prime}_{f_{k^{\prime}},{\epsilon}}$ and a solid part $T^{\prime}_{k^{\prime},{\epsilon}}$, \emph{i.e.}\! is similar to the unit cell $Z^{\prime}$ rescaled to size ${\epsilon}$. We define $Z=Z^{\prime}\times (0,1)\subset \mathbb{R}^3$, which is divided in a fluid part $Z_{f}=Z'_f\times (0,1)$ and a solid part $T=T'\times(0,1)$, and consequently $Z_{k^{\prime},{\epsilon}}=Z^{\prime}_{k^{\prime},{\epsilon}}\times (0,1)\subset \mathbb{R}^3$, which is also divided in a fluid part $Z_{f_{k^{\prime}},{\epsilon}}$ and a solid part $T_{{k^{\prime}},{\epsilon}}$ (see Figures \ref{fig:cell} and \ref{fig:cellep}).

We denote by $\tau(\overline T'_{k',\epsilon})$ the set of all translated images of $\overline T'_{k',\epsilon}$. The set $\tau(\overline T'_{k',\epsilon})$ represents the obstacles in $\mathbb{R}^2$.

\begin{figure}[h!]
\begin{center}
\includegraphics[width=4cm]{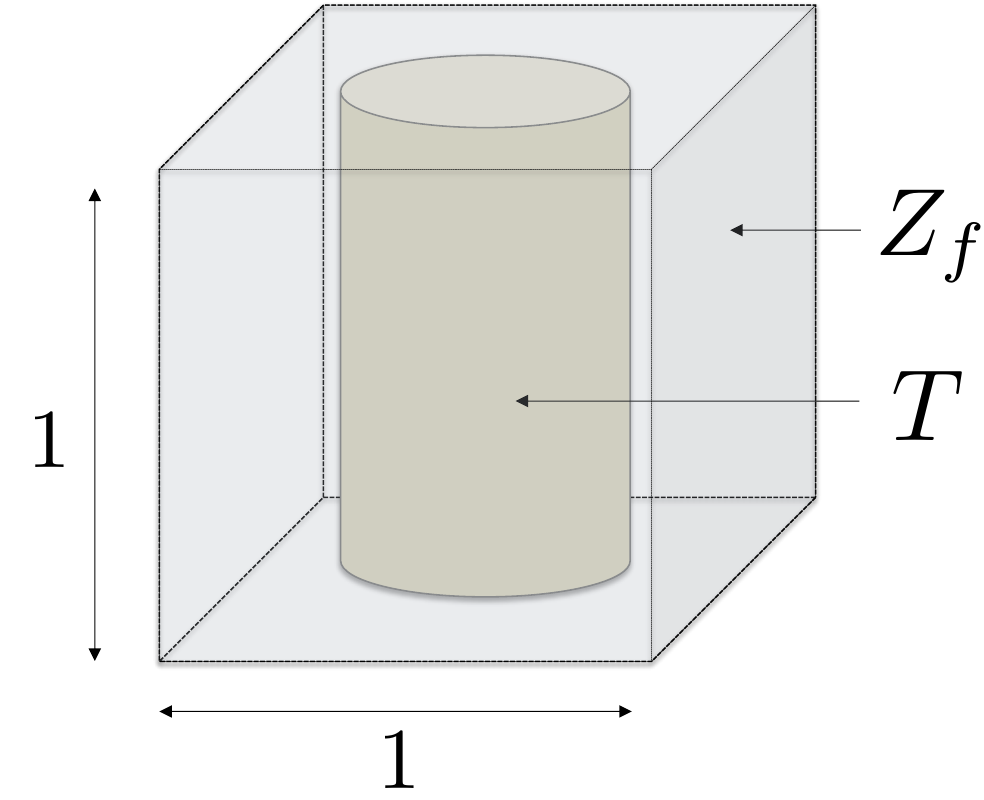}
\hspace{2.5cm}
\raisebox{.1\height}{\includegraphics[width=3cm]{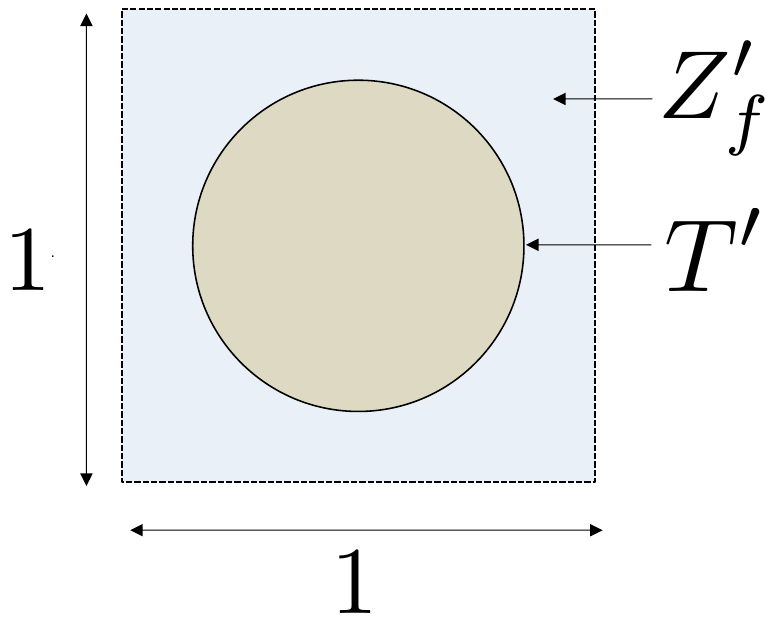}}
\end{center}
\vspace{-0.4cm}
\caption{View of the 3D reference cells  $Z$ (left) and the 2D reference cell $Z'$ (right).}
\label{fig:cell}
\end{figure}

\begin{figure}[h!]
\begin{center}
\includegraphics[width=3cm]{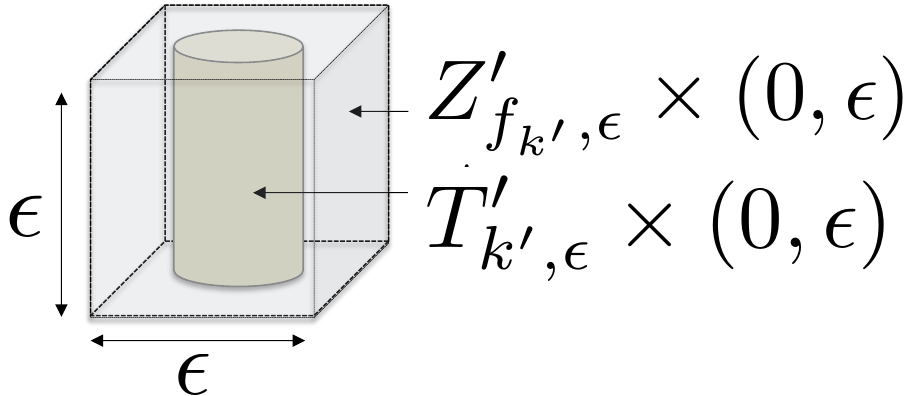}
\hspace{2.5cm}
\raisebox{.1\height}{\includegraphics[width=2cm]{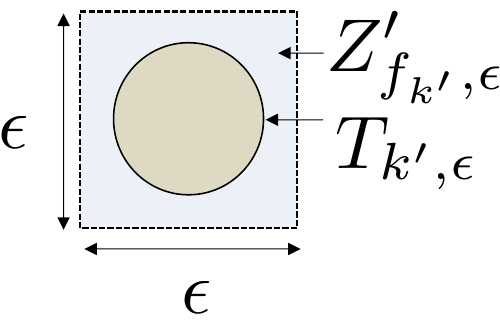}}
\end{center}
\vspace{-0.4cm}
\caption{View of the 3D reference cells  $Z_{k',\epsilon}$ (left) and the 2D reference cell $Z'_{k',\epsilon}$ (right).}
\label{fig:cellep}
\end{figure}

The fluid part of the bottom $\omega_{\epsilon}\subset \mathbb{R}^2$ of a porous medium is defined by $\omega_{\epsilon}=\omega\backslash\bigcup_{k^{\prime}\in \mathcal{K}_{\epsilon}} \overline T^{\prime}_{{k^{\prime}},{\epsilon}},$ where $\mathcal{K}_{\epsilon}=\{k^{\prime}\in \mathbb{Z}^2: Z^{\prime}_{k^{\prime}, {\epsilon}} \cap \omega \neq \emptyset \}$.  The whole fluid part $\Omega_{\epsilon}\subset \mathbb{R}^3$ in the thin porous medium is defined by (see Figure \ref{fig:omep})
\begin{equation}\label{Dominio1}
\Omega_{\epsilon}=\{  (x_1,x_2,x_3)\in \omega_{\epsilon}\times \mathbb{R}: 0<x_3<\epsilon \}.
\end{equation}
We assume that the obstacles $\tau(\overline T'_{k',\epsilon})$ do not intersect the boundary $\partial\omega$ and we denote by $S_\epsilon$ the set of the solid cylinders contained in $\Omega_\epsilon$, \emph{i.e.}\!  $S_\epsilon=\bigcup_{k^{\prime}\in \mathcal{K}_{\epsilon}} T'_{k^\prime, \epsilon}\times (0,\epsilon)$.
\begin{figure}[h!]
\begin{center}
\includegraphics[width=7.5cm]{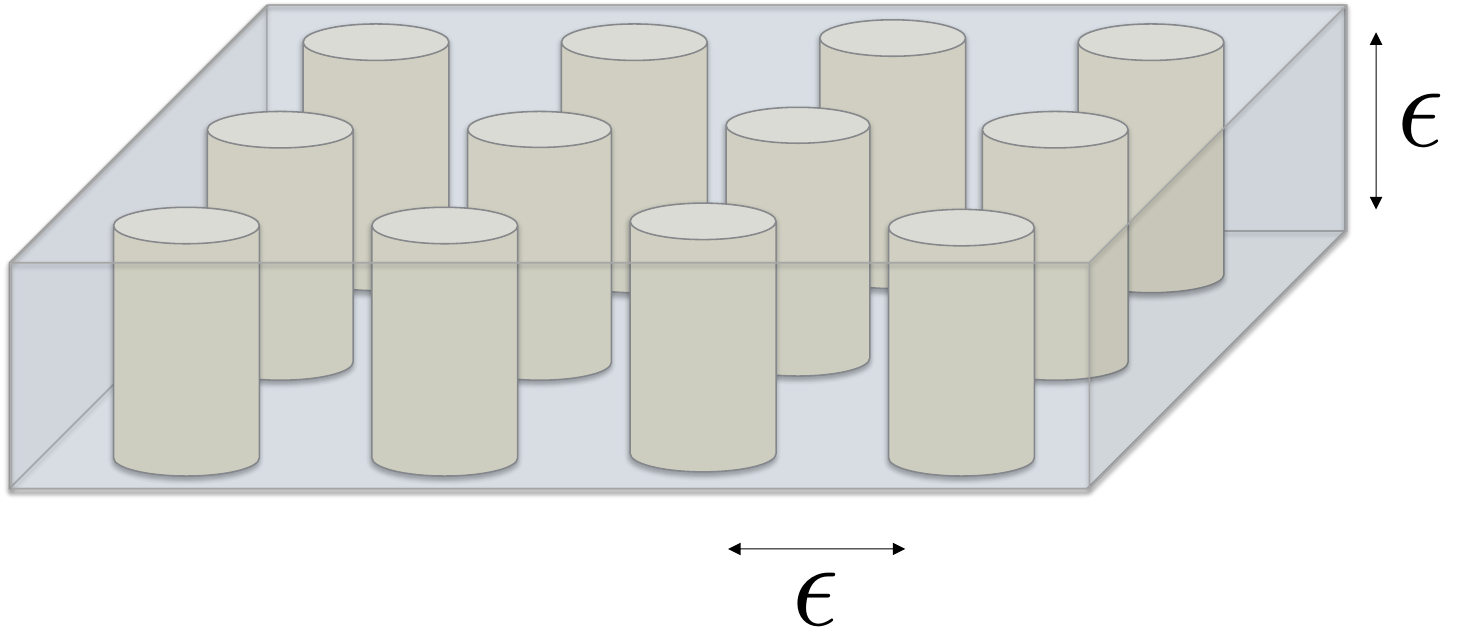}
\hspace{1cm}
\includegraphics[width=7.5cm]{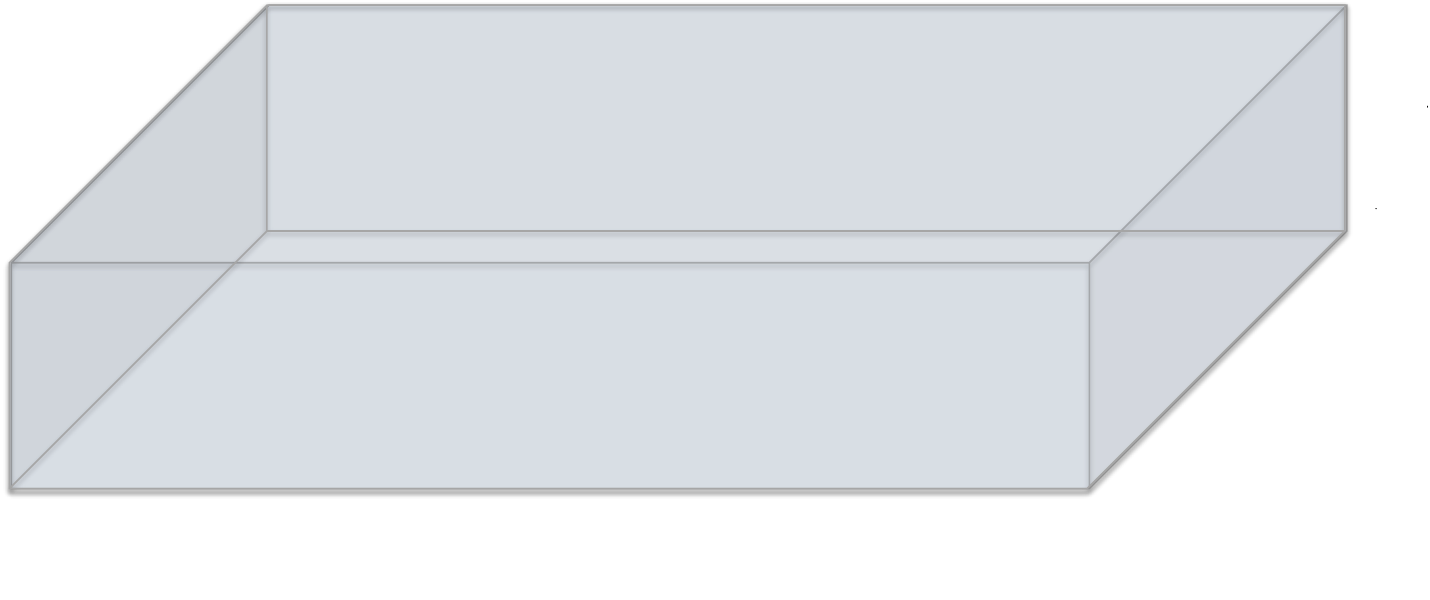}
\end{center}
\vspace{-0.4cm}
\caption{View of the thin porous media $\Omega_\epsilon$ (left) and domain without perforations $Q_\epsilon$ (right).}
\label{fig:omep}
\end{figure}
\\

We define 
\begin{equation}\label{OmegaTilde}
\widetilde{\Omega}_{\epsilon}=\omega_{\epsilon}\times (0,1), \quad \Omega=\omega\times (0,1), \quad Q_\epsilon=\omega\times (0,\epsilon).
\end{equation}
We observe that $\widetilde{\Omega}_{\epsilon}=\Omega\backslash \bigcup_{k^{\prime}\in \mathcal{K}_{\epsilon}} \overline T_{{k^{\prime}}, {\epsilon}},$ and we define $T_\epsilon=\bigcup_{k^{\prime}\in \mathcal{K}_{\epsilon}} T_{k^\prime, \epsilon}$ as the set of the solid cylinders contained in $\widetilde \Omega_\epsilon$.

To finish, we introduce some notation that will be useful throughout the paper. The points $x\in\mathbb{R}^3$ will be decomposed as $x=(x^{\prime},x_3)$ with $x^{\prime}=(x_1,x_2)\in \mathbb{R}^2$, $x_3\in \mathbb{R}$. We also use the notation $x^{\prime}$ to denote a generic vector of $\mathbb{R}^2$. 

Let $C^\infty_{\#}(Z)$ be the space of infinitely differentiable functions in $\mathbb{R}^3$ that are $Z'$-periodic. By $L^q_{\#}(Z)$ (resp. $W^{1,q}_{\#}(Z)$), $1<q<+\infty$, we denote its completion in the norm $L^q(Z)$ (resp. $W^{1,q}(Z)$) and by $L^q_{0,\#}(Z)$  the space of functions in $L^q_{\#}(Z)$ with zero mean value.

 We denote by $W^{1,q}_{0,\#}(Z_f)$ the subspace of $W^{1,q}(Z)$ composed of functions vanishing in $T$, with zero trace on $Z'\times \{0,1\}$. For $q=2$, we set $H^1(Z)=W^{1,2}(Z)$ and $H^{1}_{0,\#}(Z_f)=W^{1,2}_{0,\#}(Z_f)$. 

 For a vectorial function $\varphi=(\varphi',\varphi_3)$ and a scalar function $\psi$, we will denote $\mathbb{D}_{x^\prime}\left[\varphi\right]=\frac{1}{2}(D_{x^\prime}\varphi+D_{x^\prime}^t \varphi)$ and $\partial_{z_3}\left[\varphi\right]=\frac{1}{2}(\partial_{z_3}\varphi+\partial_{z_3}^t \varphi)$, where $\partial_{z_3}=(0,0,\frac{\partial}{\partial z_3})^t$.

Finally, we denote by $O_\epsilon$ a generic real sequence, which tends to zero with $\epsilon$ and can change from line to line, and by $C$ a generic positive constant which also can change from line to line.

\paragraph{Statement of the problem.}  We consider the following stationary Stokes system with non-linear viscosity following the {\it Carreau law}  (\ref{Carreau}) in $\Omega_\epsilon$, with a zero boundary condition on the exterior boundary $\partial Q_\epsilon$ and the cylinders $\partial S_\epsilon$:
\begin{equation}\label{1}
\left\{\begin{array}{rl}\displaystyle
\medskip
-\epsilon^\gamma{\rm div}\left(\eta_r(\mathbb{D}[u_\epsilon])\mathbb{D}[u_\epsilon]\right)+ \nabla p_\epsilon = f & \hbox{in }\Omega_\epsilon,\\
\medskip
\displaystyle
{\rm div}\,u_\epsilon=0& \hbox{in }\Omega_\epsilon,\\
\medskip
\displaystyle
u_\epsilon=0& \hbox{on }\partial Q_\epsilon\cup \partial S_\epsilon,
\end{array}\right.
\end{equation}
where the source term $f$ is of the form
\begin{equation}\label{fassump}
f (x)=(f'(x'),0)\quad \hbox{with }f'\in L^\infty(\omega)^2.
\end{equation}
Notice that the assumptions of neglecting the vertical component  of the exterior force and  the independence of the vertical variable  are usual when dealing with fluids through thin domains (see \cite{Tapiero2} for more details).

Under previous assumptions, the classical theory (see for instance \cite{Tapiero2,Bourgeat1,Lions2}), gives the existence of a unique weak  solution $(u_\epsilon, p_\epsilon)\in H^1_0(\Omega_\epsilon)^3\times L^{2}_0(\Omega_\epsilon)$, for $1<r\leq 2$, and $(u_\epsilon, p_\epsilon)\in W^{1,r}_0(\Omega_\epsilon)^3\times L^{r'}_0(\Omega_\epsilon)$ with $1/r+1/r'=1$, for $r> 2$, where  $L^{2}_0$ (respectively $L^{r'}_0$) is the space of functions of $L^2$ (respectively $L^{r'}$) with zero mean value.

The goal of this paper is to study the asymptotic behaviour of $u_{\epsilon}$ and $p_{\epsilon}$ when $\epsilon$  tends to zero. For this purpose, we use the dilatation in the variable $x_3$ as follows
\begin{equation}\label{dilatacion}
z_3=\frac{x_3}{\epsilon},
\end{equation}
which allows to define the functions   in the open set $\widetilde\Omega_{\epsilon}$, which has a fixed height. Accordingly, we define $\tilde{u}_{\epsilon}$ and $\tilde{p}_{\epsilon }$ by $$\tilde{u}_{\epsilon }(x^{\prime},z_3)=u_{\epsilon }(x^{\prime}, \epsilon z_3),\text{\ \ }\tilde{p}_{\epsilon }(x^{\prime},z_3)=p_{\epsilon }(x^{\prime}, \epsilon z_3), \text{\ \ } \textrm{ a.e.} (x^{\prime},z_3)\in \widetilde{\Omega}_{\epsilon }.$$
Moreover, associated to the change of variables (\ref{dilatacion}), we introduce the {\color{darkgreen} rescaled} operators $\mathbb{D}_{\epsilon}$, $D_{\epsilon}$, ${\rm div}_{\epsilon}$ and $\nabla_{\epsilon}$, defined by
\begin{equation*}
\mathbb{D}_{\epsilon}[\varphi]=\frac{1}{2}\left(D_{\epsilon} \varphi+D^t_{\epsilon} \varphi \right),
\end{equation*}
\begin{equation*}
(D_{\epsilon}\varphi)_{i,j}=\partial_{x_j}\varphi_i\text{\ for \ }i=1,2,3,\ j=1,2,\quad 
(D_{\epsilon}\varphi)_{i,3}=\epsilon^{-1}\partial_{z_3}\varphi_i\text{\ for \ }i=1,2,3,
\end{equation*}
\begin{equation*}
{\rm div}_{ \epsilon}\varphi={\rm div}_{x^{\prime}}\varphi^{\prime}+  \epsilon^{-1}\partial_{z_3}\varphi_3,\quad
\nabla_{ \epsilon}\psi=(\nabla_{x^{\prime}}\psi,  \epsilon^{-1}\partial_{z_3}\psi)^t.
\end{equation*}
Using the change of variable (\ref{dilatacion}), system (\ref{1}) can be rewritten as
\begin{equation}
\left\{
\begin{array}
[c]{rl}%
\medskip
\displaystyle -\epsilon^\gamma {\rm div}_{ \epsilon} \left(  \eta_r(\mathbb{D}_\epsilon[\tilde u_\epsilon])\mathbb{D}_{ \epsilon}\left[\tilde{u}_{\epsilon}\right] \right)+ \nabla_{ \epsilon} \tilde{p}_{\epsilon }
=
f & \text{\ in \ }\widetilde{\Omega}_{\epsilon },\\
\medskip
{\rm div}_{\epsilon} \tilde{u}_{\epsilon } = 0 &\text{\ in \ }\widetilde{\Omega}_{\epsilon },\\
\medskip
\tilde{u}_{\epsilon}  = 0 & \text{\ on \ } \partial {\Omega}\cup \partial T_\epsilon.
\end{array}
\right. \label{2}%
\end{equation}
Our goal is to describe the asymptotic behaviour of this new sequence $(\tilde{u}_{\epsilon}$, $\tilde{p}_{\epsilon})$. The first difficulty in performing this task is that this sequence is not defined in a fixed domain, but in the set $\widetilde{\Omega}_{\epsilon}$ that varies with $\epsilon$. Since we need convergences in fixed Sobolev spaces (defined in $\Omega$) to pass to the limit when $\epsilon$ goes to zero, we have to extend $(\tilde{u}_{\epsilon}$, $\tilde{p}_{\epsilon})$ to the whole domain $\Omega$. To this aim, we define an extension $(\tilde{u}_{\epsilon}, \tilde{P}_{\epsilon})\in W_0^{1,q}(\Omega)^3\times L^{q'}_0(\Omega)$ which coincides with $(\tilde{u}_{\epsilon}$, $\tilde{p}_{\epsilon})$ on $\widetilde{\Omega}_{\epsilon}$. Note that, for simplicity, we use the same notation $\tilde u_\epsilon$ for the velocity in $\widetilde\Omega_\epsilon$ and its continuation in $\Omega$.

Our main results are given by the following theorems. 

\begin{theorem}[Pseudoplastic fluids]\label{mainthmPseudo}
Consider $1<r<2$ and $\gamma\in \mathbb{R}$. Then, there exist $\tilde u\in H^1(0,1;L^2(\omega)^3)$, with $\tilde u=0$ on $\omega \times \{0, 1\}$ and $\tilde u_3\equiv 0$, and $\tilde P\in L^{2}_0(\omega)$, such that the extension $(\tilde u_{\epsilon},\tilde P_{\epsilon})$ of the solution of (\ref{2}) satisfies the following convergences:
$$\epsilon^{\gamma-2}\tilde u_{\epsilon }\rightharpoonup \tilde  u\quad\hbox{weakly in } H^1(0,1;L^2(\omega)^3),\quad \tilde   P_{\epsilon }\to \tilde P\quad\hbox{strongly in }L^{2}(\Omega).$$
Moreover, defining $\tilde V(x')=\int_0^1\tilde u(x',z_3)\,dz_3$, the pair $(\tilde V, \tilde P)\in L^2(\omega)^3\times (L^{2}_0(\omega)\cap H^1(\omega))$ is the unique solution of a lower-dimensional effective Darcy's law depending on the value of $\gamma$. More precisely:
\begin{itemize}
\item[--] If $\gamma\neq 1$, then $(\tilde V, \tilde P)$  is the unique solution of the linear 2D Darcy's law
\begin{equation}\label{thm:system}
\left\{\begin{array}{l}
\medskip
\displaystyle
\tilde V'(x')={1\over \eta}\mathcal{A}\left(f'(x')-\nabla_{x'}\tilde P(x')\right),\quad \tilde V_3(x')=0\quad \hbox{in }\omega,\\
\medskip
\displaystyle
{\rm div}_{x'} \tilde V'(x')=0\ \hbox{ in }\omega,\quad \tilde V'(x')\cdot n=0\ \hbox{ on }\partial\omega,
\end{array}
\right.
\end{equation}
where  
$$\eta=\left\{\begin{array}{ll}\eta_0&\hbox{ if }\gamma<1,\\
\eta_\infty&\hbox{ if }\gamma>1.
\end{array}\right.$$ 
In system~\eqref{thm:system}, the permeability tensor $\mathcal{A}\in\mathbb{R}^{2\times 2}$ is defined by its entries
\begin{equation}\label{permfuncNew}
\mathcal{A}_{ij}=\int_{Z_f}w^i_j(z)\,dz,\quad i,j=1,2,
\end{equation}
where for $i=1,2$, the pair $(w^i, \pi^i)\in H^1_{0,\#}(Z_f)^3\times L^2_{0,\#}(Z_f)$ is the unique solution of the local Stokes system
\begin{equation}\label{LocalProblemNewtonian}
\left\{\begin{array}{rl}
\medskip
\displaystyle
-\Delta_zw^i +\nabla_{z}\pi^i=e_i &\hbox{in }Z_f,
\\
\medskip
\displaystyle
{\rm div}_z w^i=0&\hbox{in }Z_f,
\\
\medskip
\displaystyle
w^i=0&\hbox{on }\partial T\cup (Z_f'\times \{0,1\}),
\\
\medskip
\displaystyle z\to  w^i, \pi^i & Z-\hbox{periodic},
\end{array}\right.
\end{equation}
and $\{e_i\}_{i=1,2,3}$ being the canonical basis of $\mathbb{R}^3$. 

\item[--] If $\gamma=1$, then $(\tilde V, \tilde P)$  is the unique solution of the non-linear 2D Darcy's law of Carreau type
\begin{equation}\label{thm:system_gamma1}
\left\{\begin{array}{l}
\medskip
\displaystyle
\tilde V'(x')= \mathcal{U}\left(f'(x')-\nabla_{x'}\tilde P(x')\right),\quad \tilde V_3(x')=0\quad \hbox{in }\omega,\\
\medskip
\displaystyle
{\rm div}_{x'} \tilde V'(x')=0\ \hbox{ in }\omega,\quad \tilde V'(x')\cdot n=0\ \hbox{ on }\partial\omega.
\end{array}
\right.
\end{equation}
The permeability operator $\mathcal{U}:\mathbb{R}^2\to \mathbb{R}^2$ appearing in system~\eqref{thm:system_gamma1} is defined by
\begin{equation}\label{permfunc}
\mathcal{U}(\xi')=\int_{Z_f}w'_{\xi'}(z)\,dz,\quad\forall\,\xi'\in\mathbb{R}^2,
\end{equation}
where, for any $\xi'\in\mathbb{R}^2$, the pair $(w_{\xi'}, \pi_{\xi'})\in H^1_{0,\#}(Z_f)^3\times L^{2}_{0,\#}(Z_f)$ is the unique solution of the local Stokes system \begin{equation}\label{LocalProblemNonNewtonian}
\left\{\begin{array}{rl}
\medskip
\displaystyle
-{\rm div}_z(\eta_r(\mathbb{D}_z[w_{\xi'}])\mathbb{D}_z[w_{\xi'}]) +\nabla_{z}\pi_{\xi'}=\xi' &\hbox{in }Z_f,
\\
\medskip
\displaystyle
{\rm div}_zw_{\xi'}=0&\hbox{in }Z_f,
\\
\medskip
\displaystyle
w_{\xi'}=0&\hbox{on }\partial T\cup (Z_f'\times \{0,1\}),
\end{array}\right.
\end{equation}
and the nonlinear viscosity $\eta_r$ is given by the  Carreau law (\ref{Carreau}).
\end{itemize}
\end{theorem}

\begin{remark}
 According to \cite[Theorem 1.1]{Allaire0}, the permeability tensor $\mathcal{A}$ is  symmetric and definite positive. 
\end{remark}

\begin{theorem}[Dilatant fluids]\label{mainthmDilatant}
Consider $r>2$ and $\gamma\in\mathbb{R}$. We divide the theorem depending on the value of $\gamma$:
\begin{itemize}
\item[$(i)$]   If $\gamma<1$, then there exist $\tilde u\in H^1(0,1;L^2(\omega)^3)$, with $\tilde u=0$ on $\omega \times \{0, 1\}$ and $\tilde u_3\equiv 0$,  and $\tilde P\in L^{r'}_0(\omega)$, such that the extension $(\tilde u_{\epsilon},\tilde P_{\epsilon})$ of the solution of (\ref{2}) satisfies the following convergences:
$$\epsilon^{\gamma-2}\tilde u_{\epsilon }\rightharpoonup \tilde  u\quad\hbox{weakly in } H^1(0,1;L^2(\omega)^3),\quad \tilde   P_{\epsilon }\to \tilde P\quad\hbox{strongly in }L^{r'}(\Omega).$$

Moreover, defining $\tilde V(x')=\int_0^1\tilde u(x',z_3)\,dz_3$, then the pair $(\tilde V, \tilde P)\in L^2(\omega)^3\times (L^{2}_0(\omega)\cap H^1(\omega))$ is the unique solution of the linear 2D Darcy's law (\ref{thm:system})$-$(\ref{LocalProblemNewtonian}) with $\eta=\eta_0$.

\item[$(ii)$] If $\gamma>1$, then there exist $\tilde u\in W^{1,r}(0,1;L^r(\omega)^3)$, with $\tilde u=0$ on $\omega \times \{0, 1\}$ and $\tilde u_3\equiv 0$, and $\tilde P\in L^{r'}_0(\omega)$, such that the extension $(\tilde u_{\epsilon},\tilde P_{\epsilon})$ of the solution of (\ref{2}) satisfies the following convergences:
$$\epsilon^{\gamma-r\over r-1}\tilde u_{\epsilon }\rightharpoonup \tilde  u\quad\hbox{weakly in } W^{1,r}(0,1;L^r(\omega)^3),\quad \tilde   P_{\epsilon }\to \tilde P\quad\hbox{strongly in }L^{r'}(\Omega).$$
Moreover, defining $\tilde V(x')=\int_0^1\tilde u(x',z_3)\,dz_3$, the pair $(\tilde V, \tilde P)\in L^r(\omega)^3\times (L^{r'}_0(\omega)\cap W^{1,r'}(\omega))$ is the unique solution of the lower-dimensional  effective non-linear Darcy's law
\begin{equation}\label{thm:system_22}
\left\{\begin{array}{l}
\medskip
\displaystyle
\tilde V'(x')=\frac{1}{\lambda^{\frac{2-r'}{2}}(\eta_0-\eta_{\infty})^{r'-1}}\mathcal{U}\left(f'(x')-\nabla_{x'}\tilde P(x')\right), \color{black} \quad \tilde  V_3(x')=0\quad \hbox{in }\omega,\\
\medskip
\displaystyle
{\rm div}_{x'} \tilde V'(x')=0\ \hbox{ in }\omega,\quad \tilde  V'(x')\cdot n=0\ \hbox{ on }\partial\omega.
\end{array}
\right.
\end{equation}
The permeability operator $\mathcal{U}:\mathbb{R}^2\to \mathbb{R}^2$ appearing in system~\eqref{thm:system_22} is defined by \eqref{permfunc} where, for any $\xi'\in\mathbb{R}^2$,  $w_{\xi'}$ is the unique solution of the local Stokes system \eqref{LocalProblemNonNewtonian} with nonlinear viscosity of type power law given by $\eta_r(\mathbb{D}_z[w_{\xi'}])=|\mathbb{D}_z[w_{\xi'}]|^{r-2}$.
\item[(iii)] If $\gamma=1$, then there exist $\tilde u\in W^{1,r}(0,1;L^r(\omega)^3)$, with $\tilde u=0$ on $\omega \times \{0, 1\}$ and $\tilde u_3\equiv 0$ and $\tilde P\in L^{r'}_0(\omega)$, such that the extension $(\tilde u_{\epsilon},\tilde P_{\epsilon})$ of the solution of (\ref{2}) satisfies the following convergences:
$$\epsilon^{-1}\tilde u_{\epsilon }\rightharpoonup \tilde  u\quad\hbox{weakly in } W^{1,r}(0,1;L^r(\omega)^3),\quad \tilde   P_{\epsilon }\to \tilde P\quad\hbox{strongly in }L^{r'}(\Omega).$$
Moreover, defining $\tilde V(x')=\int_0^1\tilde u(x',z_3)\,dz_3$, the pair $(\tilde V, \tilde P)\in L^r(\omega)^3\times (L^{r'}_0(\omega)\cap W^{1,r'}(\omega))$ is the unique solution of the lower-dimensional  effective non-linear 2D Darcy's law of Carreau type (\ref{thm:system_gamma1}). For every $\xi'\in \R^2$, $\mathcal{U}(\xi')$ is defined by (\ref{permfunc}) where $(w_{\xi'}, \pi_{\xi'})\in W^{1,r}_{0,\#}(Z_f)^3\times L^{r'}_{0,\#}(Z_f)$ is the unique solution of the local Stokes system (\ref{LocalProblemNonNewtonian}) with nonlinear viscosity given by the Carreau law (\ref{Carreau}).
\end{itemize}

\end{theorem}

\begin{remark}
\begin{itemize}
\item[--]  According to \cite[Lemma 2]{Bourgeat2}, the permeability operator $\mathcal{U}$ is  coercive and strictly monotone. 
 
 \item[--]  
 In case $r>2$ and $\gamma<1$, we derive a linear 2D Darcy's law, contrary to what is obtained  in~\cite{Tapiero2}, where the filtration velocity is zero (\emph{i.e.}\! $\tilde V\equiv 0$). This difference can be explained by the fact that the velocity estimates used in~\cite{Tapiero2} were not optimal in that case. 
 
 \end{itemize}
\end{remark}

\begin{theorem}[Newtonian fluids]\label{mainthmNewtonian}
Let $r=2$ and $\gamma\in \mathbb{R}$. Then, there exist $\tilde u\in H^1(0,1;L^2(\omega)^3)$, with $\tilde u=0$ on $\omega \times \{0, 1\}$ and $\tilde u_3\equiv 0$ and $\tilde P\in L^{2}_0(\omega)$, such that the extension $(\tilde u_{\epsilon},\tilde P_{\epsilon})$ of the solution of (\ref{2}) satisfies the following convergences:
$$\epsilon^{\gamma-2}\tilde u_{\epsilon }\rightharpoonup \tilde  u\quad\hbox{weakly in } H^1(0,1;L^2(\omega)^3),\quad \tilde   P_{\epsilon }\to \tilde P\quad\hbox{strongly in }L^{2}(\Omega).$$
Moreover, defining $\tilde V(x')=\int_0^1\tilde u(x',z_3)\,dz_3$, the pair $(\tilde V, \tilde P)\in L^2(\omega)^3\times (L^{2}_0(\omega)\cap H^1(\omega))$ is the unique solution of the linear 2D Darcy's law
\begin{equation}\label{thm:systemNewtonian}
\left\{\begin{array}{l}
\medskip
\displaystyle
\tilde V'(x')={1\over \eta_0}\mathcal{A}\left(f'(x')-\nabla_{x'}\tilde P(x')\right),\quad \tilde V_3(x')=0\quad \hbox{in }\omega,\\
\medskip
\displaystyle
{\rm div}_{x'} \tilde V'(x')=0\ \hbox{ in }\omega,\quad \tilde V'(x')\cdot n=0\ \hbox{ on }\partial\omega.
\end{array}
\right.
\end{equation}
The permeability tensor $\mathcal{A}\in\mathbb{R}^{2\times 2}$ appearing in system~\eqref{thm:systemNewtonian} is defined by its entries
\begin{equation}\label{permfuncNewNewtonian}
\mathcal{A}_{ij}=\int_{Z_f}w^i_j(z)\,dz,\quad i,j=1,2,
\end{equation}
where, setting $\{e_k\}_{k=1,2,3}$ the canonical basis of $\R^3$, for $i=1,2$, the pair $(w^i, \pi^i)\in H^1_{0,\#}(Z_f)^3\times L^2_{0,\#}(Z_f)$, is the unique solution of the local Stokes system
\begin{equation}\label{LocalProblemNewtonianNewtonian}
\left\{\begin{array}{rl}
\medskip
\displaystyle
-\Delta_zw^i +\nabla_{z}\pi^i=e_i &\hbox{in }Z_f,
\\
\medskip
\displaystyle
{\rm div}_z w^i=0&\hbox{in }Z_f,
\\
\medskip
\displaystyle
w^i=0&\hbox{on }\partial T\cup (Z_f'\times \{0,1\}),
\\
\medskip
\displaystyle z\to  w^i, \pi^i & Z-\hbox{periodic}.
\end{array}\right.
\end{equation}

\end{theorem}

\begin{remark}
 According to \cite[Theorem 1.1]{Allaire0}, the permeability tensor $\mathcal{A}$ is  symmetric and definite positive. 
\end{remark}

\section{Proof of the main results}\label{sec:proofs}
In this section we provide the proof of the main results (Theorems \ref{mainthmPseudo}, \ref{mainthmDilatant} and \ref{mainthmNewtonian}). To this aim, we first establish  some {\it a priori} estimates of the solution of (\ref{2}) and we define its extension. Second, we  introduce the version of the 
unfolding method depending on $\epsilon$.  Next, we establish a compactness result, which is the main key for passing to the limit in the system, and conclude the proof of the Theorems.

\subsection{{\it A priori} estimates}\label{sec:estimates}

In this subsection, we establish sharp {\it a priori} estimates on the dilated solution in $\widetilde \Omega_\epsilon$. One key ingredient are Poincar\'e and Korn inequalities in $\widetilde\Omega_\epsilon$, which are proved in~\cite{Anguiano_SuarezGrau}.
\begin{lemma}[Remark 4.3-(i) in \cite{Anguiano_SuarezGrau}] \label{Lemma_Poincare} We have the following two estimates in thin domains:
\begin{itemize}
\item[$(i)$] For every $\tilde  \varphi\in W^{1,q}_0(\widetilde \Omega_\epsilon)^3$, $1\leq q<+\infty$, there exists a positive constant $C$, independent of $\epsilon$, such that
\begin{equation}\label{Poincare}
\|\tilde \varphi\|_{L^q(\widetilde \Omega_\epsilon)^3}\leq C\epsilon\|D_\epsilon \tilde  \varphi\|_{L^q(\widetilde \Omega_\epsilon)^{3\times 3}},\quad (\hbox{Poincar\'e inequality}).
\end{equation}
\item[$(ii)$] For every $\tilde  \varphi\in W^{1,q}_0(\widetilde \Omega_\epsilon)^3$, $1< q<+\infty$, there exists a positive constant $C$, independent of $\epsilon$, such that
\begin{equation}\label{Korn}
\|D_\epsilon \tilde  \varphi\|_{L^q(\widetilde \Omega_\epsilon)^{3\times 3}}\leq C\|\mathbb{D}_\epsilon[\tilde  \varphi]\|_{L^q(\widetilde \Omega_\epsilon)^{3\times 3}},\quad (\hbox{Korn inequality}).
\end{equation}
\end{itemize}
\end{lemma}
Inequalities~\eqref{Poincare}-\eqref{Korn} allow us to derive estimates for the velocity  $\tilde u_\epsilon$ in $\widetilde \Omega_\epsilon$.
\begin{lemma} \label{Estimates_lemma}  The velocity $\tilde u_\epsilon$ solution of (\ref{2}) satisfies the following estimates, depending on the value of parameters $r$ and $\gamma$.
\begin{itemize}
\item[$(i)$] {\it (Pseudoplastic fluid and Newtonian fluid)} Assume that $1<r\leq 2$. There exists a positive constant $C$, independent of $\epsilon$, such that for every value of $\gamma$,
\begin{equation}\label{estimates_u_tilde}
\|\tilde u_\epsilon\|_{L^2(\widetilde \Omega_\epsilon)^{3}}\leq C\epsilon^{2-\gamma} ,\quad \|D_\epsilon \tilde u_\epsilon\|_{L^2(\widetilde \Omega_\epsilon)^{3\times 3}}\leq C\epsilon^{1-\gamma} ,\quad \|\mathbb{D}_\epsilon[\tilde u_\epsilon]\|_{L^2(\widetilde \Omega_\epsilon)^{3\times 3}}\leq C\epsilon^{1-\gamma} \,.
\end{equation}
\item[$(ii)$] {\it (Dilatant fluid)} Assume that $r>2$. There exists a positive constant $C$, independent of $\epsilon$, such that estimates (\ref{estimates_u_tilde}) hold true. Also, depending on the value of $\gamma$, we have:
\begin{itemize}
\item if $\gamma<1$,
\begin{equation}\label{estimates_u_tilde2less1}
\|\tilde u_\epsilon\|_{L^r(\widetilde \Omega_\epsilon)^{3}}\leq C\epsilon^{-{2\over r}(\gamma-1)+1} ,\  \|D_\epsilon \tilde u_\epsilon\|_{L^r(\widetilde \Omega_\epsilon)^{3\times 3}}\leq C\epsilon^{-{2\over r}(\gamma-1)} ,\  \|\mathbb{D}_\epsilon[\tilde u_\epsilon]\|_{L^r(\widetilde \Omega_\epsilon)^{3\times 3}}\leq C\epsilon^{-{2\over r}(\gamma-1)} \,,
\end{equation}
\item if $\gamma>1$,
\begin{equation}\label{estimates_u_tilde2greater1}
\|\tilde u_\epsilon\|_{L^r(\widetilde \Omega_\epsilon)^{3}}\leq C\epsilon^{{-{\gamma-1\over r-1}}+1} ,\quad \|D_\epsilon \tilde u_\epsilon\|_{L^r(\widetilde \Omega_\epsilon)^{3\times 3}}\leq C\epsilon^{{-{\gamma-1\over r-1}}} ,\quad \|\mathbb{D}_\epsilon[\tilde u_\epsilon]\|_{L^r(\widetilde \Omega_\epsilon)^{3\times 3}}\leq C\epsilon^{{-{\gamma-1\over r-1}}} \,,
\end{equation}
\item if $\gamma=1$,
\begin{equation}\label{estimates_u_tilde2equal1}
\|\tilde u_\epsilon\|_{L^r(\widetilde \Omega_\epsilon)^{3}}\leq C\epsilon ,\quad \|D_\epsilon \tilde u_\epsilon\|_{L^r(\widetilde \Omega_\epsilon)^{3\times 3}}\leq C ,\quad \|\mathbb{D}_\epsilon[\tilde u_\epsilon]\|_{L^r(\widetilde \Omega_\epsilon)^{3\times 3}}\leq C \,.
\end{equation}
\end{itemize}

\end{itemize}
\end{lemma}
\begin{proof}
 Multiplying (\ref{2}) by $\tilde u_\epsilon$, integrating over $\widetilde\Omega_\epsilon$ and taking into account  that ${\rm div}_\epsilon(\tilde u_\epsilon)=0$ in $\widetilde\Omega_\epsilon$,  we get
\begin{equation}\label{form_var_estim}\epsilon^\gamma (\eta_0-\eta_\infty)\int_{\widetilde\Omega_\epsilon}\left(1+\lambda|\mathbb{D}_{\epsilon}[\tilde u_\epsilon]|^2\right)^{{r\over 2}-1}|\mathbb{D}_\epsilon[\tilde u_\epsilon]|^2dx'dz_3+\epsilon^\gamma \eta_\infty\int_{\widetilde\Omega_\epsilon}|\mathbb{D}_{\epsilon}[\tilde u_\epsilon]|^2dx'dz_3= \int_{\widetilde\Omega_\epsilon}f'\cdot \tilde u_\epsilon'\,dx'dz_3.
\end{equation}
We divide the proof in two steps. First, we derive estimates (\ref{estimates_u_tilde}) for every $r>1$ and then, for $r>2$, we establish estimates (\ref{estimates_u_tilde2less1})-(\ref{estimates_u_tilde2equal1}) depending on the value of $\gamma$.\\

{\it Step 1}. We consider $r>1$. Taking into account that $\eta_0>\eta_{\infty}$, and $\lambda>0$, we have
$$\epsilon^\gamma (\eta_0-\eta_\infty)\int_{\widetilde\Omega_\epsilon}\left(1+\lambda|\mathbb{D}_{\epsilon}[\tilde u_\epsilon]|^2\right)^{{r\over 2}-1}|\mathbb{D}_\epsilon[\tilde u_\epsilon]|^2dx'dz_3\geq 0.
$$
From  Cauchy-Schwarz inequality and the assumption on $f'$ given in (\ref{fassump}), we deduce from (\ref{form_var_estim}) that
$$\epsilon^\gamma\eta_\infty\| \mathbb{D}_\epsilon[\tilde u_\epsilon]\|_{L^2(\widetilde\Omega_\epsilon)^{3\times 3}}^2\leq C \|\tilde u_\epsilon\|_{L^2(\widetilde\Omega_\epsilon)^3}.$$
Applying Poincar\'e  and Korn inequalities \eqref{Poincare}-\eqref{Korn} to the right-hand side, we get (\ref{estimates_u_tilde})$_3$. Finally, applying once again \eqref{Poincare} and \eqref{Korn} yields (\ref{estimates_u_tilde})$_1$ and (\ref{estimates_u_tilde})$_2$.\\

{\it Step 2}. Assume that $r>2$. The idea is now to estimate the first integral in~\eqref{form_var_estim}, starting off by noticing that 
$$\epsilon^\gamma \eta_\infty\int_{\widetilde\Omega_\epsilon}|\mathbb{D}_{\epsilon}[\tilde u_\epsilon]|^2dx'dz_3\geq 0.$$
 Hence,\eqref{form_var_estim} and Cauchy-Schwarz inequality imply
\begin{equation}\label{form_var_estim2}\epsilon^\gamma (\eta_0-\eta_\infty)\int_{\widetilde\Omega_\epsilon}\left(1+\lambda|\mathbb{D}_{\epsilon}[\tilde u_\epsilon]|^2\right)^{{r\over 2}-1}|\mathbb{D}_\epsilon[\tilde u_\epsilon]|^2dx'dz_3\leq  C\|\tilde u_\epsilon\|_{L^2(\widetilde\Omega_\epsilon)^3}.
\end{equation}
Noticing that
$$\epsilon^\gamma \lambda^{r-2\over 2} (\eta_0-\eta_\infty)\int_{\widetilde\Omega_\epsilon}|\mathbb{D}_\epsilon[\tilde u_\epsilon]|^rdx'dz_3\leq \epsilon^\gamma (\eta_0-\eta_\infty)\int_{\widetilde\Omega_\epsilon}\left(1+\lambda|\mathbb{D}_{\epsilon}[\tilde u_\epsilon]|^2\right)^{{r\over 2}-1}|\mathbb{D}_\epsilon[\tilde u_\epsilon]|^2dx'dz_3,$$
and applying Poincar\'e and Korn inequalities (\ref{Poincare})-(\ref{Korn}) to the right-hand side of (\ref{form_var_estim2}), we get
\begin{equation}\label{estim_du}\|\mathbb{D}_\epsilon[\tilde u_\epsilon]\|^r_{L^r(\widetilde\Omega_\epsilon)^{3\times 3}}\leq C\epsilon^{1-\gamma}\|\mathbb{D}_\epsilon[\tilde u_\epsilon]\|_{L^2(\widetilde\Omega_\epsilon)^{3\times 3}}.
\end{equation}
On the one hand, applying estimate (\ref{estimates_u_tilde})$_3$, we deduce
$$\|\mathbb{D}_\epsilon[\tilde u_\epsilon]\|_{L^r(\widetilde\Omega_\epsilon)^{3\times 3}}\leq C\epsilon^{-{2\over r}(\gamma-1)}.$$
On the other hand, from the continuity of the embedding $L^r(\widetilde\Omega_\epsilon)\hookrightarrow L^{2}(\widetilde\Omega_\epsilon)$ in (\ref{estim_du}), we also have
$$\|\mathbb{D}_\epsilon[\tilde u_\epsilon]\|^r_{L^r(\widetilde\Omega_\epsilon)^{3\times 3}}\leq C\epsilon^{1-\gamma}\|\mathbb{D}_\epsilon[\tilde u_\epsilon]\|_{L^r(\widetilde\Omega_\epsilon)^{3\times 3}},$$
which gives
$$\|\mathbb{D}_\epsilon[\tilde u_\epsilon]\|_{L^r(\widetilde\Omega_\epsilon)^{3\times 3}}\leq C\epsilon^{-{\gamma-1\over r-1}}.$$
As a result, we have derived two different estimates of $\mathbb{D}_\epsilon[\tilde u_\epsilon]$ in $L^r(\widetilde\Omega_\epsilon)^{3\times 3}$, that we may now compare in order to obtain the more accurate one, depending on the value of $\gamma$. Since 
$-{2\over r}(\gamma-1)>-{\gamma-1\over r-1}$ if $\gamma<1$ and $-{2\over r}(\gamma-1)<-{\gamma-1\over r-1}$ if $\gamma>1$, we deduce estimates (\ref{estimates_u_tilde2less1})$_3$ and (\ref{estimates_u_tilde2greater1})$_3$. In case $\gamma=1$, both estimates give (\ref{estimates_u_tilde2equal1})$_3$. Finally, from Poincar\'e inequality (\ref{Poincare}) and Korn inequality (\ref{Korn}), we derive the remaining estimates (\ref{estimates_u_tilde2less1}), (\ref{estimates_u_tilde2greater1}) and (\ref{estimates_u_tilde2equal1}).
\end{proof}

\begin{remark}
We extend the velocity $\tilde u_\epsilon$ by zero in $\Omega\setminus\widetilde \Omega_\epsilon$ (this is compatible with the homogeneous boundary condition on $\partial \Omega\cup \partial T_\epsilon$), and   denote the extension by the same symbol. Obviously, estimates given in Lemma \ref{Estimates_lemma} remain valid and the extension $\tilde u_\epsilon$ is divergence free too.
\end{remark}

Recall that $Q_\epsilon=\omega\times (0,\epsilon)$. To extend the pressure $\tilde p_\epsilon$ to the whole domain $\Omega$ and obtain {\it a priori} estimates, we rely on a duality argument and on the existence of restriction operators from $W^{1,q}_0(Q_\epsilon)^3$ into $W^{1,q}_0(\Omega_\epsilon)^3$, introduced in~\cite{Anguiano_SuarezGrau}.

\begin{lemma}[Lemma 4.5-(i) in \cite{Anguiano_SuarezGrau}] \label{restriction_operator}
Let $1<q<+\infty$. There exists  a (restriction) operator $R^\epsilon_q$ mapping $W^{1,q}_0(Q_\epsilon)^3$  to $W^{1,q}_0(\Omega_\epsilon)^3$, $1<q<+\infty$, such that
\begin{enumerate}
\item $R^\epsilon_q \varphi=\varphi$, if $\varphi \in W^{1,q}_0(\Omega_\epsilon)^3$ (elements of $W^{1,q}_0(\Omega_\epsilon)$ are extended by $0$ to $Q_\epsilon$).
\item ${\rm div}R^\epsilon_q \varphi=0\hbox{  in }\Omega_\epsilon$, if ${\rm div}\,\varphi=0\hbox{  in }Q_\epsilon$.
\item  There exists a positive constant $C$, independent of $\epsilon$, such that for every $\varphi\in W^{1,q}_0(Q_\epsilon)^3$,
\begin{equation}\label{estim_restricted}
\begin{array}{l}
\|R^\epsilon_q \varphi\|_{L^q(\Omega_\epsilon)^{3}}+ \epsilon\|D R^\epsilon_q \varphi\|_{L^q(\Omega_\epsilon)^{3\times 3}} \leq C\left(\|\varphi\|_{L^q(Q_\epsilon)^3}+\epsilon \|D \varphi\|_{L^q(Q_\epsilon)^{3\times 3}}\right)\,.
\end{array}
\end{equation}
\end{enumerate}
\end{lemma}

 In the next result, using the restriction operator defined in Lemma~\ref{restriction_operator},  we extend the pressure gradient $\nabla p_\epsilon$ by duality in $W^{-1,q'}(Q_\epsilon)^3$. Then, by means of the dilatation, we extend $\tilde p_\epsilon$ to $\Omega$ and derive the corresponding estimates.

\begin{lemma}\label{Estimates_extended_lemma}  Let $\tilde p_\epsilon$ the pressure solution of (\ref{2}).
\begin{itemize}
\item[$(i)$] (Pseudoplastic and Newtonian fluid.) If $1<r\leq 2$, there exist an extension  $\tilde P_\epsilon \in L^2_0(\Omega)$ of $\tilde p_\epsilon$ and a positive constant $C$, independent of $\epsilon$, such that 
\begin{equation}\label{esti_P}
\|\tilde P_\epsilon\|_{L^2(\Omega)}\leq C\,, \quad \|\nabla_\epsilon \tilde P_\epsilon\|_{H^{-1}(\Omega)^3}\leq C.
\end{equation}
\item[$(ii)$] (Dilatant fluid.) If $r>2$, there exist an extension $\tilde P_\epsilon \in L^{r'}_0(\Omega)$ of $\tilde p_\epsilon$ and a positive constant $C$, independent of $\epsilon$, such that 
\begin{equation}\label{esti_P_12}
\|\tilde P_\epsilon\|_{L^{r'}(\Omega)}\leq C\,, \quad \|\nabla_\epsilon \tilde P_\epsilon\|_{W^{-1,r'}(\Omega)^3}\leq C,
\end{equation}
where $r'$ is the conjugate exponent of $r$.
\end{itemize}
\end{lemma}
\begin{proof}
 We divide the proof in three steps. First, we extend the pressure in all cases (pseudoplastic, Newtonian and dilatant). Then, we obtain the estimates for pseudoplastic and Newtonian fluids, before deriving the estimate for dilatant fluids.\\

{\it Step 1.} {\it Extension of the pressure}.
   Let $q=\max\{2, r\}$ and $q'$ be the conjugate exponent of $q$. 
	 Using the restriction operator $R^\epsilon_q$ given in Lemma~\ref{restriction_operator}, we define the linear functional $F_\epsilon$ on $W^{1,q}_0(Q_\epsilon)^3$ by
	\begin{equation}\label{F}
	F_\epsilon(\varphi)=\langle \nabla p_\epsilon, R^\epsilon_q \varphi\rangle_{{W^{-1,q'}(\Omega_\epsilon)^3, W^{1,q}_0(\Omega_\epsilon)^3}}\,,\quad \hbox{for any }\varphi\in W^{1,q}_0(Q_\epsilon)^3\, .
	\end{equation}
Using the variational formulation of problem (\ref{1}), the right hand side of (\ref{F}) can be rephrased as follows:
\begin{equation}\label{equality_duality}
\begin{array}{rl}
\medskip
\displaystyle
F_{\epsilon}(\varphi)=&\displaystyle
-\epsilon^\gamma(\eta_0-\eta_\infty)\int_{\Omega_\epsilon} (1+\lambda|\mathbb{D}[u_\epsilon]|^2)^{{r\over 2}-1}\mathbb{D}[u_\epsilon]: DR^{\epsilon}_q\varphi\,dx \\
\medskip
&\displaystyle
- \epsilon^\gamma\eta_\infty  \int_{\Omega_\epsilon}   \mathbb{D}[u_\epsilon]: DR^{\epsilon}_q\varphi\,dx+ \int_{\Omega_\epsilon} f'\cdot (R^{\epsilon}_q\varphi)'\,dx \,.
\end{array}\end{equation}
Using Lemma \ref{Estimates_lemma} for fixed $\epsilon$, we see that $F_\epsilon\in W^{-1,q'}(Q_\epsilon)^3$. Moreover, ${\rm div}\, \varphi=0$ implies $F_{\epsilon}(\varphi)=0\,,$ hence De Rham theorem gives the existence of $P_\epsilon$ in $L^{q'}_0(Q_\epsilon)$ such that $F_\epsilon=\nabla P_\epsilon$.

Now, we define $\tilde P_\epsilon\in L^{q'}_0(\Omega)$ by $\tilde P_\epsilon(x',z_3)=P_\epsilon(x',\epsilon z_3)$, and take $\tilde \varphi \in W^{1,q}_0(\Omega)^3$ and the corresponding function $\varphi\in  W^{1,q}_0(Q_\epsilon)^3$ satisfying $\tilde \varphi(x', z_3)=\varphi(x',\epsilon z_3)$.
Using the change of variables (\ref{dilatacion}) and the identification (\ref{equality_duality}) of $F_\epsilon$, we see that
\begin{align}
\langle \nabla_{\epsilon}\tilde P_\epsilon, \tilde \varphi\rangle_{W^{-1,q'}(\Omega)^3, W^{1,q}_0(\Omega)^3}& = -\int_{\Omega}\tilde P_\epsilon\,{\rm div}_{\epsilon}\,\tilde \varphi\,dx'dz_3 \nonumber \\
& = -\epsilon^{-1}\int_{Q_\epsilon}P_\epsilon\,{\rm div}\,\varphi\,dx \nonumber\\
& = \epsilon^{-1}\langle \nabla P_\epsilon, \varphi\rangle_{W^{-1,q'}(Q_\epsilon)^3, W^{1,q}_0(Q_\epsilon)^3}\nonumber\\
& = \epsilon^{-1} F_\epsilon(\varphi)\nonumber\\
& = \epsilon^{-1}\left(
-\epsilon^\gamma(\eta_0-\eta_\infty) \int_{\Omega_\epsilon} (1+\lambda|\mathbb{D}[u_\epsilon]|^2)^{{r\over 2}-1}\mathbb{D}[u_\epsilon]: DR^{\epsilon}_q\varphi\,dx\right.\nonumber\\
&\left.
\quad\quad - \epsilon^\gamma\eta_\infty  \int_{\Omega_\epsilon}   \mathbb{D}[u_\epsilon]: DR^{\epsilon}_q\varphi\,dx+ \int_{\Omega_\epsilon} f'\cdot (R^{\epsilon}_q\varphi)'\,dx\right)\nonumber\\
& = - \epsilon^\gamma(\eta_0-\eta_\infty)\int_{\widetilde \Omega_\epsilon} (1+\lambda|\mathbb{D}_\epsilon[\tilde u_\epsilon]|^2)^{{r\over 2}-1}\mathbb{D}_\epsilon[\tilde u_\epsilon] : D_{\epsilon}\tilde R^{\epsilon}_q\tilde \varphi\,dx'dz_3\nonumber\\
&\quad\quad - \epsilon^\gamma\eta_\infty  \int_{\widetilde\Omega_\epsilon}   \mathbb{D}_\epsilon[\tilde u_\epsilon]: D_\epsilon\tilde R^{\epsilon}_q\tilde \varphi\,dx'dz_3+\int_{\widetilde \Omega_\epsilon} f'(x')\cdot (\tilde R^{\epsilon}_q \tilde \varphi)'\,dx'dz_3\,,\label{extension_1}
\end{align}
where $\tilde R^{\epsilon}_q$ is defined by
$
(\tilde R^{\epsilon}_q \tilde \varphi)(x',z_3) = (R^{\epsilon}_q \varphi)(x',\epsilon z_3)
$.\\

{\it Step 2}. {\it  Estimates of the extended pressure for pseudoplastic fluids and Newtonian fluids}.  Applying the dilatation in (\ref{estim_restricted}) for $q=2$, we have that $\tilde R^\epsilon_2\tilde \varphi$ satisfies the following estimate
\begin{equation}\label{ext_1}\begin{array}{l}
\medskip\displaystyle
\|\tilde R^\epsilon_2\tilde \varphi\|_{L^2(\widetilde\Omega_\epsilon)^3}+ \epsilon \|D_{\epsilon}\tilde R^\epsilon_2\tilde \varphi\|_{L^2(\widetilde\Omega_\epsilon)^{3\times 3}}\leq  C\left(\|\tilde \varphi\|_{L^2(\Omega)^3} 
+ \epsilon\|D_{\epsilon}\tilde \varphi\|_{L^2(\Omega)^{3\times 3}}\right),
\end{array} 
\end{equation}
and since $\epsilon\ll 1$, we deduce 
\begin{equation}\label{ext_2}
\|\tilde R^\epsilon_2 \tilde \varphi\|_{L^2(\widetilde\Omega_\epsilon)^3}\leq  C \|\tilde \varphi\|_{H^1_0(\Omega)^3},\quad \|D_\epsilon \tilde R^\epsilon_2\tilde \varphi\|_{L^2(\widetilde\Omega_\epsilon)^{3\times 3}}\leq {C\over \epsilon}\|\tilde \varphi\|_{H^1_0(\Omega)^3}.
\end{equation}
Taking into account that $1<r\leq 2$, we notice that 
$$(1+\lambda|\mathbb{D}_\epsilon[\tilde u_\epsilon]|^2)^{{r\over 2}-1}\leq 1,$$ 
so that Cauchy-Schwarz inequality yields
\begin{align*}
\left|\int_{\widetilde \Omega_\epsilon}(1+\lambda|\mathbb{D}_\epsilon[\tilde u_\epsilon]|^2)^{{r\over 2}-1}\mathbb{D}_\epsilon[\tilde u_\epsilon] : D_{\epsilon}\tilde R^{\epsilon}_2 \tilde \varphi\,dx'dz_3\right| & \leq
\int_{\widetilde \Omega_\epsilon}|\mathbb{D}_\epsilon[\tilde u_\epsilon]||D_{\epsilon}\tilde R^{\epsilon}_2 \tilde \varphi|\,dx'dz_3\\
& \leq 
 \|\mathbb{D}_\epsilon[\tilde u_\epsilon]\|_{L^2(\widetilde\Omega_\epsilon)^{3\times 3}}\|D_{\epsilon}\tilde R^{\epsilon}_2 \tilde \varphi\|_{L^2(\widetilde\Omega_\epsilon)^{3\times 3}}.
\end{align*}
Using last estimate in (\ref{estimates_u_tilde}) and last estimate on the dilated restricted operator given in (\ref{ext_2}),  we  obtain
\begin{equation}\label{estimate_Carreau1}
\begin{array}{rl}
\medskip
\displaystyle
\left|\epsilon^{\gamma}(\eta_0-\eta_\infty)\int_{\widetilde \Omega_\epsilon}(1+\lambda|\mathbb{D}_\epsilon[\tilde u_\epsilon]|^2)^{{r\over 2}-1}\mathbb{D}_\epsilon[\tilde u_\epsilon] : D_{\epsilon}\tilde R^{\epsilon}_2 \tilde \varphi\,dx'dz_3\right|& \displaystyle \leq C \|\tilde \varphi\|_{H^1_0(\Omega)^3}.
\end{array}
\end{equation}
and
\begin{equation}\label{estimate_Carreau2}
\begin{array}{rl}
\medskip
\displaystyle
\left|\epsilon^\gamma\eta_\infty  \int_{\widetilde\Omega_\epsilon}   \mathbb{D}_\epsilon[\tilde u_\epsilon]: D_\epsilon\tilde R^{\epsilon}_2\tilde \varphi\,dx'dz_3\right| &\displaystyle \leq C\epsilon^\gamma \|\mathbb{D}_{\epsilon}[\tilde u_\epsilon]\|_{L^2(\widetilde\Omega_\epsilon)^{3\times 3}}\|D_{\epsilon}\tilde R^\epsilon_2\tilde \varphi\|_{L^2(\widetilde \Omega_\epsilon)^{3\times 3}}\leq C \|\tilde \varphi\|_{H^1_0(\Omega)^3}.
\end{array}
\end{equation}
Since $f'=f'(x')$ is in $L^{\infty}(\omega)$, we also get by the first estimate in~(\ref{ext_2}) that
\begin{equation}\label{estimate_Carreau3}
\begin{array}{rl}
\medskip
\displaystyle
 \left|\int_{\widetilde\Omega_\epsilon}f'\cdot (\tilde R^\epsilon_2 \tilde \varphi)' \,dx'dz_3\right|&\leq C\|\tilde R^\epsilon_2 \tilde v \|_{L^2(\widetilde\Omega_\epsilon)^3}\leq C\|\tilde \varphi\|_{H^1_0(\Omega)^3}\,.
\end{array}
\end{equation}
Coming back to the expression~(\ref{extension_1}) of $\langle \nabla_{\epsilon}\tilde P_\epsilon, \tilde \varphi\rangle$, we deduce from (\ref{estimate_Carreau1})--(\ref{estimate_Carreau3})
the second estimate in (\ref{esti_P}). Finally, by Ne${\check{\rm c}}$as inequality, there exists a representative $\tilde P_\epsilon\in L^2_0(\Omega)$ such that
$$\|\tilde P_\epsilon\|_{L^2(\Omega)}\leq C\|\nabla\tilde P_\epsilon\|_{H^{-1}(\Omega)^3}\leq C\|\nabla_{\epsilon}\tilde P_\epsilon\|_{H^{-1}(\Omega)^3},$$
which implies the first estimate in (\ref{esti_P}).

{\it Step 3}. {\it  Estimates of  the extended pressure for dilatant fluids}. Applying the dilatation in (\ref{estim_restricted}) for $q=r$, we get that $\tilde R^\epsilon_r\tilde \varphi$ satisfies the following estimate:
\begin{equation}\label{ext_12}\begin{array}{l}
\medskip\displaystyle
\|\tilde R^\epsilon_r\tilde \varphi\|_{L^r(\widetilde\Omega_\epsilon)^3}+ \epsilon \|D_{\epsilon}\tilde R^\epsilon_r\tilde \varphi\|_{L^r(\widetilde\Omega_\epsilon)^{3\times 3}}\leq  C\left(\|\tilde \varphi\|_{L^r(\Omega)^3} 
+ \epsilon\|D_{\epsilon}\tilde \varphi\|_{L^r(\Omega)^{3\times 3}}\right),
\end{array} 
\end{equation}
and since $\epsilon\ll 1$, this yields 
\begin{equation}\label{ext_22}
\|\tilde R^\epsilon_r \tilde \varphi\|_{L^r(\widetilde\Omega_\epsilon)^3}\leq  C \|\tilde \varphi\|_{W^{1,r}_0(\Omega)^3},\quad \|D_\epsilon \tilde R^\epsilon_r\tilde \varphi\|_{L^r(\widetilde\Omega_\epsilon)^{3\times 3}}\leq {C\over \epsilon}\|\tilde \varphi\|_{W^{1,r}_0(\Omega)^3}.
\end{equation}
Since $r>2$, the  embedding $L^r(\widetilde\Omega_\epsilon) \hookrightarrow L^2(\widetilde\Omega_\epsilon)$ is continuous, so we can deduce from H${\rm \ddot{o}}$lder inequality and from the inequality $(1+X)^\alpha\leq C(1+X^\alpha)$ that holds true for $X\geq 0, \alpha>0$:
\[
\begin{array}{rl}
\medskip
&\displaystyle
\int_{\widetilde \Omega_\epsilon}\left|(1+\lambda|\mathbb{D}_\epsilon[\tilde u_\epsilon]|^2)^{{r\over 2}-1}\mathbb{D}_\epsilon[\tilde u_\epsilon] : D_{\epsilon}\tilde R^{\epsilon}_r \tilde \varphi\right|\,dx'dz_3 \\
\medskip
&\displaystyle \leq C \left( \int_{\widetilde \Omega_\epsilon}
|\mathbb{D}_\epsilon[\tilde u_\epsilon]|\, |D_{\epsilon}\tilde R^{\epsilon}_r \tilde \varphi|\,dx'dz_3 + \int_{\widetilde \Omega_\epsilon}|\mathbb{D}_\epsilon[\tilde u_\epsilon]|^{r-1}|D_{\epsilon}\tilde R^{\epsilon}_r \tilde \varphi|\,dx'dz_3 \right)\\
\medskip
& \displaystyle \leq C \left(\|\mathbb{D}_{\epsilon}[\tilde u_\epsilon]\|_{L^{2}(\widetilde\Omega_\epsilon)^{3\times 3}}\|D_{\epsilon}\tilde R^\epsilon_r\tilde \varphi\|_{L^2(\widetilde \Omega_\epsilon)^{3\times 3}}+\|\mathbb{D}_{\epsilon}[\tilde u_\epsilon]\|^{r-1}_{L^{r}(\widetilde\Omega_\epsilon)^{3\times 3}}\|D_{\epsilon}\tilde R^\epsilon_r\tilde \varphi\|_{L^r(\widetilde \Omega_\epsilon)^{3\times 3}} \right)
\\
\medskip
& \displaystyle \leq C \left(\|\mathbb{D}_{\epsilon}[\tilde u_\epsilon]\|_{L^{2}(\widetilde\Omega_\epsilon)^{3\times 3}}+\|\mathbb{D}_{\epsilon}[\tilde u_\epsilon]\|^{r-1}_{L^{r}(\widetilde\Omega_\epsilon)^{3\times 3}} \right)\|D_{\epsilon}\tilde R^\epsilon_r\tilde \varphi\|_{L^r(\widetilde \Omega_\epsilon)^{3\times 3}}.
\\
\end{array}
\]

Observe that if $\gamma<1$, taking into account that $-{2\over r}(\gamma-1)>-{\gamma-1\over (r-1)}$, using the last estimates in (\ref{estimates_u_tilde}) and (\ref{estimates_u_tilde2less1}), and the last estimate of the dilated restricted operator given in (\ref{ext_22}), we obtain 

$$\begin{array}{l}
\displaystyle
\left|\epsilon^\gamma(\eta_0-\eta_\infty)\int_{\widetilde \Omega_\epsilon}(1+\lambda|\mathbb{D}_\epsilon[\tilde u_\epsilon]|^2)^{{r\over 2}-1}\mathbb{D}_\epsilon[\tilde u_\epsilon] : D_{\epsilon}\tilde R^{\epsilon}_r \tilde \varphi\,dx'dz_3\right|\\
\\
\displaystyle
\leq C \epsilon^\gamma(\eta_0-\eta_\infty)\left (\epsilon^{1-\gamma}+\epsilon^{-{2\over r}(\gamma-1)(r-1)}\right)\epsilon^{-1}\|\tilde \varphi\|_{W^{1,r}_0(\widetilde\Omega_\epsilon)^3}\\
\\
\displaystyle
\leq C	\epsilon^\gamma(\eta_0-\eta_\infty)\left(\epsilon^{1-\gamma}+\epsilon^{-{\gamma-1\over r-1}(r-1)}\right) \epsilon^{-1} \|\tilde \varphi\|_{W^{1,r}_0(\widetilde\Omega_\epsilon)^3}\\
\\
	\displaystyle
	\leq C\|\tilde \varphi\|_{W^{1,r}_0(\widetilde\Omega_\epsilon)^3}.
\end{array}
$$
If $\gamma\geq 1$, by the last estimate in (\ref{estimates_u_tilde2greater1}) and (\ref{estimates_u_tilde2equal1}), $\|\mathbb{D}_{\epsilon}[\tilde u_\epsilon]\|_{L^{r}(\widetilde\Omega_\epsilon)^{3\times 3}}\|\leq C\epsilon^{-\frac{\gamma-1}{r-1}}$, so a similar argument proves that the estimate
\begin{equation}\label{estimate_Carreau12}
\begin{array}{rl}
\medskip
\displaystyle
\left|\epsilon^\gamma(\eta_0-\eta_\infty)\int_{\widetilde \Omega_\epsilon}(1+\lambda|\mathbb{D}_\epsilon[\tilde u_\epsilon]|^2)^{{r\over 2}-1}\mathbb{D}_\epsilon[\tilde u_\epsilon] : D_{\epsilon}\tilde R^{\epsilon}_r \tilde \varphi\,dx'dz_3\right|& \displaystyle \leq C \|\tilde \varphi\|_{W^{1,r}_0(\Omega)^3}
\end{array}
\end{equation}
remains valid for any $\gamma\in \R$.

Moreover, from Cauchy-Schwarz inequality, last estimate in (\ref{estimates_u_tilde}), the continuous embedding $L^r(\widetilde\Omega_\epsilon)\hookrightarrow L^2(\widetilde\Omega_\epsilon)$, the assumption on $f'$ given in (\ref{fassump}) and estimates (\ref{ext_22}),  we deduce the upper bounds
\begin{equation}\label{estimate_Carreau22}
\begin{array}{rl}
\medskip
\displaystyle
\left|\epsilon^\gamma \eta_\infty  \int_{\widetilde\Omega_\epsilon}   \mathbb{D}_\epsilon[\tilde u_\epsilon]: D_\epsilon\tilde R^{\epsilon}_r\tilde \varphi\,dx'dz_3\right| &\displaystyle \leq C\epsilon^\gamma\|\mathbb{D}_{\epsilon}[\tilde u_\epsilon]\|_{L^2(\widetilde\Omega_\epsilon)^{3\times 3}}\|D_{\epsilon}\tilde R^\epsilon_r\tilde \varphi\|_{L^r(\widetilde \Omega_\epsilon)^{3\times 3}}\leq C \|\tilde \varphi\|_{W^{1,r}_0(\Omega)^3},
\end{array}
\end{equation}
\begin{equation}\label{estimate_Carreau32}
\begin{array}{rl}
\medskip
\displaystyle
 \left|\int_{\widetilde\Omega_\epsilon}f'\cdot (\tilde R^\epsilon_r \tilde \varphi)' \,dx'dz_3\right|&\leq C\|\tilde R^\epsilon_r \tilde \varphi \|_{L^r(\widetilde\Omega_\epsilon)^3}\leq C\|\tilde \varphi\|_{W^{1,r}_0(\Omega)^3}\,.
\end{array}
\end{equation}
Taking into account the above estimates (\ref{estimate_Carreau12})--(\ref{estimate_Carreau32}), the relation (\ref{extension_1}) (with $q=r$) yields
$$\left|\langle \nabla_{\epsilon}\tilde P_\epsilon, \tilde \varphi\rangle_{W^{-1,r'}(\Omega)^3, W^{1,r}_0(\Omega)^3}\right|\leq C\|\tilde \varphi\|_{W^{1,r}_0(\Omega)^3}.$$
This implies the second estimate in (\ref{esti_P_12}) and, by Ne${\breve{\rm c}}$as inequality, the existence of a representative $\tilde P_\epsilon\in L^{r'}_0(\Omega)$ such that
$$\|\tilde P_\epsilon\|_{L^{r'}(\Omega)}\leq C\|\nabla\tilde P_\epsilon\|_{W^{-1,r'}(\Omega)^3}\leq C\|\nabla_{\epsilon}\tilde P_\epsilon\|_{W^{-1,r'}(\Omega)^3},$$
which provides the first estimate in (\ref{esti_P_12}).

\end{proof}

\subsection{Adaptation of the unfolding method}\label{sec:unfolding}
The change of variables (\ref{dilatacion}) does not provide the information we need about the behaviour of $\tilde u_\epsilon$ in the microstructure associated to $\widetilde\Omega_\epsilon$. To solve this difficulty, we use an adaptation introduced in \cite{Anguiano_SuarezGrau} of the unfolding method from \cite{CDG}.

Let us recall that this adaptation of the unfolding method divides the domain $\widetilde\Omega_\epsilon$ in cubes of lateral length $\epsilon$ and vertical length $1$. Thus, given $(\tilde{\varphi}_{\epsilon},  \tilde \psi_\epsilon) \in L^q(\Omega)^3\times L^{q'}(\Omega)$, $1<q<+\infty$ and $1/q+1/q'=1$, we define $(\hat{\varphi}_{\epsilon},  \hat \psi_\epsilon)\in L^q(\omega\times Z)^3\times L^{q'}(\omega\times Z)$ by
\begin{equation}\label{phihat}
\hat{\varphi}_{\epsilon}(x^{\prime},z)=\tilde{\varphi}_{\epsilon}\left( {\epsilon}\kappa\left(\frac{x^{\prime}}{{\epsilon}} \right)+{\epsilon}z^{\prime},z_3 \right),\quad 
\hat{\psi}_{\epsilon}(x^{\prime},z)=\tilde{\psi}_{\epsilon}\left( {\epsilon}\kappa\left(\frac{x^{\prime}}{{\epsilon}} \right)+{\epsilon}z^{\prime},z_3 \right),\quad \hbox{ a.e. }(x^{\prime},z)\in \omega\times Z,
\end{equation}
assuming $\tilde \varphi_\epsilon$ and $\tilde \psi_\epsilon$ are extended by zero outside $\omega$, where the function $\kappa:\mathbb{R}^2\to \mathbb{Z}^2$ is defined by 
$$\kappa(x')=k'\Longleftrightarrow x'\in Z'_{k',1},\quad\forall\,k'\in\mathbb{Z}^2.$$

\begin{remark}\label{remarkCV}We make the following comments:
\begin{itemize}
\item[-] The function $\kappa$ is well defined up to a set of zero measure in $\mathbb{R}^2$ (the set $\cup_{k'\in \mathbb{Z}^2}\partial Z'_{k',1}$). Moreover, for every $\epsilon>0$, we have
$$\kappa\left({x'\over \epsilon}\right)=k'\Longleftrightarrow x'\in Z'_{k',\epsilon}.$$

\item[-]For $k^{\prime}\in \mathcal{K}_{\epsilon}$, the restriction of $(\hat{u}_{\epsilon},   \hat P_\epsilon)$  to $Z^{\prime}_{k^{\prime},{\epsilon}}\times Z$ does not depend on $x^{\prime}$, whereas as a function of $z$ it is obtained from $(\tilde{u}_{\epsilon},  \tilde{P}_{\epsilon})$ by using the change of variables 
\begin{equation}\label{CVunfolding}
\displaystyle z^{\prime}=\frac{x^{\prime}- {\epsilon}k^{\prime}}{{\epsilon}},\end{equation}
which transforms $Z_{k^{\prime}, {\epsilon}}$ into $Z$.
\end{itemize}
\end{remark}

Following the  proof of \cite[Lemma 4.9]{Anguiano_SuarezGrau}, the following estimates relate  $(\hat \varphi_\epsilon,\hat \psi_\epsilon)$ to $(\tilde \varphi_\epsilon, \tilde \psi_\epsilon)$.
\begin{lemma}\label{estimates_relation}
We have the following estimates:
\begin{itemize}
\item[$(i)$] For every $\tilde\varphi_\epsilon \in L^q(\widetilde\Omega_\epsilon)^3$, $1\leq q<+\infty$,
$$\|\hat \varphi_\epsilon\|_{L^q(\omega\times Z)^3}\leq  \|\tilde \varphi_\epsilon\|_{L^q(\Omega)^3},$$
where $\hat \varphi_\epsilon$ is given by (\ref{phihat})$_1$. Similarly, for every $\tilde\psi \in L^{q'}(\widetilde\Omega_\epsilon)$, the function $\hat \psi_\epsilon$, given by (\ref{phihat})$_2$ satisfies
$$\|\hat \psi_\epsilon\|_{L^q(\omega\times Z)}\leq  \|\tilde \psi_\epsilon\|_{L^q(\Omega)}.$$
\item[$(ii)$] For every $\tilde \varphi\in W^{1,q}(\widetilde\Omega_\epsilon)^3$, $1\leq q<+\infty$, the function $\hat \varphi_\epsilon$ given by (\ref{phihat})$_1$ belongs to $L^q(\omega;W^{1,q}(Z)^3)$, and
$$\|D_{z'} \hat \varphi_\epsilon\|_{L^q(\omega\times Z)^{3\times 2}}\leq \epsilon  \|D_{x'}\tilde \varphi_\epsilon\|_{L^q(\Omega)^{3\times 2}},\quad \|\partial_{z_3} \hat \varphi_\epsilon\|_{L^q(\omega\times Z)^{3 }}\leq  \|\partial_{z_3}\tilde \varphi_\epsilon\|_{L^q(\Omega)^{3}},$$
$$ \|\mathbb{D}_{z'}[\hat \varphi_\epsilon]\|_{L^q(\omega\times Z)^{3\times 2}}\leq \epsilon  \|\mathbb{D}_{x'}[\tilde \varphi_\epsilon]\|_{L^q(\Omega)^{3\times 2}},\quad \|\partial_{z_3}[\hat \varphi_\epsilon]\|_{L^q(\omega\times Z)^{3 }}\leq \|\partial_{z_3}[\tilde \varphi_\epsilon]\|_{L^q(\Omega)^{3}}.$$
\end{itemize} 
\end{lemma}

\begin{definition}[Unfolded velocity and pressure] Let us define the unfolded velocity and pressure $(\hat u_\epsilon, \hat P_\epsilon)$ from $(\tilde u_\epsilon, \tilde P_\epsilon)$ depending on the type of fluid:
\begin{itemize}
\item[--] ({\it Pseudoplastic fluids and Newtonian fluids.}) From $(\tilde u_\epsilon, \tilde P_\epsilon)\in H^1_0(\Omega)^3\times L^2_0(\Omega)$, we define $(\hat u_\epsilon, \hat P_\epsilon)$ by (\ref{phihat}) with $\tilde \varphi_\epsilon=\tilde u_\epsilon$, $\tilde \psi_\epsilon=\tilde P_\epsilon$ and $q=2$.\\

\item[--] ({\it Dilatant fluids.}) From $(\tilde u_\epsilon, \tilde P_\epsilon)\in W^{1,r}_0(\Omega)^3\times L^{r'}_0(\Omega)$, we define $(\hat u_\epsilon, \hat P_\epsilon)$ by (\ref{phihat}) with $\tilde \varphi_\epsilon=\tilde u_\epsilon$, $\tilde \psi_\epsilon=\tilde P_\epsilon$ and $q=r$.
\end{itemize}
\end{definition}
Now, combining  estimates on the extended velocity (\ref{estimates_u_tilde})-(\ref{estimates_u_tilde2equal1}) and pressure (\ref{esti_P})-(\ref{esti_P_12}) with  Lemma~\ref{estimates_relation}, we deduce the following estimates on $(\hat u_\epsilon,\hat P_\epsilon)$.

\begin{lemma}\label{estimates_hat} The unfolded velocity/pressure pair $(\hat u_\epsilon,\hat P_\epsilon)$ satisfies the following estimates, depending on the type of fluid.
 \begin{itemize}
 \item[$(i)$] {\it (Pseudoplastic fluids and Newtonian fluids.)} Consider $1<r\leq 2$. There exists a constant $C>0$ independent of $\epsilon$, such that, for every value of $\gamma $,
 \begin{eqnarray}\medskip
 \|\hat u_\epsilon\|_{L^2(\omega\times Z)^3}\leq C\epsilon^{2-\gamma},& 
 \|D_{z}\hat u_\epsilon\|_{L^2(\omega\times Z)^{3\times 3}}\leq C\epsilon^{2-\gamma},& 
  \|\mathbb{D}_{z}[\hat u_\epsilon]\|_{L^2(\omega\times Z)^{3\times 3}}\leq C\epsilon^{2-\gamma},\label{estim_u_hat}\\
 \medskip
& \|\hat P_\epsilon\|_{L^2(\omega\times Z)}\leq C.&\label{estim_P_hat}
    \end{eqnarray}
 \item[$(ii)$] ({\it Dilatant fluids.}) Consider $r>2$. There exists a constant $C>0$ independent of $\epsilon$, such that  estimates (\ref{estim_u_hat}) hold true, and also, depending on the value of $\gamma$, we have:
 \begin{itemize}
 \item[--] If $\gamma<1$,
 \begin{equation}\label{estim_u_hat2}
 \begin{array}{c}
 \|\hat u_\epsilon\|_{L^r(\omega\times Z)^3}\leq C\epsilon^{-{2\over r}(\gamma-1)+1},\quad  
 \|D_{z}\hat u_\epsilon\|_{L^r(\omega\times Z)^{3\times 3}}\leq C\epsilon^{-{2\over r}(\gamma-1)+1},\\
\\
  \|\mathbb{D}_{z}[\hat u_\epsilon]\|_{L^r(\omega\times Z)^{3\times 3}}\leq C\epsilon^{-{2\over r}(\gamma-1)+1}.
 \end{array} \end{equation}
 \item[--] If $\gamma>1$,
 \begin{equation}\label{estim_u_hat22}
 \begin{array}{c}
 \|\hat u_\epsilon\|_{L^r(\omega\times Z)^3}\leq C\epsilon^{{-{\gamma-1\over r-1}}+1},\quad 
 \|D_{z}\hat u_\epsilon\|_{L^r(\omega\times Z)^{3\times 3}}\leq C\epsilon^{{-{\gamma-1\over r-1}}+1},\\
 \\
  \|\mathbb{D}_{z}[\hat u_\epsilon]\|_{L^r(\omega\times Z)^{3\times 3}}\leq C\epsilon^{{-{\gamma-1\over r-1}}+1}.
  \end{array}
  \end{equation}
  \item[--] If $\gamma=1$,
 \begin{equation}\label{estim_u_hat22_1}
 \begin{array}{c}
 \|\hat u_\epsilon\|_{L^r(\omega\times Z)^3}\leq C\epsilon,\quad 
 \|D_{z}\hat u_\epsilon\|_{L^r(\omega\times Z)^{3\times 3}}\leq C\epsilon,\\
 \\
  \|\mathbb{D}_{z}[\hat u_\epsilon]\|_{L^r(\omega\times Z)^{3\times 3}}\leq C\epsilon.
  \end{array}
  \end{equation}

 \end{itemize}
 Moreover, the pressure satisfies
  \begin{eqnarray}\medskip
  \|\hat P_\epsilon\|_{L^{r'}(\omega\times Z)}\leq C,&\label{estim_P_hat2}
    \end{eqnarray}
    where $r'$ is the conjugate exponent of $r$.
 \end{itemize}

 \end{lemma}
\subsection{Compactness results.}\label{sec:compactness}  In this section, we analyze the asymptotic behaviour of extended functions $(\tilde u_\epsilon, \tilde P_\epsilon)$ and the corresponding unfolded functions $(\hat u_\epsilon,\hat P_\epsilon)$, when $\epsilon$ tends to zero.
 
\begin{lemma} \label{lemma_compactness} The velocities $\tilde u_\epsilon, \hat u_\epsilon$ satisfy the following convergence results.
\begin{itemize}
\item[$(i)$] (Pseudoplastic fluids and Newtonian fluids.) Consider $1<r\leq 2$, then there exist
\begin{itemize}
\item  $\tilde u\in H^1(0,1;L^2(\omega)^3)$ where  $\tilde u=0$ on $\omega\times \{0, 1\}$ and $\tilde u_3\equiv 0$,  
such that, up to a subsequence,
\begin{equation}\label{conv_vel_tilde}
\epsilon^{\gamma-2}\tilde u_\epsilon \rightharpoonup (\tilde u',0)\hbox{ weakly in } H^1(0,1;L^2(\omega)^3),
\end{equation}

\item  $\hat u\in L^2(\omega; H^1_{0,\#}(Z)^3)$, with $\hat u=0$ on $\omega\times Z'\times \{0, 1\}$,  such that, up to a subsequence,
\begin{equation}\label{conv_vel_gorro}
\epsilon^{\gamma-2}\hat u_\epsilon\rightharpoonup \hat u\hbox{ weakly in }L^2(\omega; H^1(Z)^3).\end{equation}
\end{itemize}
In addition, the following relation holds between $\tilde u$ and $\hat u$:
\begin{equation}\label{equality_integrals}
\tilde u(x',z_3)=\int_{Z'}\hat u(x',z)\,dz'\quad\hbox{with}\quad \int_{Z'}\hat u_3(x',z)\,dz'=0,
\end{equation} 
and so,
\begin{equation}\label{equality_integrals2}
\int_0^1\tilde u(x',z_3)dz_3=\int_{Z}\hat u(x',z)\,dz\quad\hbox{with}\quad \int_{Z}\hat u_3(x',z)\,dz=0.
\end{equation}
\item[$(ii)$] (Dilatant fluids.) Consider $r>2$, then
\begin{itemize}

\item[$\bullet$]   if $\gamma<1$, there exist   
\begin{itemize}
\item[--]  $\tilde u\in H^1(0,1;L^2(\omega)^3)$ where  $\tilde u=0$ on $\omega\times \{0, 1\}$ and $\tilde u_3\equiv 0$,  
such that, up to a subsequence,
\begin{equation}\label{conv_vel_tilde2}
\epsilon^{\gamma-2}\tilde u_\epsilon \rightharpoonup \tilde u=(\tilde u',0)\hbox{ weakly in } H^1(0,1;L^2(\omega)^3),
\end{equation}

\item[--] $\hat u\in L^2(\omega; H^1_{0,\#}(Z)^3)$, with $\hat u=0$ on $\omega\times Z'\times \{0, 1\}$,  such that, up to a subsequence,

\begin{equation}\label{conv_vel_gorro2}
\epsilon^{\gamma-2}\hat u_\epsilon\rightharpoonup \hat u \hbox{ weakly in }L^2(\omega; H^1(Z)^3),\end{equation}
\end{itemize}
where the relations between $\tilde u$ and $\hat u$ given by (\ref{equality_integrals}) and (\ref{equality_integrals2}) hold;

\item[$\bullet$]  if $\gamma\geq 1$, there exist  
\begin{itemize}
\item[--] $\tilde u\in W^{1,r}(0,1;L^r(\omega)^3)$ where  $\tilde u=0$ on $\omega\times \{0, 1\}$ and $\tilde u_3\equiv 0$,  such that, up to a subsequence, in the case $\gamma>1$, 
\begin{equation}\label{conv_vel_tilde22}
\epsilon^{{\gamma-r\over r-1}}\tilde u_\epsilon \rightharpoonup \tilde u=(\tilde u',0)\hbox{ weakly in } W^{1,r}(0,1;L^r(\omega)^3),
\end{equation}
and in the case $\gamma=1$, 
\begin{equation}\label{conv_vel_tilde22_1}
\epsilon^{-1}\tilde u_\epsilon \rightharpoonup (\tilde u',0)\hbox{ weakly in } W^{1,r}(0,1;L^r(\omega)^3),
\end{equation}

\item[--]  $\hat u\in L^r(\omega; W^{1,r}_{0,\#}(Z)^3)$, with $\hat u=0$ on $\omega\times Z'\times \{0, 1\}$, such that, up to a subsequence, in the case $\gamma>1$,  
\begin{equation}\label{conv_vel_gorro22}
\epsilon^{{\gamma-r\over r-1}}\hat u_\epsilon\rightharpoonup \hat u\hbox{ weakly in }L^r(\omega; W^{1,r}(Z)^3),\end{equation}
and in the case $\gamma=1$, 
\begin{equation}\label{conv_vel_gorro22_1}
\epsilon^{-1}\hat u_\epsilon\rightharpoonup \hat u\hbox{ weakly in }L^r(\omega; W^{1,r}(Z)^3),\end{equation}
\end{itemize}

where the relation between $\tilde u$ and $\hat u$ given by (\ref{equality_integrals}) and (\ref{equality_integrals2}) hold.
\end{itemize}
\end{itemize}
Moreover, $\tilde u$ and $\hat u$ satisfy the following divergence conditions for any $r>1$:
\begin{equation}\label{divxproperty}
{\rm div}_{x'}\left(\int_0^1\tilde u'(x',z_3)\,dz_3\right)=0\  \hbox{ in }\omega,\quad \left(\int_0^1\tilde u'(x',z_3)\,dz_3\right)\cdot n=0\  \hbox{ in }\partial \omega,
\end{equation}
\begin{equation}\label{divyproperty}
 {\rm div}_{z}\,\hat u(x',z)=0\  \hbox{ in }\omega\times Z_f,\quad   {\rm div}_{x'}\left(\int_{Z_f}\hat u'(x',z)\,dz\right)=0\  \hbox{ in }\omega, \quad\left(\int_{Z_f}\hat u'(x',z)\,dz\right)\cdot n=0\  \hbox{ on }\partial\omega.
\end{equation}
\end{lemma}

\begin{proof} The proof is based on compactness results given in \cite{Anguiano_SuarezGrau}: 
\begin{itemize}
\item[$(i)$] Arguing as in \cite[Lemma 5.2.-$(i)$]{Anguiano_SuarezGrau}, we obtain convergence (\ref{conv_vel_tilde}) and divergence condition (\ref{divxproperty}). Moreover, proceeding similarly as in  \cite[Lemma 5.4.-$(i)$]{Anguiano_SuarezGrau} we deduce convergence (\ref{conv_vel_gorro}), properties (\ref{equality_integrals}) and (\ref{equality_integrals2}), and divergence conditions (\ref{divyproperty}).

\item[$(ii)$] The proof also follows the lines of $(i)$, just taking into account the estimates of $\tilde u_\epsilon$ and $\hat u_\epsilon$ in each case. Nevertheless, for $r>2$, the choice of the relevant estimates depends on the value of $\gamma$.
\begin{itemize}
\item[-]In case $\gamma<1$, by Lemma \ref{Estimates_lemma}-$(ii)$,  the velocity satisfies two types of estimates: estimates (\ref{estimates_u_tilde}) in $L^2(\widetilde\Omega_\epsilon)$ and (\ref{estimates_u_tilde2less1}) in $L^r(\widetilde\Omega_\epsilon)$. Noticing that $2-\gamma> -{2\over r}(\gamma-1)$, estimates (\ref{estimates_u_tilde}) are in fact the optimal ones, so we proceed as in $(i)$ to finish the proof.

\item[-]In case $\gamma>1$,  if we used the  $L^2$-estimates given in Lemmas \ref{Estimates_lemma}-$(ii)$ and \ref{estimates_hat}-$(ii)$, we would obtain convergences (\ref{conv_vel_tilde2}) and (\ref{conv_vel_gorro2}), respectively. However, we would not be able to pass to the limit in the formulation (\ref{v_ineq_carreau_12}) because the term $(1+\lambda\epsilon^{2(1-\gamma)}|\mathbb{D}_y[\varphi]|^2)^{{r\over 2}-1}$ diverges. For that reason, we apply the $L^r$-estimates of the velocities $\tilde u_\epsilon$ and $\hat u_\epsilon$  given in Lemmas \ref{Estimates_lemma}-$(ii)$ and \ref{estimates_hat}-$(ii)$ respectively. We argue as in $(i)$ to conclude.

\item[-]In case $\gamma=1$, we proceed as in case $\gamma>1$ by considering the $L^r$-estimates of the velocities $\tilde u_\epsilon$ and $\hat u_\epsilon$ given in Lemmas \ref{Estimates_lemma}-$(ii)$ and \ref{estimates_hat}-$(ii)$ respectively.
\end{itemize}
\end{itemize}

\end{proof}

\begin{lemma}\label{lemma_compactness_p}   The extended pressure $\tilde P_\epsilon$ and the corresponding unfolding function $\hat P_\epsilon$ satisfy the following convergence results.
\begin{itemize}
\item[$(i)$] (Pseudoplastic fluids and Newtonian fluids.) Consider $1<r\leq 2$. There exists $\tilde P\in L^2_0(\omega)$ such that
\begin{equation}\label{conv_pres_tilde}
\tilde P_\epsilon\to \tilde P \hbox{ strongly in }L^2(\Omega),
\end{equation}
\begin{equation}\label{conv_pres_hat}
\hat P_\epsilon\to \tilde P \hbox{ strongly in }L^2(\omega\times Z).
\end{equation}
\item[$(ii)$] (Dilatant fluids.) Consider $r>2$. There exists $\tilde P\in L^{r'}_0(\omega)$ such that
\begin{equation}\label{conv_pres_tilde2}
\tilde P_\epsilon\to \tilde P \hbox{ strongly in }L^{r'}(\Omega),
\end{equation}
\begin{equation}\label{conv_pres_hat2}
\hat P_\epsilon\to \tilde P \hbox{ strongly in }L^{r'}(\omega\times Z),
\end{equation}
where $r'$ is the conjugate exponent of $r$.
\end{itemize}
\end{lemma}
\begin{proof} We give some remarks concerning case $(i)$, case $(ii)$ being similar. The first estimate in (\ref{esti_P}) implies, up to a subsequence, the existence of $\tilde P\in L^2_0(\Omega)$ such that
\begin{equation}\label{conv_p_proof}
\tilde P_\epsilon\rightharpoonup \tilde P\hbox{ weakly in }L^2(\Omega).
\end{equation}
Also, from the second estimate in (\ref{esti_P}), since $\partial_{z_3}\tilde P_\epsilon/\epsilon$ also converges weakly in $H^{-1}(\Omega)$, we obtain $\partial_{z_3}\tilde P=0$ and so $\tilde P$ is independent of $z_3$. Moreover, arguing in \cite[Lemma 4.4]{Bourgeat1}, we deduce that the convergence (\ref{conv_p_proof}) of the pressure $\tilde P_\epsilon$ is in fact strong. Since $\tilde P_\epsilon$ has null mean value in $\Omega$, then $\tilde P$ has null mean value in $\omega$, which concludes the proof of (\ref{conv_pres_tilde}). Finally,  the strong convergence of $\hat P_\epsilon$ given in (\ref{conv_pres_hat}) follows from  \cite[Proposition 1.9-(ii)]{Cioran-book} and the strong convergence of $\tilde P_\epsilon$ given in (\ref{conv_pres_tilde}).
\end{proof}

\subsection{Proof of Theorems \ref{mainthmPseudo}, \ref{mainthmDilatant} and \ref{mainthmNewtonian}.}\label{sec:proofs_of_thms}
Using monotonicity arguments together with Minty's lemma (see for instance \cite{EkelandTemam, Tapiero2}), we derive a variational inequality that will be useful in the proofs.  
 
We choose a test function $v(x',z)\in \mathcal{D}(\omega; C^\infty_{\#}(Z)^3)$ with $v(x',z)=0$ in $\omega\times T$ and on $\omega \times Z'\times \{0, 1\}$. Multiplying (\ref{2}) by $v(x',x'/\epsilon,z_3)$, integrating by parts, and taking into account the extension of $\tilde u_\epsilon$ and $\tilde P_\epsilon$, we have
$$\begin{array}{l}
\displaystyle
\medskip
\epsilon^\gamma (\eta_0-\eta_\infty)\int_{ \Omega }(1+\lambda|\mathbb{D}_\epsilon[\tilde u_\epsilon]|^2)^{{r\over 2}-1}\mathbb{D}_\epsilon[\tilde u_\epsilon] :\left( \mathbb{D}_{x'}[ v]+ \epsilon^{-1}\mathbb{D}_{z}[ v]\right)\,dx'dz_3\\
\medskip
\displaystyle
+\epsilon^\gamma\eta_\infty\int_{ \Omega}\mathbb{D}_\epsilon[\tilde u_\epsilon] :\left( \mathbb{D}_{x'}[ v]+ \epsilon^{-1}\mathbb{D}_{z}[ v]\right)\,dx'dz_3\\
\medskip
\displaystyle
-\int_{\Omega}\tilde P_\epsilon \left({\rm div}_{x'}v'+\epsilon^{-1}{\rm div}_{z}v\right)\,dx'dz_3=\int_{\Omega} f'\cdot v'\,dx'dz_3+O_\epsilon,
\end{array}
$$
where $O_\epsilon$ is a generic real sequence depending on $\epsilon$ that can change from line to line.

By the change of variables given in Remark \ref{remarkCV}, we obtain
\begin{equation}\label{form_var_hat}\begin{array}{l}
\displaystyle
\medskip
\epsilon^{\gamma-1}(\eta_0-\eta_\infty)\int_{ \omega\times Z }(1+\lambda |\epsilon^{-1}\mathbb{D}_{z}[\hat u_\epsilon]|^2)^{{r\over 2}-1}\left(\epsilon^{-1} \mathbb{D}_{z}[\hat u_\epsilon] \right): \mathbb{D}_{z}[  v]\,dx'dz\\
\medskip
\displaystyle
+\epsilon^{\gamma-1}\eta_\infty\int_{ \omega\times Z}\epsilon^{-1}\mathbb{D}_{z}[\hat u_\epsilon] : \mathbb{D}_{z}[ v]\,dx'dz\\
\medskip
\displaystyle
-\int_{\omega\times Z}\hat P_\epsilon\, {\rm div}_{x'}v'\,dx'dz-\epsilon^{-1}\int_{\omega\times Z}\hat P_\epsilon\,{\rm div}_{z}v\,dx'dz=\int_{\omega\times Z} f'\cdot v'\,dx'dz+O_\epsilon,
\end{array}
\end{equation}
with $|O_\epsilon|\leq C\epsilon$ for every $\gamma\in\mathbb{R}$.

Now, let us define the functional $J_r$ by
$$
J_r(v)={\eta_0-\eta_\infty\over r\lambda}\int_{\omega\times Z}(1+\lambda|\mathbb{D}_z[v]|^2)^{r\over 2}dx'dz+{\eta_\infty\over 2}\int_{\omega\times Z}|\mathbb{D}_z[v]|^2dx'dz.
$$
Observe that $J_r$ is convex and Gateaux differentiable on $L^q(\omega; W^{1,q}_{\#}(Z)^3)$ with  $q=\max\{2,r\}$, (see \cite[Proposition 2.1 and Section 3]{Baranger} for more details) and $A_r=J'_r$ is given by
$$
(A_r(w),v)=(\eta_0-\eta_\infty)\int_{\omega\times Z}(1+\lambda|\mathbb{D}_z[w]|^2)^{{r\over 2}-1}\mathbb{D}_z[w]:\mathbb{D}_z[v]dx'dz+\eta_\infty\int_{\omega\times Z}\mathbb{D}_z[w]:\mathbb{D}_z[v]dx'dz.
$$
Applying \cite[Proposition 1.1., p.158]{Lions2}, $A_r$ is monotone, \emph{i.e.}
\begin{equation}\label{monotonicity}
(A_r(w)-A_r(v),w-v)\ge0,\quad \forall w,v\in L^q(\omega; W^{1,q}_{\#}(Z)^3).
\end{equation}
On the other hand, for all $\varphi\in \mathcal{D}(\omega; C^\infty_{\#}(Z)^3)$ with $\varphi=0$ in $\omega\times T$ and on $\omega\times Z'\times \{0, 1\}$, satisfying the divergence conditions ${\rm div}_{x'}\int_Z \varphi'\,dz=0$ in $\omega$, $\int_{Z_f}\varphi(x',z)\,dz \cdot n=0$ on $\partial\omega$, and ${\rm div}_{z}\varphi=0$ in $\omega\times Z$,  we choose $v_\epsilon$ defined by
$$
v_\epsilon=\varphi-\epsilon^{-1}\hat u_\epsilon,
$$
as a test function in (\ref{form_var_hat}). 

Taking into account that ${\rm div}_{\epsilon}\tilde u_\epsilon=0$, we get that $\epsilon^{-1}{\rm div}_{\color{darkgreen}z} \hat u_\epsilon=0$, and then we obtain
\begin{equation*}\begin{array}{l}
\displaystyle
\medskip
\epsilon^{\gamma-1}(A_r(\epsilon^{-1}\hat u_\epsilon),v_\epsilon)-\int_{\omega\times Z}\hat P_\epsilon\, {\rm div}_{x'}v'_\epsilon\,dx'dz=\int_{\omega\times Z} f'\cdot v'_\epsilon\,dx'dz+O_\epsilon,
\end{array}
\end{equation*}
which is equivalent to
\begin{equation*}\begin{array}{l}
\displaystyle
\medskip
\epsilon^{\gamma-1}(A_r(\varphi)-A_r(\epsilon^{-1}\hat u_\epsilon),v_\epsilon)-\epsilon^{\gamma-1}(A_r(\varphi),v_\epsilon)+\int_{\omega\times Z}\hat P_\epsilon\, {\rm div}_{x'}v'_\epsilon\,dx'dz=-\int_{\omega\times Z} f'\cdot v'_\epsilon\,dx'dz+O_\epsilon.
\end{array}
\end{equation*}
Due to (\ref{monotonicity}), we can deduce
\begin{equation*}\begin{array}{l}
\displaystyle
\medskip
\epsilon^{\gamma-1}(A_r(\varphi),v_\epsilon)-\int_{\omega\times Z}\hat P_\epsilon\, {\rm div}_{x'}v'_\epsilon\,dx'dz\ge\int_{\omega\times Z} f'\cdot v'_\epsilon\,dx'dz+O_\epsilon,
\end{array}
\end{equation*}
\emph{i.e.}
\begin{equation}\label{v_ineq_carreau}\begin{array}{l}
\displaystyle
\medskip
\epsilon^{\gamma-1}(\eta_0-\eta_\infty)\int_{\omega\times Z}(1+\lambda|\mathbb{D}_z[\varphi]|^2)^{{r\over 2}-1}\mathbb{D}_z[\varphi]:\mathbb{D}_z[v_\epsilon]dx'dz+\epsilon^{\gamma-1}\eta_\infty\int_{\omega\times Z}\mathbb{D}_z[\varphi]:\mathbb{D}_z[v_\epsilon]dx'dz\\
\medskip
\displaystyle-\int_{\omega\times Z}\hat P_\epsilon\, {\rm div}_{x'}v'_\epsilon\,dx'dz\ge\int_{\omega\times Z} f'\cdot v'_\epsilon\,dx'dz+O_\epsilon.
\end{array}
\end{equation}
In the above inequality, in case $1<r\leq 2$, $|O_\epsilon|\leq C\epsilon^\alpha$, with $\alpha=1$ if  $\gamma\leq 1$ and $\alpha=2-\gamma$ if $\gamma>1$. In case $r>2$, $|O_\epsilon|\leq C\epsilon^\alpha$, with $\alpha=\gamma+{2\over r}(1-\gamma)$ if $\gamma\leq 1$ and $\alpha=1-{\gamma-1\over r-1}$ if $\gamma>1$.

\begin{proof}[Proof of Theorem \ref{mainthmPseudo}]

We recall that $1<r<2$. The case $\gamma=1$ is developed in \cite{Carreau_Ang_Bonn_SG}, so we omit it and consider that $\gamma\neq 1$. The proof will be divided in two steps. In the first step, we obtain the homogenized behaviour  given by a coupled system, with a constant macroviscosity, and in the second step we decouple it to obtain the macroscopic law. \\

{\it Step 1.} Using Lemmas \ref{lemma_compactness} and \ref{lemma_compactness_p},  in this step we will prove that the sequence $(\epsilon^{\gamma-2}\hat u_\epsilon, \tilde P_\epsilon)$ converges to $(\hat u, \tilde P)\in L^2(\omega;H^1_{\#}(Z_f)^3)\times (L^2_0(\omega)\cap H^1(\omega))$, characterized as the unique solutions of  the following two-pressures Newtonian Stokes problem with the linear viscosity $\eta$ equal to $\eta_0$ if $\gamma<1$ and $\eta_\infty$ if $\gamma>1$:
\begin{equation}\label{prueba_lim_1}
\left\{\begin{array}{rl}
\medskip
-\eta\,{\rm div}_{\color{darkgreen}z}\mathbb{D}_z[\hat u] + \nabla_z\hat \pi= f' -\nabla_{x'} \tilde P & \hbox{in}\quad\omega\times Z_f, \\
\medskip\displaystyle
{\rm div}_z \hat u =0& \hbox{in}\quad\omega\times Z_f, \\
\medskip\displaystyle
{\rm div}_{x'}\left(\int_{Z_f}\hat u'\,dz\right)=0&\hbox{in }\omega, \\
\medskip\displaystyle
\left(\int_{Z_f}\hat u'\,dz\right)\cdot n=0&\hbox{on }\partial\omega, \\
\medskip\displaystyle
\hat u=0& \hbox{in }\omega\times T,
\\
\medskip\displaystyle
\hat u=0& \hbox{in }\omega\times Z'\times \{0, 1\},\\
\medskip\displaystyle
\hat \pi\in  L^2(\omega;L^2_{0,\#}(Z_f)).
\end{array}\right.
\end{equation}
Divergence conditions (\ref{prueba_lim_1})$_{2,3,4}$ and condition (\ref{prueba_lim_1})$_{5,6}$ follow from Lemma \ref{lemma_compactness}. To prove that $(\hat u, \tilde P)$ satisfies the momentum equation given in (\ref{prueba_lim_1}), 
we follow the lines of the proof to obtain (\ref{v_ineq_carreau}) but choose now $v_\epsilon$ and $\varphi$ such that $v_\epsilon=\epsilon^{1-\gamma}\varphi-\epsilon^{-1}\hat u_\epsilon$  with $\varphi\in \mathcal{D}(\omega; C^\infty_{\#}(Z)^3)$ satisfying $\varphi=0$ in $\omega\times T$ and on $\omega \times Z'\times \{0, 1\}$,   ${\rm div}_{x'}\int_Z \varphi'\,dz=0$ in $\omega$, $\int_Z \varphi'\,dz\cdot n=0$ on $\partial \omega$, and ${\rm div}_{z}\varphi=0$ in $\omega\times Z$.  Then, we get
\begin{equation}\label{v_ineq_carreau_1}\begin{array}{l}
\displaystyle
\medskip
\epsilon^{1-\gamma}(\eta_0-\eta_\infty)\int_{\omega\times Z}(1+\lambda\epsilon^{2(1-\gamma)}|\mathbb{D}_z[\varphi]|^2)^{{r\over 2}-1}\mathbb{D}_z[\varphi]:\mathbb{D}_z[\varphi-\epsilon^{\gamma-2}  \hat u_\epsilon]\,dx'dz\\
\medskip
\displaystyle
+\epsilon^{1-\gamma}\eta_\infty\int_{\omega\times Z}\mathbb{D}_z[\varphi]:\mathbb{D}_z[\varphi-\epsilon^{\gamma-2}  \hat u_\epsilon]dx'dz\\
\medskip
\displaystyle-\epsilon^{1-\gamma}\int_{\omega\times Z}\hat P_\epsilon\, {\rm div}_{x'}(\varphi'-\epsilon^{\gamma-2}  \hat u'_\epsilon)\,dx'dz\ge \epsilon^{1-\gamma}\int_{\omega\times Z} f'\cdot (\varphi'-\epsilon^{\gamma-2}  \hat u'_\epsilon)\,dx'dz+O_\epsilon,
\end{array}
\end{equation}
where $|O_\epsilon|\leq C\epsilon^{2-\gamma}$. Dividing by $\epsilon^{1-\gamma}$, we deduce
\begin{equation}\label{v_ineq_carreau_12}\begin{array}{l}
\displaystyle
\medskip
 (\eta_0-\eta_\infty)\int_{\omega\times Z}(1+\lambda\epsilon^{2(1-\gamma)}|\mathbb{D}_z[\varphi]|^2)^{{r\over 2}-1}\mathbb{D}_z[\varphi]:\mathbb{D}_z[\varphi-\epsilon^{\gamma-2}  \hat u_\epsilon]\,dx'dz\\
\medskip
\displaystyle
+\ \eta_\infty\int_{\omega\times Z}\mathbb{D}_z[\varphi]:\mathbb{D}_z[\varphi-\epsilon^{\gamma-2}  \hat u_\epsilon]dx'dz\\
\medskip
\displaystyle - \int_{\omega\times Z}\hat P_\epsilon\, {\rm div}_{x'}(\varphi'-\epsilon^{\gamma-2}  \hat u'_\epsilon)\,dx'dz \ge  \int_{\omega\times Z} f'\cdot (\varphi'-\epsilon^{\gamma-2}  \hat u'_\epsilon)\,dx'dz+O_\epsilon,
\end{array}
\end{equation}
where $|O_\epsilon|\leq C\epsilon$, which tends to zero when $\epsilon\to 0$. Now, we can pass to the limit depending on the value of $\gamma$:
\begin{itemize}
\item[-] If $\gamma<1$, then $2(1-\gamma)>0$ and so, $\lambda\epsilon^{2(1-\gamma)}|\mathbb{D}_z[\varphi]|^2$ tends to zero. From convergence (\ref{conv_vel_gorro}), passing to the limit when $\epsilon$ tends to zero in (\ref{v_ineq_carreau_12}), we have that the first and second terms converge to 
$$\eta_0\int_{ \omega\times Z}\mathbb{D}_{z}[\varphi] : \mathbb{D}_{z}[ \varphi- \hat u ]\,dx'dz.$$

From convergences (\ref{conv_vel_gorro}) and (\ref{conv_pres_hat}), the third term converges to 
$$\int_{\omega\times Z}\tilde P\, {\rm div}_{x'}(\varphi'- \hat u' )\,dx'dz.$$
Since $\tilde P$ does not depend on $z$, by the divergence conditions ${\rm div}_{x'}\int_Z \varphi'\,dz=0$ and (\ref{divyproperty})$_2$,
$$\int_{\omega\times Z}\tilde P\, {\rm div}_{x'}(\varphi'- \hat u' )\,dx'dz=\int_{\omega}\tilde P\, {\rm div}_{x'}\left(\int_Z(\varphi'- \hat u' )dz\right)\,dx'=0.$$
Thus, passing to the limit in the variational inequality (\ref{v_ineq_carreau_12}) yields
$$\begin{array}{l}
\displaystyle
\medskip
\eta_0\int_{ \omega\times Z}\mathbb{D}_{z}[\varphi] : \mathbb{D}_{z}[ \varphi-  \hat u]\,dx'dz\geq \int_{\omega\times Z} f'\cdot (\varphi'-\hat u')\,dx'dz.
\end{array}
$$
Since $\varphi$ is arbitrary, by Minty's lemma, see \cite[Chapter 3, Lemma 1.2]{Lions2}, we deduce 
\[
\eta_0\int_{ \omega\times Z}\mathbb{D}_{z}[\hat u] : \mathbb{D}_{z}[ \varphi]\,dx'dz= \int_{\omega\times Z} f'\cdot \varphi'\,dx'dz\, .
\]

\item[-] If $\gamma>1$, then $2(1-\gamma)<0$ and so, $(1+\lambda\epsilon^{2(1-\gamma)}|\mathbb{D}_z[\varphi]|^2)^{{r\over 2}-1}$ tends to zero. From convergence (\ref{conv_vel_gorro}), passing to the limit when $\epsilon$ tends to zero in (\ref{v_ineq_carreau_12}), we have that the first and second terms converge to 
$$\eta_\infty\int_{ \omega\times Z}\mathbb{D}_{z}[\varphi] : \mathbb{D}_{z}[ \varphi- \hat u ]\,dx'dz.$$
Treating the third integral term exactly as above, we conclude that the following equality holds true:
\[
\eta_\infty\int_{ \omega\times Z}\mathbb{D}_{z}[\hat u] : \mathbb{D}_{z}[ \varphi]\,dx'dz= \int_{\omega\times Z} f'\cdot \varphi'\,dx'dz\, .
\]

\end{itemize}

\noindent In summary,  considering  $\eta$ equal  to $\eta_0$ if $\gamma<1$ or $\eta_\infty$ if $\gamma>1$,  we have obtained that by density, the equality
\begin{equation}\label{form_var_hat2}\begin{array}{l}
\displaystyle
\medskip
\eta\int_{ \omega\times Z}\mathbb{D}_{z}[\hat u] : \mathbb{D}_{z}[ v]\,dx'dz= \int_{\omega\times Z} f'\cdot v'\,dx'dz
\end{array}
\end{equation}
 is satisfied by every $v$ in the Hilbert space $\mathcal{V}$ defined by
\begin{equation}\label{Vspace}
\mathcal{V}=\left\{\begin{array}{l}
\medskip
v(x',z)\in L^2(\omega;H^1_{\#}(Z)^3) \hbox{ such that }\\
\medskip
\displaystyle {\rm div}_{x'}\left(\int_{Z_f}v(x',z)\,dz \right)=0\hbox{ in }\omega,\quad \left(\int_{Z_f}v(x',z)\,dz \right)\cdot n=0\hbox{ on }\partial\omega\\
\medskip
{\rm div}_zv(x',z)=0\hbox{ in }\omega\times Z_f,\quad v(x',z)=0\hbox{ in }\omega\times T \hbox{ and on } \omega\times Z'\times \{0, 1\}
\end{array}\right\}\,.
\end{equation}
Reasoning as in \cite[Lemma 1.5]{Allaire0}, the orthogonal of $\mathcal{V}$, a subset of $L^2(\omega;H^{-1}_{\#}(Z)^3)$,  is made of gradients of the form 
 $\nabla_{x'}\tilde \pi(x')+\nabla_z \hat \pi(x',y)$, with $\tilde \pi(x')\in H^1(\omega)/\mathbb{R}$ and $\hat \pi(x',z)\in L^2(\omega;L^2_{\#}(Z_f)/\mathbb{R})$.  Thus, integrating by parts, the variational formulation (\ref{form_var_hat2})  is equivalent to the two-pressures  Newtonian Stokes problem (\ref{prueba_lim_1}).  It remains to prove that  $\tilde \pi$ coincides with pressure $\tilde P$, which can be easily done by passing to the limit similarly as above, considering a test function $\varphi$ that is divergence-free only in ${\color{darkgreen}z}$, and identifying limits. Hence, $\tilde P\in L^2_0(\omega)\cap H^1(\omega)$. Last, from \cite{Allaire0}, problem (\ref{prueba_lim_1}) admits a unique solution $(\hat u, \hat \pi, \tilde P)\in L^2(\omega;H^1_{\#}(Z_f)^3)\times L^2(\omega;L^2_{0,\#}(Z_f))\times (L^2_0(\omega)\cap H^1(\omega))$, which implies that the entire sequence $(\hat u_\epsilon, \hat P_\epsilon)$ converges to  $(\hat  u, \tilde P)$.\\

 {\it Step 2.} To prove (\ref{thm:system}), it remains to eliminate the microscopic variable ${\color{darkgreen}z}$ in the effective problem (\ref{prueba_lim_1}).   The procedure is rather standard and is detailed for instance in \cite{Anguiano_SuarezGrau2}, but for the reader's convenience, we give some details on the proof.
 From the first equation of (\ref{prueba_lim_1}), the velocity $\hat u$ and pressure $\hat \pi$ can be expressed in terms of the macroscopic force and pressure gradient, and local velocity/pressure $w^i,\pi^i$ (defined by~\eqref{LocalProblemNewtonian}), as
$$
\hat u(x',z)={1\over \eta}\sum_{i=1}^2\left(f_i(x')-\partial_{x_i}\tilde P(x')\right) w^i(z),\quad  \hat \pi(x',z)=\sum_{i=1}^2\left(f_i(x')-\partial_{x_i}\tilde P(x')\right) \pi^i(z).$$ 
Integrating the expression of $\hat u$ on $Z$ and taking into account that $\tilde V(x')=\int_0^1\tilde u(x',z_3)dz_3=\int_{Z_f}\hat u(x',z)dz$ and $\int_{Z_f}\hat u_3\,dz=0$, we get Darcy's law (\ref{thm:system})$_1$ since the matrix $\mathcal{A}\in \mathbb{R}^{2\times 2}$ satisfies (\ref{permfuncNew}). Combining the expression of $\tilde V$ with the divergence-free condition on $\tilde V$ given by (\ref{prueba_lim_1})$_3$ yields the lower-dimensional homogenized Darcy's law (\ref{thm:system})$_2$.

\end{proof}

\begin{proof}[Proof of Theorem \ref{mainthmDilatant}.]
We recall that in this case $r>2$. The proof will be divided in four steps. In the first step, we obtain the homogenized behaviour in case $\gamma<1$. Second step derives the homogenized behaviour for $\gamma>1$ given by a coupled system with a non-linear macroviscosity of power law type, which will be decoupled to obtain the macroscopic law in the third step. Finally, in the fourth step, we consider the case $\gamma=1$.\\

{\it Step 1.}  Take $\gamma<1$. 
  From Lemmas \ref{lemma_compactness} and \ref{lemma_compactness_p},  first we will prove that the sequence $(\epsilon^{\gamma-2}\hat u_\epsilon, \tilde P_\epsilon)$ converges to $(\hat u, \tilde P)\in L^2(\omega;H^1_{\#}(Z_f)^3)\times (L^{r'}_0(\omega)\cap W^{1,r'}(\omega))$, which will be the unique solutions of  the following two-pressures Newtonian Stokes problem with the linear viscosity $\eta$ equal to $\eta_0$ :
\begin{equation}\label{prueba_lim_1rgreater2}
\left\{\begin{array}{rl}
\medskip
-\eta_0\,{\rm div}_z\mathbb{D}_z[\hat u] + \nabla_z\hat \pi= f' -\nabla_{x'} \tilde P & \hbox{in}\quad\omega\times Z_f, \\
\medskip\displaystyle
{\rm div}_z \hat u =0& \hbox{in}\quad\omega\times Z_f, \\
\medskip\displaystyle
{\rm div}_{x'}\left(\int_{Z_f}\hat u'\,dz\right)=0&\hbox{in }\omega, \\
\medskip\displaystyle
\left(\int_{Z_f}\hat u'\,dz\right)\cdot n=0&\hbox{on }\partial\omega, \\
\medskip\displaystyle
\hat u=0& \hbox{in }\omega\times T,\\
\medskip\displaystyle
\hat u=0 &\hbox{ on }\omega\times Z'\times \{0, 1\},
\\
\medskip
\hat \pi\in  L^{2}(\omega;L^{2}_{0,\#}(Z_f)).
\end{array}\right.
\end{equation}
The proof of (\ref{prueba_lim_1rgreater2}) is similar to {\it Step 1} of the proof of Theorem \ref{mainthmPseudo}. We give the main steps:
\begin{itemize}
\item[-] We  deduce the variational inequality 
\begin{equation}\label{v_ineq_carreau_12rgreater2}\begin{array}{l}
\displaystyle
\medskip
 (\eta_0-\eta_\infty)\int_{\omega\times Z}(1+\lambda\epsilon^{2(1-\gamma)}|\mathbb{D}_z[\varphi]|^2)^{{r\over 2}-1}\mathbb{D}_z[\varphi]:\mathbb{D}_z[\varphi-\epsilon^{\gamma-2}  \hat u_\epsilon]\,dx'dz\\
\medskip
\displaystyle
+\ \eta_\infty\int_{\omega\times Z}\mathbb{D}_z[\varphi]:\mathbb{D}_z[\varphi-\epsilon^{\gamma-2}  \hat u_\epsilon]dx'dz\\
\medskip
\displaystyle- \int_{\omega\times Z}\hat P_\epsilon\, {\rm div}_{x'}(\varphi'-\epsilon^{\gamma-2}  \hat u'_\epsilon)\,dx'dz\ge  \int_{\omega\times Z} f'\cdot (\varphi'-\epsilon^{\gamma-2}  \hat u'_\epsilon)\,dx'dz+O_\epsilon,
\end{array}
\end{equation}
for  $\varphi\in \mathcal{D}(\omega; C^\infty_{\#}(Z)^3)$ with $\varphi=0$ in $\omega\times T$ and on $\omega\times Z'\times \{0, 1\}$, satisfying the divergence conditions ${\rm div}_{x'}\int_Z \varphi'\,dz=0$ in $\omega$, $\int_Z \varphi'\,dz\cdot n=0$ on $\partial \omega$, and ${\rm div}_{z}\varphi=0$ in $\omega\times Z$, where $|O_\epsilon|\leq C\epsilon$, which tends to zero when $\epsilon\to 0$. 

\item[-] Next, we can pass to the limit in every term of (\ref{v_ineq_carreau_12rgreater2}) when $\epsilon$ tends to zero. The rest of the proof is similar to the one given in {\it Step 1}  of the proof of Theorem \ref{mainthmPseudo} but taking into account convergences (\ref{conv_vel_gorro2}) and (\ref{conv_pres_hat2}).  We notice that since $\gamma<1$, $\lambda\epsilon^{2(1-\gamma)}$ tends to zero and thus,  since $r>2$, $(1+\lambda\epsilon^{2(1-\gamma)}|\mathbb{D}_y[\varphi]|^2)^{{r\over 2}-1}$ tends to $1$. As a result, the first and second terms of (\ref{v_ineq_carreau_12rgreater2}) converge to 
$$\eta_0\int_{ \omega\times Z}\mathbb{D}_{z}[\varphi] : \mathbb{D}_{z}[ \varphi- \hat u ]\,dx'dz.$$
\item[-] We conclude that $(\hat u, \tilde P)$ satisfies the variational formulation
$$\begin{array}{l}
\displaystyle
\medskip
\eta_0\int_{ \omega\times Z}\mathbb{D}_{z}[\hat u] : \mathbb{D}_{z}[ v]\,dx'dz= \int_{\omega\times Z} f'\cdot v'\,dx'dz,
\end{array}
$$
 for every function $v$ in the Hilbert space $\mathcal{V}$ defined by (\ref{Vspace}). This variational formulation is equivalent to problem (\ref{prueba_lim_1rgreater2}).
\end{itemize}
Finally, following {\it Step 2} from the proof of Theorem \ref{mainthmPseudo}, we deduce the lower-dimensional homogenized Darcy's law (\ref{thm:system}) with  $\eta=\eta_0$. By definiteness of the matrix $\mathcal{A}$, system (\ref{thm:system})  has a unique solution $\tilde V'\in L^2(\omega)^2$, $\tilde P\in H^1(\omega)\cap L^2_0(\omega)$. Therefore, the limits do not depend on the subsequences.    \\

{\it Step 2}. Take $\gamma>1$. From Lemmas \ref{lemma_compactness} and \ref{lemma_compactness_p},  we prove that the sequence $(\hat u_\epsilon, \tilde P_\epsilon)$ converges to $(\hat u, \tilde P)\in L^r(\omega;W^{1,r}_{\#}(Z_f)^3)\times (L^{r'}_0(\omega)\cap W^{1,r'}(\omega))$, which are the unique solutions of  the following two-pressures non-Newtonian Stokes problem with a non-linear power law viscosity:
\begin{equation}\label{prueba_lim_1_power}
\left\{\begin{array}{rl}
\medskip
-(\eta_0-\eta_\infty)\lambda^{r-2\over 2}\,{\rm div}_z(|\mathbb{D}_z[\hat u]|^{r-2}\mathbb{D}_z[\hat u]) + \nabla_z\hat \pi= f' -\nabla_{x'} \tilde P & \hbox{in}\quad\omega\times Z_f, \\
\medskip\displaystyle
{\rm div}_z \hat u =0& \hbox{in}\quad\omega\times Z_f, \\
\medskip\displaystyle
{\rm div}_{x'}\left(\int_{Z_f}\hat u'\,dz\right)=0&\hbox{in }\omega, \\
\medskip\displaystyle
\left(\int_{Z_f}\hat u'\,dz\right)\cdot n=0&\hbox{on }\partial\omega, \\
\medskip\displaystyle
\hat u=0& \hbox{in }\omega\times T,\\
\medskip\displaystyle
\hat u=0 & \hbox{on }\omega\times Z'\times \{0, 1\},
\\
\medskip
\hat \pi\in  L^{r'}(\omega;L^{r'}_{0,\#}(Z_f)).
\end{array}\right.
\end{equation}
Divergence conditions (\ref{prueba_lim_1_power})$_{2,3,4}$ and condition (\ref{prueba_lim_1_power})$_5$ follow from Lemma \ref{lemma_compactness}. To prove that $(\hat u, \tilde P)$ satisfies the momentum equation given in (\ref{prueba_lim_1_power}), we follow the lines of the proof of (\ref{v_ineq_carreau}) but considering now $v_\epsilon:=\epsilon^{1-\gamma\over r-1}\varphi-\epsilon^{-1}\hat u_\epsilon$, where  $\varphi\in \mathcal{D}(\omega; C^\infty_{\#}(Z)^3)$ ssatisfies $\varphi=0$ in $\omega\times T$ and on $\omega\times Z'\times \{0, 1\}$, the divergence conditions ${\rm div}_{x'}\int_Z \varphi'\,dz=0$ in $\omega$, $\int_Z \varphi'\,dz\cdot n=0$ on $\partial\omega$, and ${\rm div}_{z}\varphi=0$ in $\omega\times Z$. We get
\begin{equation}\label{v_ineq_carreau_1_r22}\begin{array}{l}
\displaystyle
\medskip
\epsilon^{\gamma-1 + 2{1-\gamma\over r-1}}(\eta_0-\eta_\infty)\int_{\omega\times Z}(1+\lambda\epsilon^{2{1-\gamma\over r-1}}|\mathbb{D}_z[\varphi]|^2)^{{r\over 2}-1}\mathbb{D}_z[\varphi]:\mathbb{D}_z[\varphi-\epsilon^{\gamma-r\over r-1} \hat u_\epsilon]\,dx'dz\\
\medskip
\displaystyle
+\epsilon^{\gamma-1 + 2{1-\gamma\over r-1}}\eta_\infty\int_{\omega\times Z}\mathbb{D}_z[\varphi]:\mathbb{D}_z[\varphi-\epsilon^{\gamma-r\over r-1}\hat u_\epsilon]\,dx'dz\\
\medskip
\displaystyle-\epsilon^{1-\gamma\over r-1}\int_{\omega\times Z}\hat P_\epsilon\, {\rm div}_{x'}(\varphi'-\epsilon^{\gamma-r\over r-1} \hat u'_\epsilon)\,dx'dz\ge \epsilon^{1-\gamma\over r-1}\int_{\omega\times Z} f'\cdot (\varphi'-\epsilon^{\gamma-r\over r-1} \hat u'_\epsilon)\,dx'dz+O_\epsilon,
\end{array}
\end{equation}
where $|O_\epsilon|\leq C\epsilon^{1+{1-\gamma\over r-1}}$. Dividing by $\epsilon^{1-\gamma\over r-1}$, we deduce the inequality
\begin{equation}\label{v_ineq_carreau_1_r22}\begin{array}{l}
\displaystyle
\medskip
\epsilon^{(\gamma-1){r-2\over r-1}}(\eta_0-\eta_\infty)\int_{\omega\times Z}(1+\lambda\epsilon^{2{1-\gamma\over r-1}}|\mathbb{D}_z[\varphi]|^2)^{{r\over 2}-1}\mathbb{D}_z[\varphi]:\mathbb{D}_z[\varphi-\epsilon^{\gamma-r\over r-1} \hat u_\epsilon]\,dx'dz\\
\medskip
\displaystyle
+\epsilon^{(\gamma-1){r-2\over r-1}}\eta_\infty\int_{\omega\times Z}\mathbb{D}_z[\varphi]:\mathbb{D}_z[\varphi-\epsilon^{\gamma-r\over r-1}\hat u_\epsilon]\,dx'dz\\
\medskip
\displaystyle-\int_{\omega\times Z}\hat P_\epsilon\, {\rm div}_{x'}(\varphi'-\epsilon^{\gamma-r\over r-1} \hat u'_\epsilon)\,dx'dz\ge \int_{\omega\times Z} f'\cdot (\varphi'-\epsilon^{\gamma-r\over r-1} \hat u'_\epsilon)\,dx'dz+O_\epsilon,
\end{array}
\end{equation}
with $|O_\epsilon|\leq C\epsilon$.
\\

Since $\gamma>1$ and $r>2$, we have $(\gamma-1){r-2\over r-1}>0$, and from
$$
\epsilon^{(\gamma-1){r-2\over r-1}}(1+\lambda\epsilon^{2{1-\gamma\over r-1}}|\mathbb{D}_z[\varphi]|^2)^{{r\over 2}-1}=(\epsilon^{2{\gamma-1\over r-1}}+ \lambda|\mathbb{D}_z[\varphi]|^2)^{{r\over 2}-1}
$$
and convergences (\ref{conv_vel_gorro22}) and (\ref{conv_pres_hat2}), passing to the limit in (\ref{v_ineq_carreau_1_r22}) when $\epsilon$ tends to zero, we deduce 
\begin{equation}\label{v_ineq_carreau_1_r222}\begin{array}{l}
\displaystyle
\medskip
(\eta_0-\eta_\infty)\lambda^{r-2\over 2}\int_{\omega\times Z}|\mathbb{D}_z[\varphi]|^{r-2}\mathbb{D}_z[\varphi]:\mathbb{D}_z[\varphi-\hat u ]\,dx'dz\\
\medskip
\displaystyle-\int_{\omega\times Z}\tilde P\, {\rm div}_{x'}(\varphi'-  \hat u' )\,dx'dz\ge \int_{\omega\times Z} f'\cdot (\varphi'-  \hat u' )\,dx'dz.
\end{array}
\end{equation}
Since $\tilde P$ does not depend on $\color{darkgreen} z$, by the divergences conditions ${\rm div}_{x'}\int_Z \varphi'\,dz=0$ and (\ref{prueba_lim_1_power})$_3$, one has 
$$\int_{\omega\times Z}\tilde P\, {\rm div}_{x'}(\varphi'- \hat u' )\,dx'dz=\int_{\omega}\tilde P\, {\rm div}_{x'}\left(\int_Z(\varphi'- \hat u' )dz\right)\,dx'=0.$$
Hence, the variational inequality (\ref{v_ineq_carreau_1_r222}) reads
$$\begin{array}{l}
\displaystyle
\medskip
(\eta_0-\eta_\infty)\lambda^{r-2\over 2}\int_{\omega\times Z}|\mathbb{D}_z[\varphi]|^{r-2}\mathbb{D}_z[\varphi]:\mathbb{D}_z[\varphi-\hat u ]\,dx'dz\geq \int_{\omega\times Z} f'\cdot (\varphi'-\hat u')\,dx'dz.
\end{array}
$$
Using Minty's lemma~\cite[Chapter 3, Lemma 1.2]{Lions2} and a density argument, we conclude that the equality 
\begin{equation}
\label{form_var_hat2_power}
\begin{array}{l}
\displaystyle
\medskip
(\eta_0-\eta_\infty)\lambda^{r-2\over 2}\int_{ \omega\times Z}|\mathbb{D}_{z}[\hat u]|^{r-2}\mathbb{D}_{z}[\hat u] : \mathbb{D}_{z}[ v]\,dx'dz= \int_{\omega\times Z} f'\cdot v'\,dx'dz\, 
\end{array}
\end{equation}
is valid for every $v$ in the Banach space $\mathcal{V}$ defined by
$$
\mathcal{V}=\left\{\begin{array}{l}
\medskip
v(x',z)\in L^r(\omega;W^{1,r}_{\#}(Z)^3) \hbox{ such that }\\
\medskip
\displaystyle {\rm div}_{x'}\left(\int_{Z_f}v(x',z)\,dz \right)=0\hbox{ in }\omega,\quad \left(\int_{Z_f}v(x',z)\,dz \right)\cdot n=0\hbox{ on }\partial\omega\\
\medskip
{\rm div}_zv(x',z)=0\hbox{ in }\omega\times Z_f,\quad v(x',z)=0\hbox{ in }\omega\times T \hbox{ and on }\omega\times Z'\times \{0, 1\}
\end{array}\right\}\,.
$$

Reasoning as in \cite[Lemma 1.5]{Allaire0}, the orthogonal of $\mathcal{V}$, a subset of $L^r(\omega;W^{-1,r'}_{\#}(Z)^3)$,  is made of gradients of the form 
 $\nabla_{x'}\tilde \pi(x')+\nabla_z \hat \pi(x',y)$, with $\tilde \pi(x')\in W^{1,r'}(\omega)/\mathbb{R}$ and $\hat \pi(x',z)\in L^{r'}(\omega;L^{r'}_{\#}(Z_f)/\mathbb{R})$.  Thus, integra\-ting by parts, the variational formulation (\ref{form_var_hat2_power})  is equivalent to the two-pressures non-Newtonian Stokes problem (\ref{prueba_lim_1_power}).    The identification of $\tilde \pi$ with $\tilde P$ is then performed analogously as in the proof of Theorem~\ref{mainthmPseudo}, and in particular $\tilde P\in L^{r'}_0(\omega)\cap W^{1,r'}(\omega)$. From \cite[Theorem 2]{Bourgeat1}, problem (\ref{prueba_lim_1_power}) admits a unique solution $(\hat u, \hat \pi, \tilde P)\in L^r(\omega;W^{1,r}_{\#}(Z_f)^3)\times L^{r'}(\omega;L^{r'}_{0,\#}(Z_f))\times (L^{r'}_0(\omega)\cap W^{1,r'}(\omega))$, hence the entire sequence $(\hat u_\epsilon, \hat P_\epsilon)$ converges to  $(\hat  u, \tilde P)$.\\
 
 {\it Step 3.} In this step we  give an approximation of the model (\ref{prueba_lim_1_power}), where the macroscopic scale is totally decoupled from the microscopic one. To do this, we  seek a global filtration velocity of the form given in (\ref{thm:system_22}), \emph{i.e.} 
\begin{equation}\label{filtration}
 \tilde V(x')=\frac{1}{\lambda^{\frac{2-r'}{2}}(\eta_0-\eta_\infty)^{r'-1}}\mathcal{U} \big(f'(x')-\nabla_{x'}\tilde P(x') \big)\quad\hbox{in }\omega,
\end{equation}
where $\mathcal{U}:\mathbb{R}^2\to\mathbb{R}^3$ is a permeability operator, not necessary linear, and $\tilde V(x')=\int_0^1\tilde u(x',z_3)\,dz_3=\int_{Z}\hat u(x',z)\,dz$  with ${\rm div}_{x'}\tilde V'=0$ in $\omega$ and $\tilde V'\cdot n=0$ on $\partial \omega$.

Using the idea from \cite{Bourgeat2} to decouple the homogenized problems of {\it power law} type, for every $\xi'\in\mathbb{R}^2$ we consider the function $\mathcal{U}:\mathbb{R}^2\to \mathbb{R}^3$ given by
$$\mathcal{U}(\xi')=\int_{Z_f}w_{\xi'}(z)\,dz,$$
where $w_{\xi'}$ denotes the unique solution of the local Stokes problem given by (\ref{LocalProblemNonNewtonian}), see \cite[Theorem 2]{Bourgeat1}.
Thus,  $(\hat u,\hat \pi)$ takes the form
\[
\hat u(x',z)=\frac{1}{\lambda^{\frac{2-r'}{2}}(\eta_0-\eta_\infty)^{r'-1}}w_{f'(x')-\nabla_{x'}\tilde P(x')}(z),\quad \hat \pi(x',z)=\pi_{f'(x')-\nabla_{x'}\tilde P(x')}(z)\quad\hbox{in }\omega\times Z\, .\]
Then, from the relation $\tilde V(x')=\int_{Z}\hat u(x',y)\,dz$ with $\int_Z\hat u_3(x',z)\,dz=0$ given in Lemma \ref{lemma_compactness}, we deduce the filtration velocity (\ref{filtration}), where $\tilde V_3=0$.  Moreover, from second and third conditions given in  (\ref{prueba_lim_1_power})  together with (\ref{filtration}), we deduce 
$${\rm div}_{x'} \tilde V'=0\quad\hbox{in }\omega,\quad \tilde V'\cdot n=0\quad\hbox{on }\partial\omega.$$

Since $\tilde V_3=0$, to simplify the notation, we redefine $\mathcal{U}$  by the expression given in (\ref{permfunc}) and then, we get $\mathcal{U}:\mathbb{R}^2\to \mathbb{R}^2$, which concludes the proof of (\ref{thm:system_22}). Finally, from \cite[Theorem 1]{Bourgeat2}, the macroscopic problem (\ref{thm:system_22}) has a unique solution $(V, \tilde P)\in L^r(\omega)^3\times (L^{r'}_0(\omega)\cap W^{1,r}(\omega))$ and Theorem \ref{mainthmDilatant} is proved.\\

{\it Step 4}. We consider the case $\gamma=1$. Reasoning as in {\it Step 2} with $\gamma=1$ and using convergences (\ref{conv_vel_gorro22_1}) and (\ref{conv_pres_hat2}), we deduce that  the sequence $(\hat u_\epsilon, \tilde P_\epsilon)$ converges to $(\hat u, \tilde P)\in L^r(\omega;W^{1,r}_{\#}(Z_f)^3)\times (L^{r'}_0(\omega)\cap W^{1,r'}(\omega))$, which are the unique solutions of  the following two-pressures non-Newtonian Stokes problem with non-linear Carreau viscosity (\ref{Carreau})
$$
\left\{\begin{array}{rl}
\medskip
-{\rm div}_y(\eta_r(\mathbb{D}_z[\hat u])\mathbb{D}_z[\hat u]) + \nabla_z\hat \pi= f' -\nabla_{x'} \tilde P & \hbox{in}\quad\omega\times Z_f, \\
\medskip\displaystyle
{\rm div}_z \hat u =0& \hbox{in}\quad\omega\times Z_f, \\
\medskip\displaystyle
{\rm div}_{x'}\left(\int_{Z_f}\hat u'\,dz\right)=0&\hbox{in }\omega, \\
\medskip\displaystyle
\left(\int_{Z_f}\hat u'\,dz\right)\cdot n=0&\hbox{on }\partial\omega, \\
\medskip\displaystyle
\hat u=0& \hbox{in }\omega\times T,\\
\medskip\displaystyle
\hat u=0 &\hbox{on }\omega\times Z'\times \{0, 1\},
\\
\medskip
\hat \pi\in  L^{r'}(\omega;L^{r'}_{0,\#}(Z_f)).
\end{array}\right.
$$
Proceeding as in {\it Step 3}, we deduce the non-linear 2D Darcy's law of Carreau type (\ref{thm:system_gamma1}),  where the permeability operator $\mathcal{U}:\mathbb{R}^2\to \mathbb{R}^2$ is defined by (\ref{permfunc}). Here, for $\xi'\in\mathbb{R}^2$, $(w_{\xi'}, \pi_{\xi'})\in W^{1,r}_{0,\#}(Z_f)^3\times L^{r'}_{0,\#}(Z_f)$ is the unique solution of the local Stokes system (\ref{LocalProblemNonNewtonian}) with nonlinear viscosity given by the Carreau law (\ref{Carreau}).

\end{proof}

\begin{proof}[Proof of Theorem \ref{mainthmNewtonian}.]
We recall that in this case $r=2$, the Carreau law (\ref{Carreau}) reduces to $\eta_0$. Thus, by linearity, the variational formulation (\ref{form_var_hat}) can be written as follows: 
\begin{equation}\label{v_ineq_Newtonian}
\begin{array}{l}
\displaystyle
\medskip
\epsilon^{\gamma-2}\eta_0\int_{\omega\times Z} \mathbb{D}_{z}[\hat u_\epsilon]: \mathbb{D}_{z}[\varphi]dx'dy-\int_{\omega\times Z}\hat P_\epsilon\, {\rm div}_{x'}\varphi' \,dx'dz=\int_{\omega\times Z} f'\cdot \varphi' \,dx'dz+O_\epsilon,
\end{array}
\end{equation}
 for all $\varphi\in \mathcal{D}(\omega; C^\infty_{\#}(Z)^3)$ with $\varphi=0$ in $\omega\times T$ and on $\omega\times Z'\times \{0, 1\}$, satisfying the divergence conditions ${\rm div}_{x'}\int_Z \varphi'\,dz=0$ in $\omega$, $\int_Z \varphi'\,dz\cdot n=0$ on $\partial\omega$,  and ${\rm div}_{z}\varphi=0$ in $\omega\times Z$.\\
 
Passing to the limit in (\ref{v_ineq_Newtonian}) using convergences (\ref{conv_vel_gorro}) and (\ref{conv_pres_hat}), we take  into account that the second term converges to 
$$\int_{\omega\times Z}\tilde P\, {\rm div}_{x'}\varphi'\,dx'dz.$$
Since $\tilde P$ does not depend on $z$, by the divergence condition ${\rm div}_{x'}\int_Z \varphi'\,dz=0$, we have 
$$\int_{\omega\times Z}\tilde P\, {\rm div}_{x'}\varphi'\,dx'dz=\int_{\omega}\tilde P\, {\rm div}_{x'}\left(\int_Z\varphi'dz\right)\,dx'=0.$$
Therefore, the following relation holds true:
$$\begin{array}{l}
\displaystyle
\medskip
 \eta_0\int_{\omega\times Z} \mathbb{D}_{z}[\hat u]: \mathbb{D}_{z}[\varphi]dx'dz=\int_{\omega\times Z} f'\cdot \varphi' \,dx'dz,
\end{array}
$$
and by a density argument, remains valid for every $\varphi \in \mathcal{V}$ with $\mathcal{V}$ given by (\ref{Vspace}). 

Proceeding similarly as the end of  {\it Step 1} of the proof of Theorem \ref{mainthmPseudo}, this variational formulation is equivalent to  the system (\ref{prueba_lim_1}) with viscosity $\eta_0$, which admits a unique solution $(\hat u, \hat \pi, \tilde P)\in L^2(\omega;H^1_{\#}(Z_f)^3)\times L^{2}(\omega;L^2_{0,\#}(Z_f))\times (L^{2}_0(\omega)\cap H^1(\omega))$. This establishes the convergence of the whole sequence $(\tilde u_{\epsilon},\tilde P_{\epsilon})$. Finally, reasoning as in {\it Step 2} of  the proof of Theorem \ref{mainthmPseudo}, we get the linear effective 2D Darcy's law  (\ref{thm:systemNewtonian}), which concludes the proof of Theorem \ref{mainthmNewtonian}.

\end{proof}

\section{Numerical simulations of the effective models}\label{Section:Numerics}

In this section, we perform a numerical study of the asymptotic behaviour of a flow of a Carreau fluid between two parallel plates, separated by a thin layer of porous medium, as described in Section~\ref{sec:main}. 

 We assume that the flow is driven by a constant  body force $f=(f',0)$ with $f'\in \R^2$, which is a realistic assumption that is used in many applications such as enhanced oil recovery~\cite[Chapter 4]{Bird}. For simplicity, we also assume that $\omega$ is the unit square $\omega=(-1,1)^2$ and impose periodic boundary conditions on $\partial Q_\eps=\partial \omega\times (0,\eps)$. System~\eqref{1} is thus rewritten
\begin{equation}\label{Numerics:System}
\left\{\begin{array}{rl}\displaystyle
\medskip
-\epsilon^\gamma{\rm div}\left(\eta_r(\mathbb{D}[u_\epsilon])\mathbb{D}[u_\epsilon]\right)+ \nabla p_\epsilon = f & \hbox{in }\Omega_\epsilon,\\
\medskip
\displaystyle
{\rm div}\,u_\epsilon=0& \hbox{in }\Omega_\epsilon,\\
\medskip
\displaystyle
u_\epsilon=0& \hbox{on }\partial S_\epsilon,\\
u_\eps(x_1,-1)=u_\eps(x_1,1), & x_1\in (-1,1),\\ 
u_\eps(-1,x_2)=u_\eps(1,x_2), & x_2\in (-1,1),\\ 
\end{array}\right.
\end{equation}
As observed in~\cite[Section 4]{Carreau_Ang_Bonn_SG}, Lemma~\ref{lemma_compactness} needs to be slightly modified in that case: conditions~\eqref{divxproperty} and~\eqref{divyproperty} are respectively replaced by
\begin{equation}\label{divxproperty-periodic}
{\rm div}_{x'}\left(\int_0^1\tilde u'(x',z_3)\,dz_3\right)=0\quad  \hbox{in }\omega,
\end{equation}
\begin{equation}\label{divyproperty-periodic}
{\rm div}_{z}\,\hat u(x',z)=0\  \hbox{ in }\omega\times Z_f,\quad   {\rm div}_{x'}\left(\int_{Z_f}\hat u'(x',z)\,dz\right)=0\quad  \hbox{in }\omega.
\end{equation}
As a result, by the periodicity hypothesis on the flow, the boundary condition $\tilde V'\cdot n' = 0$ does not hold anymore on $\partial \omega$. Hence, in this particular configuration, we will discuss numerical simulations of Darcy's laws of the form as~\eqref{thm:system}, \eqref{thm:system_gamma1}, \eqref{thm:system_22} and \eqref{thm:systemNewtonian}, but without the aforementioned boundary condition on $\tilde V'$.

Since $f'$ is constant, one gets that $\tilde P\equiv 0$ and $\tilde V'$ is also a constant vector in $\R^2$.

\paragraph{Choice of rheological parameters. } Let us spectify the range of parameters that we use in the numerical tests. In addition to the exponent $\gamma\in \R$, the model~\eqref{Numerics:System} depends on four rheological parameters: $\eta_0, \eta_{\infty}, \lambda$ and $r$. Since in many applications (see for instance~\cite{Bird}), $\eta_{\infty}$ is very small compared with $\eta_0$, we arbitrarily fix $\eta_0=1$ and $\eta_\infty=10^{-3}$. As regards $\lambda$, we take $\lambda\in \{1,10,100\}$. In the case where the effective model is nonlinear with respect to $\lambda$, the possibility of multiplying $\lambda$ by a factor $10$ from one simulation to another will give us access to a large panel of behaviours for the effective models. Finally, we consider a pseudoplastic case $r=1.7$, the Newtonian case $r=2$ and two dilatant cases $r=2.3$ and $r=2.6$.


\paragraph{Shapes of inclusions $T$.} In order to illustrate the effect of a change of volume of the inclusion $T$, and the effect of anisotropy, we will consider four possible shapes for $T'$: two disks of respective radius $0.1$ and $0.3$, and the ellipse of semi-major axis $0.3$ and semi-minor axis $0.1$, parallel to the $x$ or the $y$ axis (see Fig.~\ref{Fig:ellipse-inclusions}). These shapes will be numbered $E_1,E_2,E_3$ and $E_4$ in the rest of this section.

\begin{figure}
	\begin{center}
		\begin{tabular}{cc}
			Shape $E_1$ & Shape $E_2$\\
			\includegraphics[height=0.25\textwidth]{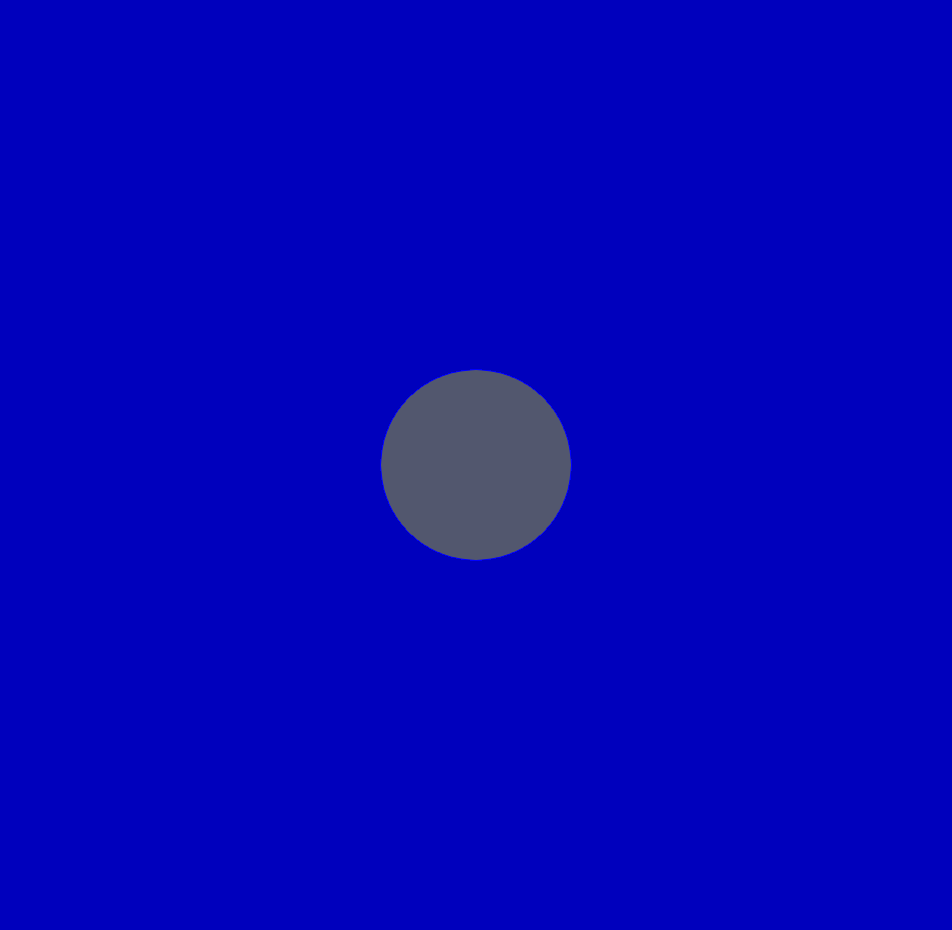}
			&
			\includegraphics[height=0.25\textwidth]{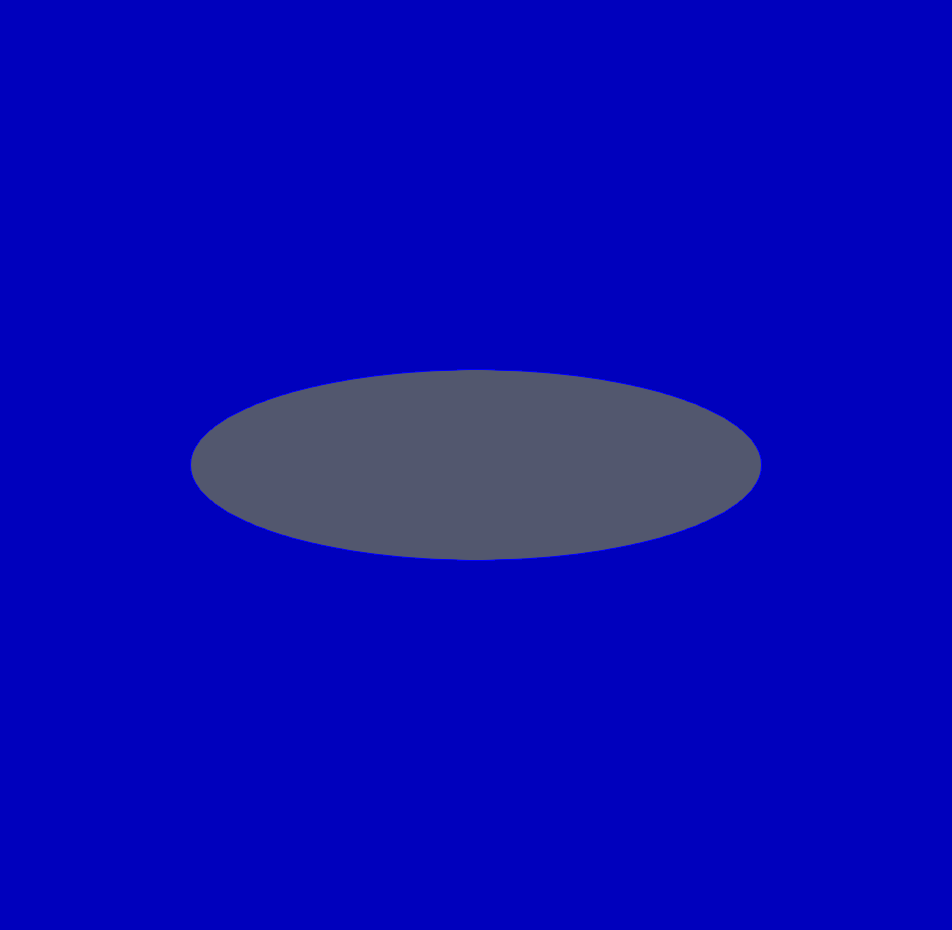}
			\\
			Shape $E_3$& Shape $E_4$\\
			\includegraphics[height=0.25\textwidth]{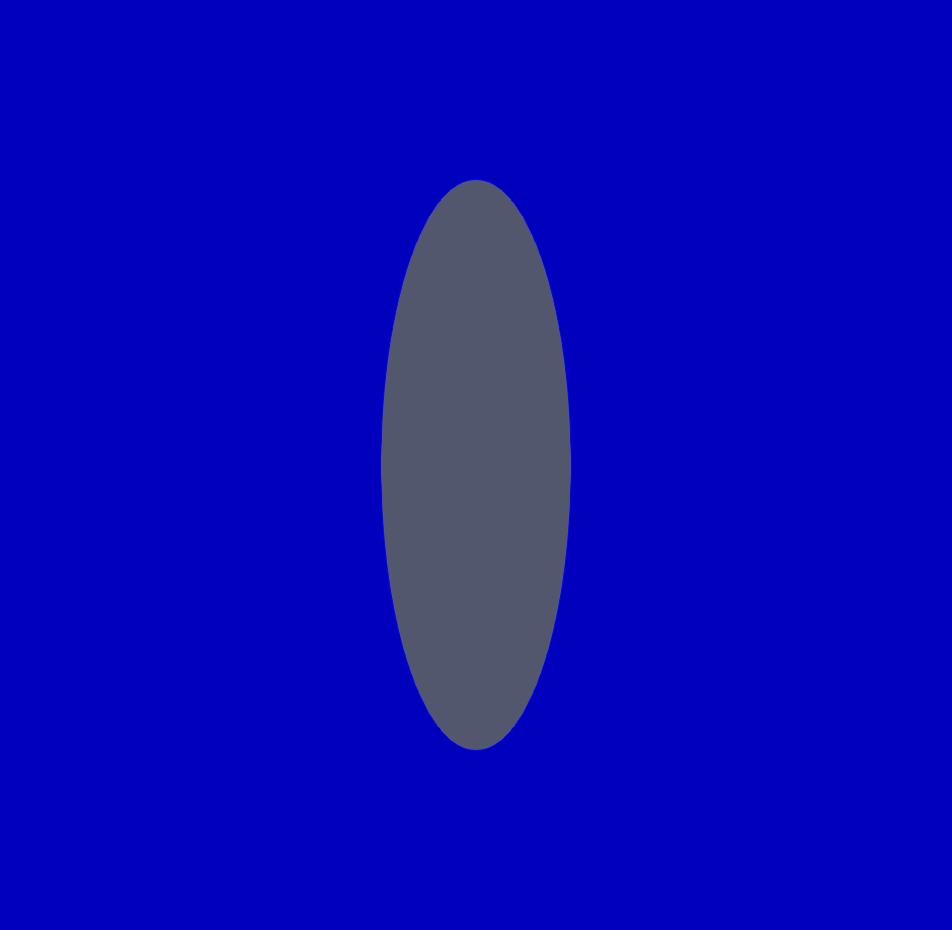}
			&
			\includegraphics[height=0.25\textwidth]{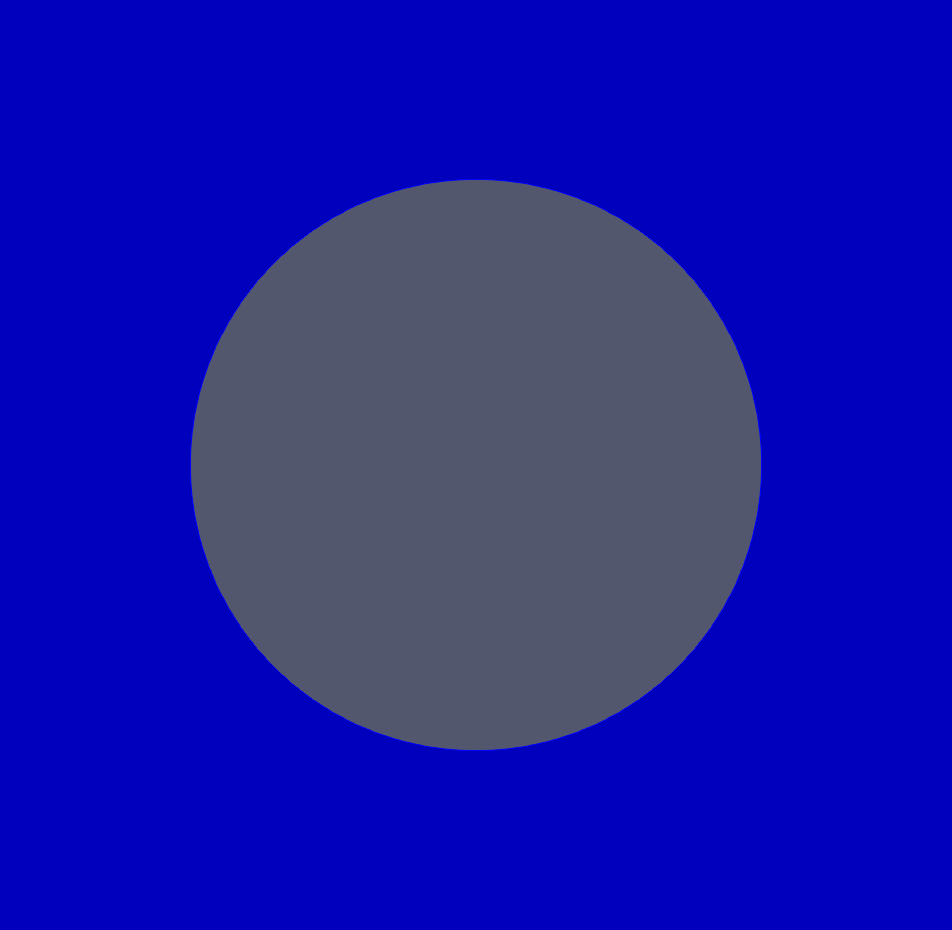}
		\end{tabular}
	\end{center}
	\caption{Representation of the $2D$ reference cells $Z'$ used in the numerical simulations of the effective models. The inclusion $T'$ (in grey) is surrounded by $Z_f'$ (in blue). From left to right and top to bottom: disk or radius $0.1$, ellipses of semi-major axis $0.3$ and semi-minor axis $0.1$, oriented respectively along the $x$ and $y$ directions, and disk of radius $0.3$. These shapes are numbered $E_1,E_2,E_3$ and $E_4$. }	
	\label{Fig:ellipse-inclusions}
\end{figure}

\paragraph{Numerical resolution of the cell problems.} The solution of each cell problem of the form~\eqref{LocalProblemNewtonian} or~\eqref{LocalProblemNonNewtonian} is computed using a mixed formulation, that we solve by a finite element method, using FreeFem++ software \cite{FREEFEM}. In the nonlinear cases~\eqref{LocalProblemNonNewtonian} where the viscosity $\eta_r$ follows a Carreau law or a power law, we rely on a fixed point algorithm (see for instance~\cite[Section 2.8]{Saramito}). We consider the Taylor-Hood approximation for the velocity-pressure pair, namely $P_2$ elements for the velocity field and $P_1$ elements for the pressure. This choice is well known to be compatible with the Babu\v{s}ka-Brezzi condition~\cite{GR86}. Each three-dimensional mesh of a cell $Z_f$ is obtained by constrained Delaunay tetrahedralization, and contains approximately $8000$ tetrahedra.

\subsection{Permeability tensor $\mathcal{A}$}

In the cases where the effective system is described by a linear $2D$ Darcy law, the response of the fluid to a constant pressure gradient $f'$ takes the form $\tilde V'=\frac{1}{\eta}\mathcal{A}f'$, where the constant viscosity $\eta$ is either equal to $\eta_0$ or $\eta_\infty$, depending on the values of $r$ and $\gamma$ (see Table~\ref{table_asymp1}). Hence, the asymptotic behaviour of the fluid is encoded in the permeability tensor $\mathcal A$.

In order to highlight certain properties of $\mathcal A$, we have summarized in Figure~\ref{Fig:permeabilityA} the coefficients that we obtain numerically, for the different shape geometries $E_1$ to $E_4$. We first notice that $\mathcal A$ is perfectly symmetric, which comes from its very definition in the continuous setting. Indeed, testing against $w^j$ in system~\eqref{LocalProblemNewtonian} satisfied by $w^i$, or against $w^i$ in the same system but satisfied by $w^j$, we obtain
\begin{align*}
\int_{Z_f}w_i^j(z)\, dz & = \int_{Z_f}D_z w^i(z)  : D_z w^j(z)\, dz 
 = \int_{Z_f}D_z w^j(z)  : D_z w^i(z)\, dz
 = \int_{Z_f} w^i_j(z)\, dz\, , 
\end{align*}
hence $\mathcal{A}_{i,j}=\mathcal{A}_{j,i}$.

One can also observe that, for isotropic inclusions $E_1$ and $E_4$, up to numerical errors, the tensor $\mathcal A$ is a diagonal matrix of the form $a\mathrm{I}$. This means that, as expected, the filtration velocity $\tilde V'$ is simply given by the product between $f'$ and the positive constant $a$ for such geometries. Also, the value of $a$ appears to be a decreasing function of the area of the obstacle, which is quite intuitive as well. In the case of anisotropic geometries $E_2$ and $E_3$, $\mathcal A$ is still diagonal, but its diagonal coefficients are not equal. Since $E_3$ is obtained by applying a rotation of angle $\pi/2$ to $E_2$, by symmetry, the associated matrix $\mathcal A$ is the transposed of the one associated with shape $E_2$.

However, for more general geometries of inclusions, the permeability tensor $\mathcal A$ is no longer symmetric. For instance, we have given in Fig.\!~\ref{Fig:permeabilityArotation} the coefficients of $\mathcal A$ computed when $E_2$ is rotated by an angle $\theta\in \{\pi/16, \pi/8, \pi/4\}$. As expected, the diagonal coefficients $\mathcal A_{1,2}$, $\mathcal A_{2,1}$ are not equal to zero in such configurations.

\begin{figure}
\begin{center}
\begin{tabular}{|c|c|}
	\hline
	Inclusion & Permeability tensor $\mathcal A$\\
	\hline\hline
	$E_1$ & $\begin{pmatrix}  0.0697955& -3.17061\times 10^{-6} \\-3.17061\times 10^{-6} & 0.0697947
	  \end{pmatrix}$	\\
	  \hline
	 $E_2$ & $\begin{pmatrix}  0.054708& -2.3436\times 10^{-6} \\-2.3436\times 10^{-6} & 0.0210978
	 \end{pmatrix}$	\\ 
	 \hline
	 $E_3$ & $\begin{pmatrix}  0.0211078& -5.69438\times 10^{-7} \\-5.69438\times 10^{-7} & 0.0547038
	\end{pmatrix}$	\\ 
	\hline
	 $E_4$ & $\begin{pmatrix}  0.0153292& -3.7408\times 10^{-7} \\-3.7408\times 10^{-7} & 0.0153284
	\end{pmatrix}$	\\
	\hline
\end{tabular}
\end{center}
\caption{Permeability tensor $\mathcal A$ computed numerically in the case of an elliptic inclusion ($E_1$ to $E_4$).}\label{Fig:permeabilityA}
\end{figure}


\begin{figure}
	\begin{center}
		\begin{tabular}{|c|c|}
			\hline
			Rotation angle $\theta$ & Permeability tensor $\mathcal A$\\
			\hline\hline
			$\pi/16$ & $\begin{pmatrix}  0.0534164& 0.00341334 \\0.00341334 & 0.0225729
			\end{pmatrix}$	\\
			\hline
			$\pi/8$ & $\begin{pmatrix}  0.0498291& 0.00653147 \\0.00653147 & 0.0266649
			\end{pmatrix}$ 	\\ 
			\hline
			$\pi/4$ & $\begin{pmatrix}  0.0385604& 0.00963438 \\0.00963438 & 0.0385636
			\end{pmatrix}$ 	\\ 
			\hline
		\end{tabular}
	\end{center}
	\caption{Permeability tensor $\mathcal A$ computed numerically in the case where the $E_2$ inclusion is rotated of an angle $\theta$ with respect to the $x_1$ axis.}\label{Fig:permeabilityArotation}
\end{figure}

\subsection{Carreau law ($\gamma=1$)}

In case $\gamma=1$, the behaviour of the effective model associated with a pseudoplastic fluid has been studied numerically in~\cite{Carreau_Ang_Bonn_SG}. In this subsection, we complete these numerical results by considering dilatant fluids as well. In order to perform comparisons, we use the same parameters as in~\cite{Carreau_Ang_Bonn_SG}, namely $\eta_0=1$, $\eta_{\infty}=10^{-3}$, and $\lambda\in \{1,10,100\}$. We consider dilatant fluids with $r\in \{2.3,2.6\}$, a Newtonian fluid ($r=2$) and a pseudoplastic fluid ($r=1.7$).

We explore numerically the influence of the amplitude of $f'$ and of its orientation, on the computed value of the filtration velocity $\tilde V'$.

\subsubsection{Influence of the amplitude of the pressure gradient $f'$}

We impose an exterior force $f'$ directed by $e_1$, \emph{i.e.}\! of the form $f'=(f_1,0)$, with $f_1\in [0,1]$. In that case, the computed filtration velocity $\tilde V'$ is also directed by $e_1$ and reads $\tilde V'=(\tilde V_1,0)$. Fig.~\ref{Fig:amplitude_Carreau} represents  $\tilde V_1$ as a function of $f_1$, for the different obstacle shapes $E_1$ to $E_4$ and the choice of parameters $r,\lambda$ specified above.

We observe that, for any choice of $r$ and $\lambda$, $\tilde V_1$ is an increasing function of $f_1$. However, for $r=1.7$ (pseudoplastic case), $\tilde V_1$ appears as a convex function of $f_1$, whereas for $r>2$ (dilatant case), it is a concave function of $f_1$. For $r=2$ (Newtonian case), the dependency on $f_1$ is linear, as expected. Also, the separation between the curves gets more pronounced as $\lambda$ increases, which comes from the fact that $\lambda$ is the coefficient in front of the nonlinear term $|D(u)|^{r-2}$ in the definition of the viscosity following the Carreau law~\eqref{Carreau}. For any values of $\lambda$, and any geometry of obstacle shape $E_1$ to $E_4$, for a given pressure gradient $f'$, the amplitude of the filtration velocity $\tilde V'$ diminishes as the exponent $r$ increases. This is consistent with the fact that, for high values of the shear rate, the viscosity of the dilatant fluid increases, whereas the behaviour of pseudoplastic fluids is the opposite.


\begin{figure}
	\begin{center}
	\begin{tabular}{rrr}
		\begin{tikzpicture}[scale=0.65]
		\begin{axis}[ 
		xlabel={$f_1$}, ylabel={$V_1$},
		title={$\lambda=1$},
		legend entries={$r=1.7$,$r=2$,$r=2.3$,$r=2.6$}, legend style={at={(0.02,0.98)}, anchor=north west}
		]
		\addplot+[mark=none] table[x index=1, y index=4]{Data-Carreau-amplitude_f-lambda=1-r=1.7_E1.txt};		
		\addplot+[mark=none] table[x index=1, y index=4]{Data-Carreau-amplitude_f-lambda=1-r=2_E1.txt};
		\addplot+[mark=none] table[x index=1, y index=4]{Data-Carreau-amplitude_f-lambda=1-r=2.3_E1.txt};
		\addplot+[mark=none] table[x index=1, y index=4]{Data-Carreau-amplitude_f-lambda=1-r=2.6_E1.txt};
		\end{axis}
		\end{tikzpicture}
		&
		\begin{tikzpicture}[scale=0.65]
\begin{axis}[ 
xlabel={$f_1$}, 
title={$\lambda=10$},
]
\addplot+[mark=none] table[x index=1, y index=4]{Data-Carreau-amplitude_f-lambda=10-r=1.7_E1.txt};		
\addplot+[mark=none] table[x index=1, y index=4]{Data-Carreau-amplitude_f-lambda=10-r=2_E1.txt};
\addplot+[mark=none] table[x index=1, y index=4]{Data-Carreau-amplitude_f-lambda=10-r=2.3_E1.txt};
\addplot+[mark=none] table[x index=1, y index=4]{Data-Carreau-amplitude_f-lambda=10-r=2.6_E1.txt};
\end{axis}
\end{tikzpicture}
&
		\begin{tikzpicture}[scale=0.65]
		\begin{axis}[ 
		xlabel={$f_1$}, 
		title={$\lambda=100$},
		]
		\addplot+[mark=none] table[x index=1, y index=4]{Data-Carreau-amplitude_f-lambda=100-r=1.7_E1.txt};		
		\addplot+[mark=none] table[x index=1, y index=4]{Data-Carreau-amplitude_f-lambda=100-r=2_E1.txt};
		\addplot+[mark=none] table[x index=1, y index=4]{Data-Carreau-amplitude_f-lambda=100-r=2.3_E1.txt};
		\addplot+[mark=none] table[x index=1, y index=4]{Data-Carreau-amplitude_f-lambda=100-r=2.6_E1.txt};
		\end{axis}
		\end{tikzpicture}
		\\
				\begin{tikzpicture}[scale=0.65]
		\begin{axis}[ 
		xlabel={$f_1$}, ylabel={$V_1$},
		legend entries={$r=1.7$,$r=2$,$r=2.3$,$r=2.6$}, legend style={at={(0.02,0.98)}, anchor=north west}
		]
		\addplot+[mark=none] table[x index=1, y index=4]{Data-Carreau-amplitude_f-lambda=1-r=1.7_E2.txt};		
		\addplot+[mark=none] table[x index=1, y index=4]{Data-Carreau-amplitude_f-lambda=1-r=2_E2.txt};
		\addplot+[mark=none] table[x index=1, y index=4]{Data-Carreau-amplitude_f-lambda=1-r=2.3_E2.txt};
		\addplot+[mark=none] table[x index=1, y index=4]{Data-Carreau-amplitude_f-lambda=1-r=2.6_E2.txt};
		\end{axis}
		\end{tikzpicture}
		&
		\begin{tikzpicture}[scale=0.65]
		\begin{axis}[ 
		xlabel={$f_1$}, 
		]
		\addplot+[mark=none] table[x index=1, y index=4]{Data-Carreau-amplitude_f-lambda=10-r=1.7_E2.txt};		
		\addplot+[mark=none] table[x index=1, y index=4]{Data-Carreau-amplitude_f-lambda=10-r=2_E2.txt};
		\addplot+[mark=none] table[x index=1, y index=4]{Data-Carreau-amplitude_f-lambda=10-r=2.3_E2.txt};
		\addplot+[mark=none] table[x index=1, y index=4]{Data-Carreau-amplitude_f-lambda=10-r=2.6_E2.txt};
		\end{axis}
		\end{tikzpicture}
		&
		\begin{tikzpicture}[scale=0.65]
		\begin{axis}[ 
		xlabel={$f_1$}, 
		]
		\addplot+[mark=none] table[x index=1, y index=4]{Data-Carreau-amplitude_f-lambda=100-r=1.7_E2.txt};		
		\addplot+[mark=none] table[x index=1, y index=4]{Data-Carreau-amplitude_f-lambda=100-r=2_E2.txt};
		\addplot+[mark=none] table[x index=1, y index=4]{Data-Carreau-amplitude_f-lambda=100-r=2.3_E2.txt};
		\addplot+[mark=none] table[x index=1, y index=4]{Data-Carreau-amplitude_f-lambda=100-r=2.6_E2.txt};
		\end{axis}
		\end{tikzpicture}
		\\
\begin{tikzpicture}[scale=0.65]
\begin{axis}[ 
xlabel={$f_1$}, ylabel={$V_1$},
legend entries={$r=1.7$,$r=2$,$r=2.3$,$r=2.6$}, legend style={at={(0.02,0.98)}, anchor=north west}
]
\addplot+[mark=none] table[x index=1, y index=4]{Data-Carreau-amplitude_f-lambda=1-r=1.7_E3.txt};		
\addplot+[mark=none] table[x index=1, y index=4]{Data-Carreau-amplitude_f-lambda=1-r=2_E3.txt};
\addplot+[mark=none] table[x index=1, y index=4]{Data-Carreau-amplitude_f-lambda=1-r=2.3_E3.txt};
\addplot+[mark=none] table[x index=1, y index=4]{Data-Carreau-amplitude_f-lambda=1-r=2.6_E3.txt};
\end{axis}
\end{tikzpicture}
&
\begin{tikzpicture}[scale=0.65]
\begin{axis}[ 
xlabel={$f_1$}, 
]
\addplot+[mark=none] table[x index=1, y index=4]{Data-Carreau-amplitude_f-lambda=10-r=1.7_E3.txt};		
\addplot+[mark=none] table[x index=1, y index=4]{Data-Carreau-amplitude_f-lambda=10-r=2_E3.txt};
\addplot+[mark=none] table[x index=1, y index=4]{Data-Carreau-amplitude_f-lambda=10-r=2.3_E3.txt};
\addplot+[mark=none] table[x index=1, y index=4]{Data-Carreau-amplitude_f-lambda=10-r=2.6_E3.txt};
\end{axis}
\end{tikzpicture}
&
\begin{tikzpicture}[scale=0.65]
\begin{axis}[ 
xlabel={$f_1$}, 
]
\addplot+[mark=none] table[x index=1, y index=4]{Data-Carreau-amplitude_f-lambda=100-r=1.7_E3.txt};		
\addplot+[mark=none] table[x index=1, y index=4]{Data-Carreau-amplitude_f-lambda=100-r=2_E3.txt};
\addplot+[mark=none] table[x index=1, y index=4]{Data-Carreau-amplitude_f-lambda=100-r=2.3_E3.txt};
\addplot+[mark=none] table[x index=1, y index=4]{Data-Carreau-amplitude_f-lambda=100-r=2.6_E3.txt};
\end{axis}
\end{tikzpicture}
\\
\begin{tikzpicture}[scale=0.65]
\begin{axis}[ 
xlabel={$f_1$}, ylabel={$V_1$},
legend entries={$r=1.7$,$r=2$,$r=2.3$,$r=2.6$}, legend style={at={(0.02,0.98)}, anchor=north west}
]
\addplot+[mark=none] table[x index=1, y index=4]{Data-Carreau-amplitude_f-lambda=1-r=1.7_E4.txt};		
\addplot+[mark=none] table[x index=1, y index=4]{Data-Carreau-amplitude_f-lambda=1-r=2_E4.txt};
\addplot+[mark=none] table[x index=1, y index=4]{Data-Carreau-amplitude_f-lambda=1-r=2.3_E4.txt};
\addplot+[mark=none] table[x index=1, y index=4]{Data-Carreau-amplitude_f-lambda=1-r=2.6_E4.txt};
\end{axis}
\end{tikzpicture}
&
\begin{tikzpicture}[scale=0.65]
\begin{axis}[ 
xlabel={$f_1$}, 
]
\addplot+[mark=none] table[x index=1, y index=4]{Data-Carreau-amplitude_f-lambda=10-r=1.7_E4.txt};		
\addplot+[mark=none] table[x index=1, y index=4]{Data-Carreau-amplitude_f-lambda=10-r=2_E4.txt};
\addplot+[mark=none] table[x index=1, y index=4]{Data-Carreau-amplitude_f-lambda=10-r=2.3_E4.txt};
\addplot+[mark=none] table[x index=1, y index=4]{Data-Carreau-amplitude_f-lambda=10-r=2.6_E4.txt};
\end{axis}
\end{tikzpicture}
&
\begin{tikzpicture}[scale=0.65]
\begin{axis}[ 
xlabel={$f_1$}, 
]
\addplot+[mark=none] table[x index=1, y index=4]{Data-Carreau-amplitude_f-lambda=100-r=1.7_E4.txt};		
\addplot+[mark=none] table[x index=1, y index=4]{Data-Carreau-amplitude_f-lambda=100-r=2_E4.txt};
\addplot+[mark=none] table[x index=1, y index=4]{Data-Carreau-amplitude_f-lambda=100-r=2.3_E4.txt};
\addplot+[mark=none] table[x index=1, y index=4]{Data-Carreau-amplitude_f-lambda=100-r=2.6_E4.txt};
\end{axis}
\end{tikzpicture}
	\end{tabular}
	\end{center}
\caption{Case $\gamma=1$ (Carreau law). Component $V_1$ of the mean filtration velocity $V'$ plotted against $f_1$, with $f'=(f_1,0)$, for $r\in \{1.7,2,2.3,2.6\}$, in the case of elliptic inclusions $E_1$ (first line), $E_2$ (second line), $E_3$ (third line) and $E_4$ (fourth line). From left to right: $\lambda=1, \lambda=10, \lambda=100$.}\label{Fig:amplitude_Carreau}
\end{figure}
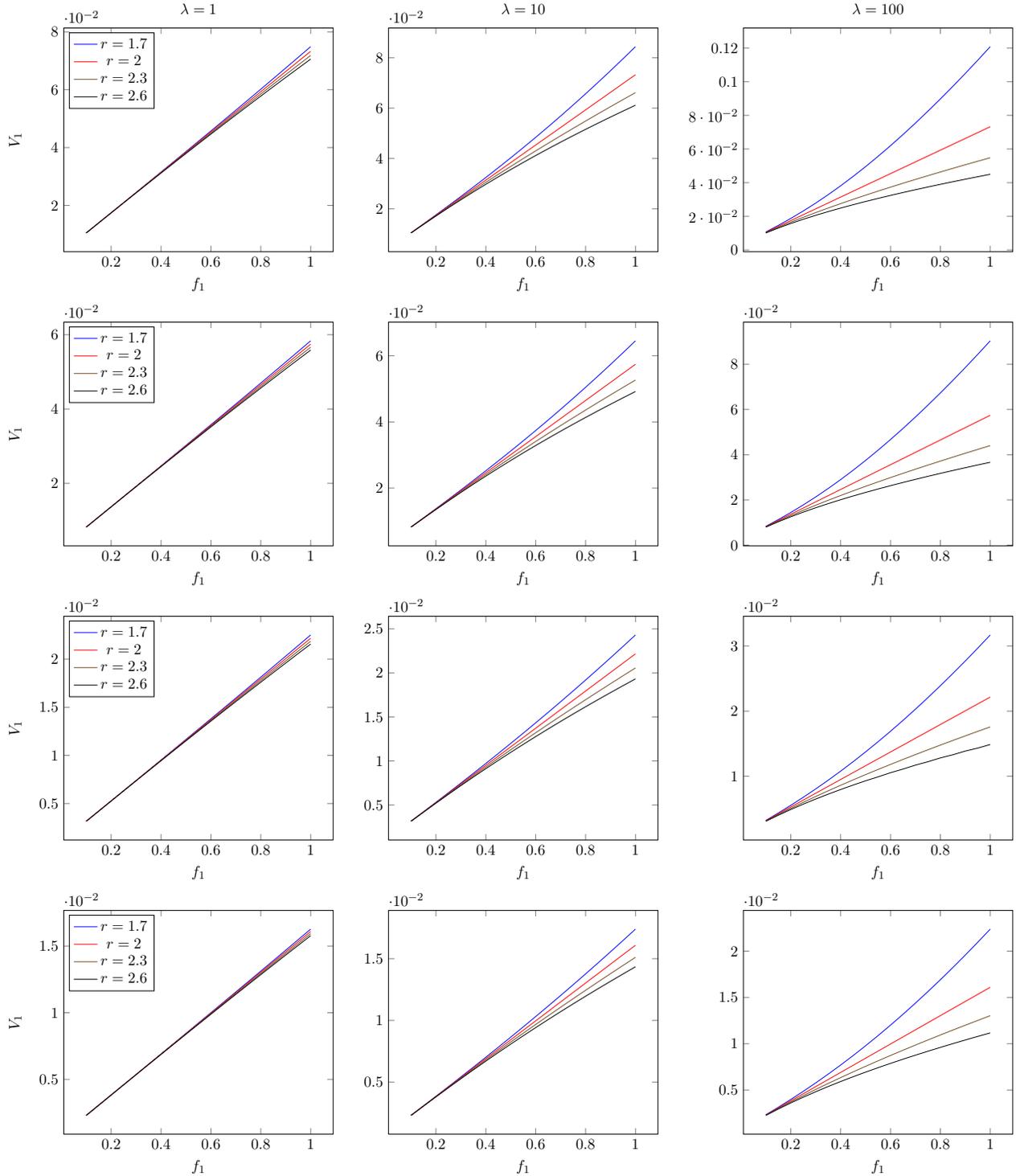

\subsubsection{Influence of the orientation of $f'$}

In order to test the impact of a rototion of $f'$ on the behaviour of the effective system, we consider the anisotropic shape $E_2$ and a family of pressure gradients $f'=(\cos \theta,\sin \theta)$, with the angle $\theta\in [0,\pi/2]$. The results that we obtain are represented in Fig.~\ref{Fig:rotation_Carreau}.

We notice that the orientation of $V'$, which is generally not parallel to the pressure gradient $f'$ due to the anisotropy of the obstacle, does not appear to depend on the rheological parameters $r, \lambda$. In all the simulated configurations, the orientation of $V'$ remains very close to what is observed in the Newtonian case $r=2$. As regards the amplitude of $V'$, as observed in the previous paragraph, there is almost no observable effect for $\lambda=1$. This can be explained by the fact that, for small values of $\lambda$, and an imposed pressure gradient of fixed size $|f'|=1$, the viscosity described by the Carreau law is close to the constant value $\eta_r=\eta_0$ corresponding to the Newtonian case $r=2$. On the contrary, for $\lambda=100$, the amplitude of $V'$ is noticeably reduced as $r$ increases: for instance, for $r=1.7$, the maximal filtration velocity is about $0.085$ while it reaches only about $0.035$ for $r=2.6$.



\begin{figure}
	\begin{center}
		\begin{tabular}{ll}
			\begin{tikzpicture}[scale=0.65]
			\begin{axis}[title={$\lambda=1,\quad r=1.7$}, xlabel=$f_1$,ylabel=$f_2$,colorbar]
			\addplot[point meta = {\thisrow{Vnorm})},quiver={u=\thisrow{Vm1norm},v=\thisrow{Vm2norm},scale arrows = 0.2},-stealth,line width=1.5pt,quiver/colored = {mapped color},] table {Data-Carreau-rotation_f-lambda=1-r=1.7.txt};
			\end{axis}
			\end{tikzpicture}&	
			\begin{tikzpicture}[scale=0.65]
			\begin{axis}[title={$\lambda=100,\quad r=1.7$}, xlabel=$f_1$,ylabel=$f_2$,colorbar]
			\addplot[point meta = {\thisrow{Vnorm})},quiver={u=\thisrow{Vm1norm},v=\thisrow{Vm2norm},scale arrows = 0.2},-stealth,line width=1.5pt,quiver/colored = {mapped color},] table {Data-Carreau-rotation_f-lambda=100-r=1.7.txt};
			\end{axis}
			\end{tikzpicture}\\			
			\begin{tikzpicture}[scale=0.65]
			\begin{axis}[title={$\lambda=1,\quad r=2$}, xlabel=$f_1$,ylabel=$f_2$,colorbar]
			\addplot[point meta = {\thisrow{Vnorm})},quiver={u=\thisrow{Vm1norm},v=\thisrow{Vm2norm},scale arrows = 0.2},-stealth,line width=1.5pt,quiver/colored = {mapped color},] table {Data-Carreau-rotation_f-lambda=1-r=2.txt};	
			\end{axis}	
			\end{tikzpicture}&
			\begin{tikzpicture}[scale=0.65]
			\begin{axis}[title={$\lambda=100,\quad r=2$}, xlabel=$f_1$,ylabel=$f_2$,colorbar]
			\addplot[point meta = {\thisrow{Vnorm})},quiver={u=\thisrow{Vm1norm},v=\thisrow{Vm2norm},scale arrows = 0.2},-stealth,line width=1.5pt,quiver/colored = {mapped color},] table {Data-Carreau-rotation_f-lambda=100-r=2.txt};	
			\end{axis}	
			\end{tikzpicture}
			\\
			\begin{tikzpicture}[scale=0.65]
			\begin{axis}[title={$\lambda=1, \quad r=2.3$}, xlabel=$f_1$,ylabel=$f_2$,colorbar]
			\addplot[point meta = {\thisrow{Vnorm})},quiver={u=\thisrow{Vm1norm},v=\thisrow{Vm2norm},scale arrows = 0.2},-stealth,line width=1.5pt,quiver/colored = {mapped color},] table {Data-Carreau-rotation_f-lambda=1-r=2.3.txt};	
			\end{axis}	
			\end{tikzpicture}
			&
			\begin{tikzpicture}[scale=0.65]
			\begin{axis}[title={$\lambda=100, \quad r=2.3$}, xlabel=$f_1$,ylabel=$f_2$,colorbar]
			\addplot[point meta = {\thisrow{Vnorm})},quiver={u=\thisrow{Vm1norm},v=\thisrow{Vm2norm},scale arrows = 0.2},-stealth,line width=1.5pt,quiver/colored = {mapped color},] table {Data-Carreau-rotation_f-lambda=100-r=2.3.txt};	
			\end{axis}	
			\end{tikzpicture}
			\\
			\begin{tikzpicture}[scale=0.65]
			\begin{axis}[title={$\lambda=1, \quad r=2.6$}, xlabel=$f_1$,ylabel=$f_2$,colorbar]
			\addplot[point meta = {\thisrow{Vnorm})},quiver={u=\thisrow{Vm1norm},v=\thisrow{Vm2norm},scale arrows = 0.2},-stealth,line width=1.5pt,quiver/colored = {mapped color},] table {Data-Carreau-rotation_f-lambda=1-r=2.6.txt};	
			\end{axis}	
			\end{tikzpicture}&
			\begin{tikzpicture}[scale=0.65]
			\begin{axis}[title={$\lambda=100, \quad r=2.6$}, xlabel=$f_1$,ylabel=$f_2$,colorbar]
			\addplot[point meta = {\thisrow{Vnorm})},quiver={u=\thisrow{Vm1norm},v=\thisrow{Vm2norm},scale arrows = 0.2},-stealth,line width=1.5pt,quiver/colored = {mapped color},] table {Data-Carreau-rotation_f-lambda=100-r=2.6.txt};	
			\end{axis}	
			\end{tikzpicture}
		\end{tabular}
	\end{center}
	\caption{Case $\gamma=1$ (Carreau law). Representation of $\tilde V'$ when $f'$ is a unit vector of the form $(f_1,f_2)=(\cos \theta,\sin \theta)$ with $\theta\in [0,\pi/2]$, in the case of an elliptic inclusion $E_2$. Each vector $\tilde V'$ is represented by a vector of length $0.2$, localized at point $(f_1,f_2)$ and colored according to $|\tilde V'|$. The left column corresponds to $\lambda=1$ and the right one to $\lambda=100$. Each line from top to bottom corresponds respectively to $r=1.7$, $r=2$, $r=2.3$ and $r=2.6$.}
	\label{Fig:rotation_Carreau}
\end{figure}
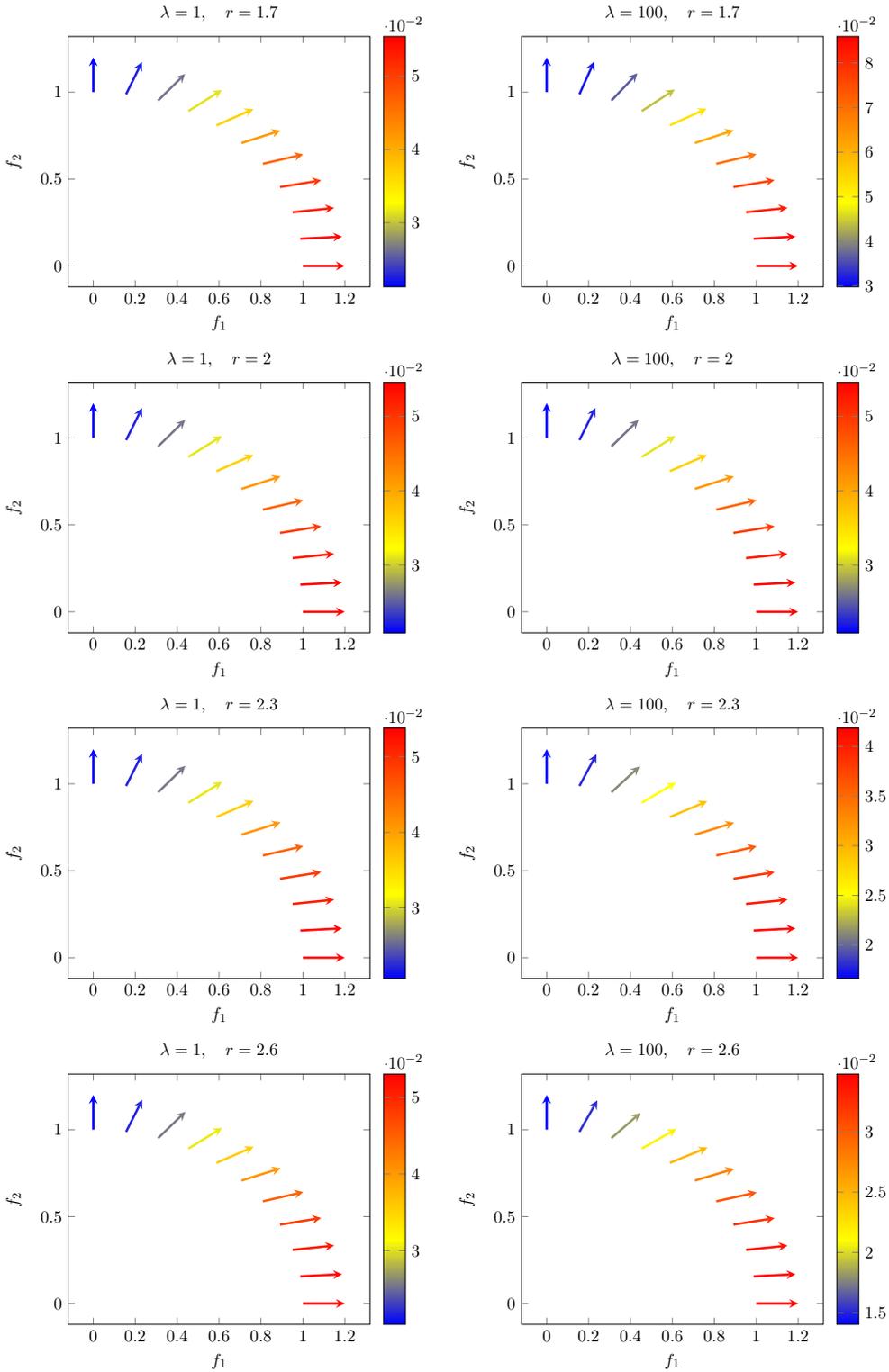

\subsection{Power law ($r>2$, $\gamma>1$)}

In order to allow for comparisons, we perform similar simulations for the power law regime ($r>2$, $\gamma>1$) as we did for the Carreau regime ($\gamma=1$), taking $\lambda\in \{1,10,100\}$ and $r\in \{2.3,2.6\}$, to separately test the influence of the amplitude of $f'$, and of the angle formed by $f'$ and the $x_1$ axis. In each case, we provide the Newtonian behaviour corresponding to $r=2$, as a reference model.

\subsubsection{Influence of the amplitude of $f'$}

We have plotted in Fig.~\ref{Fig:amplitude_power} the horizontal component of $\tilde V'$ as a function of $f_1$, when $f'$ takes the form $f'=(f_1,0)$. Contrary to what we can observe in Fig.~\ref{Fig:amplitude_Carreau}, for a given value of $f_1$, the order between the computed value of $\tilde V_1$ for $r$ in $\{2,2.3,2.6\}$ depends on the choice of $\lambda$: the behaviour of the effective system is no longer monotonous with respect to $r$. In particular, for $\lambda=100$, the different curves intersect for a specific value of $f_1$, which seems to be unique and depends on the shape of the obstacle. For instance, the $f_1$ component of this intersection point is between $0.3$ and $0.4$ for $E_1$ and close to $0.5$ for $E_4$. Moreover, when $f_1$ exceeds this value, we recover a behaviour that is very similar to what appeared in Fig.\!~\ref{Fig:amplitude_Carreau}: for $r\in \{2.3,2.6\}$, $V_1$ is an increasing concave function of $f_1$, and $V_1$ is smaller for $r=2.6$ than for $r=2.3$. 

We may interpret these features as follows. When $|f'|$ is small and $\eta_r$ follows the power law $\eta_r(\mathbb{D}_z[w_{\xi'}])=|\mathbb{D}_z[w_{\xi'}]|^{r-2}$, the deformation rate tensor $D_z[w_{f'}]$ associated to the solution $w_{f'}$ of system~\eqref{LocalProblemNonNewtonian} (with $\xi'=f'$) also has a small amplitude. Thus, as mentioned in the Introduction, the power law is not well-suited to capture the behaviour of a quasi-Newtonian fluid in such regime. On the opposite, for high values of parameter $\lambda$ such as $\lambda=100$ in our simulations, and $f_1$ large enough, the qualitative behaviour of the nonlinear $2D$ Darcy law associated with the Carreau law or with the power law become very similar, since the Carreau law behaves as a power law for large values of the deformation rate.


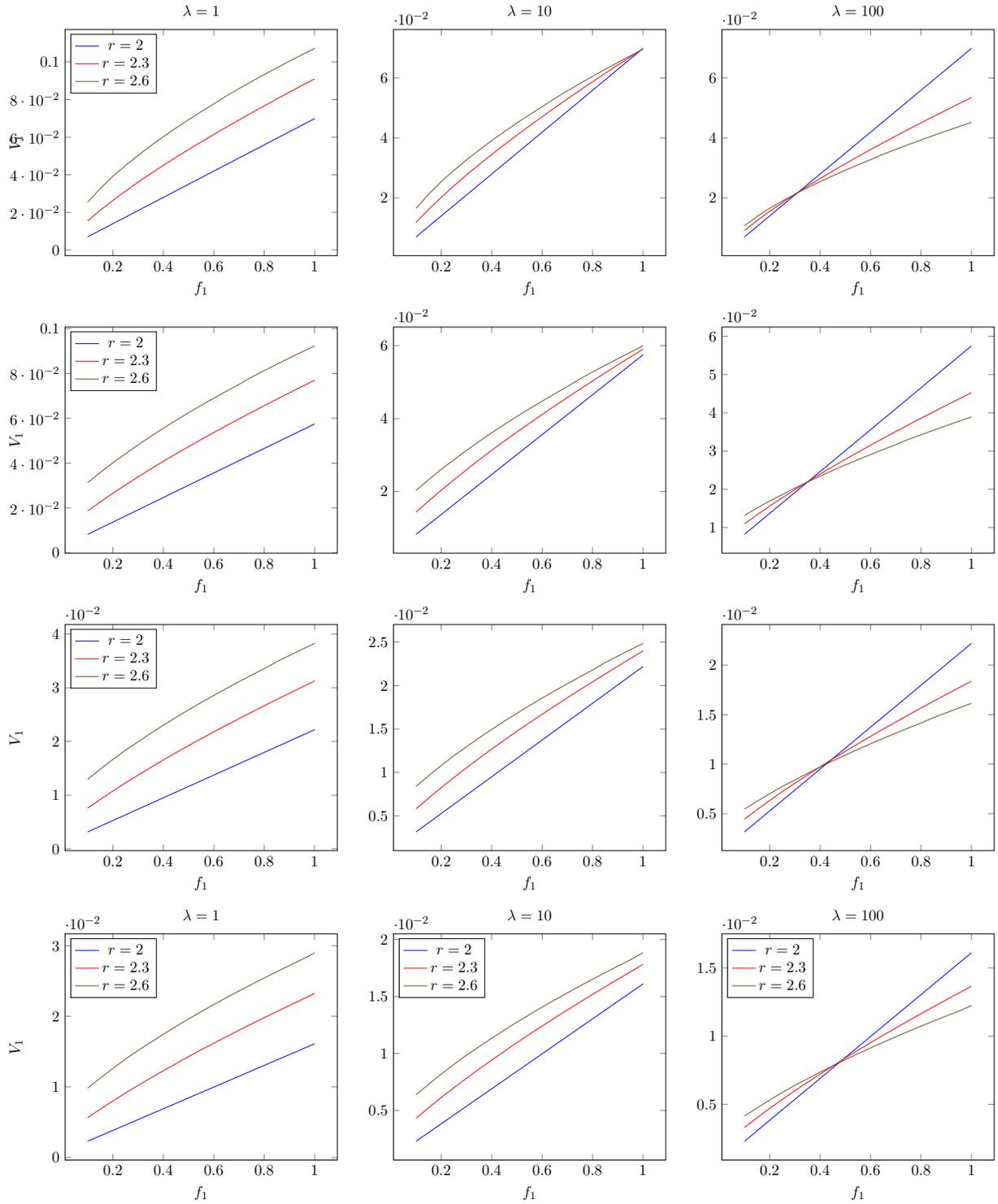
\begin{figure}
	\begin{center}
		\begin{tabular}{rrr}
			\begin{tikzpicture}[scale=0.65]
			\begin{axis}[ 
			xlabel={$f_1$}, ylabel={$V_1$},
			title={$\lambda=1$},
			legend entries={$r=2$,$r=2.3$,$r=2.6$}, legend style={at={(0.02,0.98)}, anchor=north west}
			]
			\addplot+[mark=none] table[x index=2, y index=5]{Data-Power-amplitude_f-r=2_E1.txt};		
			\addplot+[mark=none] table[x index=2, y index=5]{Data-Power-amplitude_f-r=2.3_E1.txt};
			\addplot+[mark=none] table[x index=2, y index=5]{Data-Power-amplitude_f-r=2.6_E1.txt};					
			\end{axis}
			\end{tikzpicture}
			&
			\begin{tikzpicture}[scale=0.65]
			\begin{axis}[ 
			xlabel={$f_1$}, 
			title={$\lambda=10$},
			]
			\addplot+[mark=none] table[x index=2, y index=7]{Data-Power-amplitude_f-r=2_E1.txt};		
			\addplot+[mark=none] table[x index=2, y index=7]{Data-Power-amplitude_f-r=2.3_E1.txt};
			\addplot+[mark=none] table[x index=2, y index=7]{Data-Power-amplitude_f-r=2.6_E1.txt};					
			\end{axis}
			\end{tikzpicture}
			&
			\begin{tikzpicture}[scale=0.65]
			\begin{axis}[ 
			xlabel={$f_1$}, 
			title={$\lambda=100$},
			]
			\addplot+[mark=none] table[x index=2, y index=9]{Data-Power-amplitude_f-r=2_E1.txt};		
			\addplot+[mark=none] table[x index=2, y index=9]{Data-Power-amplitude_f-r=2.3_E1.txt};
			\addplot+[mark=none] table[x index=2, y index=9]{Data-Power-amplitude_f-r=2.6_E1.txt};					
			\end{axis}
			\end{tikzpicture}
			\\
						\begin{tikzpicture}[scale=0.65]
			\begin{axis}[ 
			xlabel={$f_1$}, ylabel={$V_1$},
			legend entries={$r=2$,$r=2.3$,$r=2.6$}, legend style={at={(0.02,0.98)}, anchor=north west}
			]
			\addplot+[mark=none] table[x index=2, y index=5]{Data-Power-amplitude_f-r=2_E2.txt};		
			\addplot+[mark=none] table[x index=2, y index=5]{Data-Power-amplitude_f-r=2.3_E2.txt};
			\addplot+[mark=none] table[x index=2, y index=5]{Data-Power-amplitude_f-r=2.6_E2.txt};					
			\end{axis}
			\end{tikzpicture}
			&
			\begin{tikzpicture}[scale=0.65]
			\begin{axis}[ 
			xlabel={$f_1$}, 
			]
			\addplot+[mark=none] table[x index=2, y index=7]{Data-Power-amplitude_f-r=2_E2.txt};		
			\addplot+[mark=none] table[x index=2, y index=7]{Data-Power-amplitude_f-r=2.3_E2.txt};
			\addplot+[mark=none] table[x index=2, y index=7]{Data-Power-amplitude_f-r=2.6_E2.txt};					
			\end{axis}
			\end{tikzpicture}
			&
			\begin{tikzpicture}[scale=0.65]
			\begin{axis}[ 
			xlabel={$f_1$}, 
			]
			\addplot+[mark=none] table[x index=2, y index=9]{Data-Power-amplitude_f-r=2_E2.txt};		
			\addplot+[mark=none] table[x index=2, y index=9]{Data-Power-amplitude_f-r=2.3_E2.txt};
			\addplot+[mark=none] table[x index=2, y index=9]{Data-Power-amplitude_f-r=2.6_E2.txt};					
			\end{axis}
			\end{tikzpicture}
			\\
					\begin{tikzpicture}[scale=0.65]
		\begin{axis}[ 
		xlabel={$f_1$}, ylabel={$V_1$},
		legend entries={$r=2$,$r=2.3$,$r=2.6$}, legend style={at={(0.02,0.98)}, anchor=north west}
		]
		\addplot+[mark=none] table[x index=2, y index=5]{Data-Power-amplitude_f-r=2_E3.txt};		
		\addplot+[mark=none] table[x index=2, y index=5]{Data-Power-amplitude_f-r=2.3_E3.txt};
		\addplot+[mark=none] table[x index=2, y index=5]{Data-Power-amplitude_f-r=2.6_E3.txt};					
		\end{axis}
		\end{tikzpicture}
		&
		\begin{tikzpicture}[scale=0.65]
		\begin{axis}[ 
		xlabel={$f_1$}, 
		]
		\addplot+[mark=none] table[x index=2, y index=7]{Data-Power-amplitude_f-r=2_E3.txt};		
		\addplot+[mark=none] table[x index=2, y index=7]{Data-Power-amplitude_f-r=2.3_E3.txt};
		\addplot+[mark=none] table[x index=2, y index=7]{Data-Power-amplitude_f-r=2.6_E3.txt};					
		\end{axis}
		\end{tikzpicture}
		&
		\begin{tikzpicture}[scale=0.65]
		\begin{axis}[ 
		xlabel={$f_1$}, 
		]
		\addplot+[mark=none] table[x index=2, y index=9]{Data-Power-amplitude_f-r=2_E3.txt};		
		\addplot+[mark=none] table[x index=2, y index=9]{Data-Power-amplitude_f-r=2.3_E3.txt};
		\addplot+[mark=none] table[x index=2, y index=9]{Data-Power-amplitude_f-r=2.6_E3.txt};					
		\end{axis}
		\end{tikzpicture}
		\\
		\begin{tikzpicture}[scale=0.65]
		\begin{axis}[ 
		xlabel={$f_1$}, ylabel={$V_1$},
		title={$\lambda=1$},
		legend entries={$r=2$,$r=2.3$,$r=2.6$}, legend style={at={(0.02,0.98)}, anchor=north west}
		]
		\addplot+[mark=none] table[x index=2, y index=5]{Data-Power-amplitude_f-r=2_E4.txt};		
		\addplot+[mark=none] table[x index=2, y index=5]{Data-Power-amplitude_f-r=2.3_E4.txt};
		\addplot+[mark=none] table[x index=2, y index=5]{Data-Power-amplitude_f-r=2.6_E4.txt};					
		\end{axis}
		\end{tikzpicture}
		&
		\begin{tikzpicture}[scale=0.65]
		\begin{axis}[ 
		xlabel={$f_1$}, 
		title={$\lambda=10$},
		legend entries={$r=2$,$r=2.3$,$r=2.6$}, legend style={at={(0.02,0.98)}, anchor=north west}
		]
		\addplot+[mark=none] table[x index=2, y index=7]{Data-Power-amplitude_f-r=2_E4.txt};		
		\addplot+[mark=none] table[x index=2, y index=7]{Data-Power-amplitude_f-r=2.3_E4.txt};
		\addplot+[mark=none] table[x index=2, y index=7]{Data-Power-amplitude_f-r=2.6_E4.txt};					
		\end{axis}
		\end{tikzpicture}
		&
		\begin{tikzpicture}[scale=0.65]
		\begin{axis}[ 
		xlabel={$f_1$}, 
		title={$\lambda=100$},
		legend entries={$r=2$,$r=2.3$,$r=2.6$}, legend style={at={(0.02,0.98)}, anchor=north west}
		]
		\addplot+[mark=none] table[x index=2, y index=9]{Data-Power-amplitude_f-r=2_E4.txt};		
		\addplot+[mark=none] table[x index=2, y index=9]{Data-Power-amplitude_f-r=2.3_E4.txt};
		\addplot+[mark=none] table[x index=2, y index=9]{Data-Power-amplitude_f-r=2.6_E4.txt};					
		\end{axis}
		\end{tikzpicture}	
		\end{tabular}
	\end{center}
\caption{Case $\gamma>1$, $r>1$ (power law). Component $\tilde V_1$ of the mean filtration velocity $\tilde V'$ plotted against $f_1$, with $f'=(f_1,0)$, for $r\in \{2,2.3,2.6\}$ and elliptic inclusions $E_1$ (first line), $E_2$ (second line), $E_3$ (third line) and $E_4$ (fourth line). From left to right: $\lambda=1, \lambda=10, \lambda=100$.}\label{Fig:amplitude_power}
\end{figure}

\subsubsection{Influence of the orientation of $f'$}

Finally, we have represented in Fig.\!~\ref{Fig:rotation_power} the vector $\tilde V'$ computed for different orientations of the imposed pressure gradient $f'$, similarly as in Fig.\!~\ref{Fig:rotation_Carreau}. We can observe for $\lambda=100$ and $r\in \{2.3,2.6\}$ a very similar behaviour of the effective systems~\eqref{thm:system_gamma1} and \eqref{thm:system_22}, which seems to confirm the above interpretation. The differences between both systems appear for $\lambda=1$ and concern only the amplitude of $\tilde V'$, which increases as $r$ increases in the case of the power law, while it was not affected by variations of $r$ in the case of Carreau law and for this particular value of $\lambda$.



\begin{figure}
	\begin{center}
		\begin{tabular}{ll}
			\begin{tikzpicture}[scale=0.65]
			\begin{axis}[title={$\lambda=1,\quad r=2$}, xlabel=$f_1$,ylabel=$f_2$,colorbar]
			\addplot[point meta = {\thisrow{Vnorm})},quiver={u=\thisrow{Vm1norm},v=\thisrow{Vm2norm},scale arrows = 0.2},-stealth,line width=1.5pt,quiver/colored = {mapped color},] table {Data-power_law-rotation_f-lambda=1-r=2.txt};	
			\end{axis}	
			\end{tikzpicture}&
			\begin{tikzpicture}[scale=0.65]
			\begin{axis}[title={$\lambda=100,\quad r=2$}, xlabel=$f_1$,ylabel=$f_2$,colorbar]
			\addplot[point meta = {\thisrow{Vnorm})},quiver={u=\thisrow{Vm1norm},v=\thisrow{Vm2norm},scale arrows = 0.2},-stealth,line width=1.5pt,quiver/colored = {mapped color},] table {Data-power_law-rotation_f-lambda=100-r=2.txt};	
			\end{axis}	
			\end{tikzpicture}
			\\
			\begin{tikzpicture}[scale=0.65]
			\begin{axis}[title={$\lambda=1, \quad r=2.3$}, xlabel=$f_1$,ylabel=$f_2$,colorbar]
			\addplot[point meta = {\thisrow{Vnorm})},quiver={u=\thisrow{Vm1norm},v=\thisrow{Vm2norm},scale arrows = 0.2},-stealth,line width=1.5pt,quiver/colored = {mapped color},] table {Data-power_law-rotation_f-lambda=1-r=2.3.txt};	
			\end{axis}	
			\end{tikzpicture}
			&
			\begin{tikzpicture}[scale=0.65]
			\begin{axis}[title={$\lambda=100, \quad r=2.3$}, xlabel=$f_1$,ylabel=$f_2$,colorbar]
			\addplot[point meta = {\thisrow{Vnorm})},quiver={u=\thisrow{Vm1norm},v=\thisrow{Vm2norm},scale arrows = 0.2},-stealth,line width=1.5pt,quiver/colored = {mapped color},] table {Data-power_law-rotation_f-lambda=100-r=2.3.txt};	
			\end{axis}	
			\end{tikzpicture}
			\\
			\begin{tikzpicture}[scale=0.65]
			\begin{axis}[title={$\lambda=1, \quad r=2.6$}, xlabel=$f_1$,ylabel=$f_2$,colorbar]
			\addplot[point meta = {\thisrow{Vnorm})},quiver={u=\thisrow{Vm1norm},v=\thisrow{Vm2norm},scale arrows = 0.2},-stealth,line width=1.5pt,quiver/colored = {mapped color},] table {Data-power_law-rotation_f-lambda=1-r=2.6.txt};	
			\end{axis}	
			\end{tikzpicture}&
			\begin{tikzpicture}[scale=0.65]
			\begin{axis}[title={$\lambda=100, \quad r=2.6$}, xlabel=$f_1$,ylabel=$f_2$,colorbar]
			\addplot[point meta = {\thisrow{Vnorm})},quiver={u=\thisrow{Vm1norm},v=\thisrow{Vm2norm},scale arrows = 0.2},-stealth,line width=1.5pt,quiver/colored = {mapped color},] table {Data-power_law-rotation_f-lambda=100-r=2.6.txt};	
			\end{axis}	
			\end{tikzpicture}
		\end{tabular}
	\end{center}
	\caption{Case $\gamma>1$, $r>2$ (power law). Representation of $\tilde V'$ when $f'=(f_1,f_2)$ takes the form $(f_1,f_2)=(\cos \theta,\sin \theta)$ with $\theta\in [0,\pi/2]$, in the case of the obstacle shape $E_2$. As in Fig.\!~\ref{Fig:rotation_Carreau}, each vector $\tilde V'$ is represented by a vector of length $0.2$, localized at point $(f_1,f_2)$ and colored according to $|\tilde V'|$. The left column corresponds to $\lambda=1$, the right one to $\lambda=100$, and each line from top to bottom corresponds respectively to $r=2$, $r=2.3$ and $r=2.6$.}\label{Fig:rotation_power}
\end{figure}
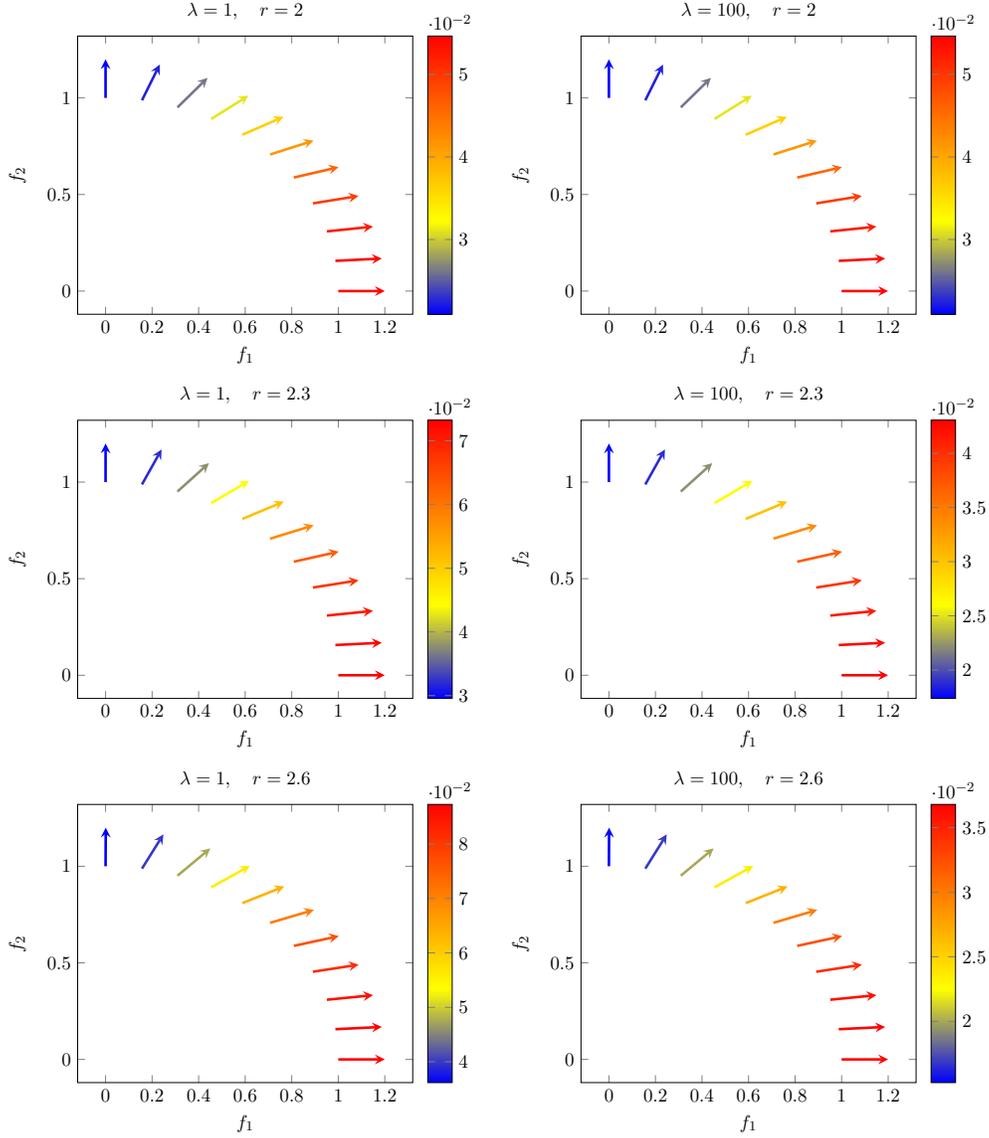

\newpage
{\bf Acknowledgments:}
The main idea of this paper was developed while  Mar\'ia Anguiano and Francisco J. Su\'arez-Grau were visiting the Laboratoire Jacques-Louis Lions (in July 2022), at Sorbonne Université, Université Paris Cité and Centre National de la Recherche Scientifique in Paris (France), invited by Dr.\! Matthieu Bonnivard in order to develop a joint contribution on the thematics of the project ``An\'alisis, homogeneizaci\'on y simulaci\'on num\'erica de EDPs con origen en mec\'anica de fluidos",  financially supported by the Laboratoire Jacques-Louis Lions. Mar\'ia Anguiano and Francisco J. Su\'arez-Grau would like to thank people from this Laboratoire for their very kind hospitality and especially Dr.\! Matthieu Bonnivard. Mar\'ia Anguiano wants to dedicate this paper to her father, Julio, for all his love and for the good times they had together in Paris in July 2022.

\end{document}